\documentclass[reqno,oneside,12pt]{amsart}
\usepackage[T1]{fontenc}
\usepackage[utf8]{inputenc}
\usepackage[top=1.5cm,bottom=2cm,right=2cm,left=2cm]{geometry}
\usepackage{amsmath,amsfonts,amssymb,amscd,amsthm}
\usepackage[usenames,dvipsnames,svgnames]{xcolor}
\usepackage[numbers]{natbib}
\usepackage{doi}
\usepackage{graphicx}
\usepackage{caption} 
\usepackage{subcaption}
\usepackage{autonum} 
\usepackage{hyperref} 

\usepackage{newtxtext,newtxmath}

\usepackage{acronym}

\usepackage{booktabs}
 
\usepackage{tikz}
\definecolor{myellow}{RGB}{255,230,128}
\definecolor{gray20}{RGB}{204,204,204}
\definecolor{mygray}{RGB}{204,204,204}
\definecolor{mygreen}{RGB}{138,203,95}
\definecolor{myblue}{RGB}{77,151,214}

\usepackage[textsize=small]{todonotes}

\acrodef{fe}[FE]{finite element}
\acrodef{dof}[DOF]{degree of freedom}
\acrodefplural{dof}[DOFs]{degrees of freedom}
\acrodef{vef}[VEF]{vertex, edge, and face}
\acrodefplural{vef}[VEFs]{vertices, edges, and faces}
\acrodef{dg}[DG]{discontinuous Galerkin}
\acrodef{vms}[VMS]{variational multiscale}
\acrodef{sps}[SPS]{symmetric projection stabilization}
\acrodef{agfem}[agFEM]{aggregated finite element method}
\acrodef{xfem}[XFEM]{extended finite element method}
\acrodef{agfe}[agFE]{aggregated finite element}
\acrodef{pde}[PDE]{partial differential equation}

\newtheorem{theorem}{Theorem}[section]
\newtheorem{lemma}[theorem]{Lemma}

\newtheorem{definition}[theorem]{Definition}

\newtheorem{assumption}[theorem]{Assumption}

\newtheorem{method}[theorem]{Algorithm}

\makeatletter
\newcommand{\opnorm}{\@ifstar\@opnorms\@opnorm}
\newcommand{\@opnorms}[1]{%
  \left|\mkern-1.5mu\left|\mkern-1.5mu\left|
   #1
  \right|\mkern-1.5mu\right|\mkern-1.5mu\right|
}
\newcommand{\@opnorm}[2][]{%
  \mathopen{#1|\mkern-1.5mu#1|\mkern-1.5mu#1|}
  #2
  \mathclose{#1|\mkern-1.5mu#1|\mkern-1.5mu#1|}
}
\makeatother

\def\opnormh#1{\opnorm{#1}_h}

\def\x{\boldsymbol{x}}
\def\Dom{\Omega}
\def\Domap{{\Dom}}
\def\Domex{{\Omega_{\rm act}}}
\def\Domin{{\Omega_{\rm in}}}
\def\Domart{{\Omega_{\rm art}}}
\def\Domext{{\Omega_{\rm ext}}}
\def\Domcut{{\Omega_{\rm cut}}}

\def\agints{\mathcal{F}_h}

\def\mesh{\mathcal{K}_h}
\def\meshex{\mathcal{K}^{\rm act}_h}
\def\meshin{\mathcal{K}^{\rm in}_h}

\def\meshact{\mathcal{K}_h}
\def\meshart{\mathcal{K}^{\rm art}_h}
\def\meshext{\mathcal{K}^{\rm ext}_h}
\def\meshcut{\mathcal{K}^{\rm cut}_h}
\def\meshag{\mathcal{T}_h}

\def\cell{K}

\def\uh{u_h}
\def\vh{v_h}
\def\shpf#1{\phi^{#1}}

\def\lset{\psi}

\def\hc{h_\cell}

\def\h{h}

\def\fespst{{\mathcal{V}_h}}
\def\fespin{{\mathcal{V}_h^{\rm in}}}
\def\fespex{{\mathcal{V}_h^{\rm act}}}

\def\fespag{{\mathcal{V}_h}}
\def\fespagm{{\breve{\mathcal{V}}_h}}

\def\nodes#1{\mathcal{N}(#1)}
\def\nodesm#1{\breve{\mathcal{N}}(#1)}

\def\nodesgin{\mathcal{N}_h^{\rm in}}
\def\nodesgex{\mathcal{N}_h^{\rm act}}
\def\nodesgou{\mathcal{N}_h^{\rm out}}

\def\grad{{\boldsymbol{\nabla}}}

\def\interp#1{{\mathcal{I}(#1)}}

\def\ext#1{\mathcal{E}_h(#1)}
\def\extm#1{\breve{\mathcal{E}}_h(#1)}

\def\owner#1{\mathcal{O}(#1)}

\def\jumpl{\lbrack\!\lbrack}
\def\jumpr{\rbrack\!\rbrack}
\def\jump#1{\jumpl #1 \jumpr}

\def\bsphih{\bs{\phi}_h}
\def\bszetah{\bs{\zeta}_h}
\def\vph#1{\bs{\varphi}_h(#1)}
\def\FEMPAR{{\texttt{FEMPAR}}}

\def\E{\mathcal{E}}

\def\x{\boldsymbol{x}}

\def\n{{\boldsymbol{n}}}

\def\lfesp{\mathcal{V}}
\def\moments{\Sigma}

\def\Pq{\mathcal{P}_q}
\def\Qq{\mathcal{Q}_q}
\def\Qqm{\breve{\mathcal{Q}}_q}
\def\Qbq{\bs{\mathcal{Q}}_q}

\def\Qbord#1{\bs{\mathcal{Q}}_{#1}}

\def\Pqh{\mathcal{P}_{q,h}}
\def\Pqhd{\mathcal{P}^-_{q,h}}
\def\Qqh{\mathcal{Q}_{q,h}}
\def\Qqhm{\breve{\mathcal{Q}}_{q,h}}

\def\Qqth{\mathcal{Q}_{2,h}}
\def\Qqthm{\breve{\mathcal{Q}}_{2,h}}
\def\Pqthd{{\mathcal{P}}^-_{1,h}}

\def\Qbqh{\bs{\mathcal{Q}}_{q,h}}
\def\Qbqhm{\breve{\bs{\mathcal{Q}}}_{q,h}}

\def\Pordhd#1{\mathcal{P}_{{#1},h}^-}

\def\Pbordhd#1{\bs{\mathcal{P}}_{{#1},h}^-}

\def\bsphi{\bs{\Phi}_h}
\def\bspsi{\bs{\Psi}_h}

\newcommand{\gradient}{\boldsymbol{\nabla}}

\tikzstyle{fig-ph}=[draw,minimum width=\textwidth, minimum height=\textwidth,text width=0.9\textwidth,color=red]

\tikzstyle{fig-ph-r}=[draw,minimum width=\textwidth, minimum height=1.4\textwidth,text width=0.9\textwidth,color=red]
\tikzstyle{fig-ph-rl}=[draw,minimum width=\textwidth, minimum height=0.5\textwidth,text width=0.9\textwidth,color=red]

\graphicspath{{figures/}}

\begin{document}

\title[Mixed aggregated FE methods for the  unfitted discretization of the Stokes problem]{Mixed aggregated finite element methods for the  unfitted discretization of the Stokes problem}

\author[S. Badia]{Santiago Badia}

\author[A. F. Mart\'in]{Alberto F. Mart\'in}

\author[F. Verdugo]{Francesc Verdugo}

\address{Department of Civil and Environmental Engineering. Universitat Polit\`ecnica de Catalunya, Jordi Girona 1-3, Edifici C1, 08034 Barcelona, Spain.}
\address{CIMNE – Centre Internacional de M\`etodes Num\`erics en 
Enginyeria, Parc Mediterrani de la Tecnologia, UPC, Esteve Terradas 5, 08860 
Castelldefels, Spain.}

\thanks{SB gratefully acknowledges the support received from the Catalan Government through the ICREA Acad\`emia Research Program. {FV gratefully acknowledges the support received from the  \emph{Secretaria d'Universitats i Recerca} of the Catalan Government in the framework of the Beatriu Pinós Program (Grant Id.: 2016 BP 00145). The authors thankfully acknowledge the computer resources at Marenostrum-IV and the technical support provided by the Barcelona Supercomputing Center (RES-ActivityID: FI-2018-1-0014). \\E-mails: {\tt sbadia@cimne.upc.edu} (SB), {\tt amartin@cimne.upc.edu} (AM)}, {\tt fverdugo@cimne.upc.edu} (FV)}

\date{\today}

\begin{abstract}
  In this work, we consider unfitted finite element methods for the numerical approximation of the Stokes problem. It is well-known that this kind of methods lead to arbitrarily ill-conditioned systems. In order to solve this issue, we consider the recently proposed aggregated finite element method, originally motivated for coercive problems. However, the well-posedness of the Stokes problem is far more subtle and relies on a discrete inf-sup condition. We consider mixed finite element methods that satisfy the discrete version of the inf-sup condition for body-fitted meshes, and analyze how the discrete inf-sup is affected when considering the unfitted case. We propose different aggregated mixed finite element spaces combined with simple stabilization terms, which can include pressure jumps and/or cell residuals, to fix the potential deficiencies of the aggregated inf-sup. We carry out a complete numerical analysis, which includes stability, optimal \emph{a priori} error estimates, and condition number bounds that are not affected by the small cut cell problem. For the sake of conciseness, we have restricted the analysis to hexahedral meshes and discontinuous pressure spaces. A thorough numerical experimentation bears out the numerical analysis. The aggregated mixed finite element method is ultimately applied to two problems with non-trivial geometries.
\end{abstract}

\maketitle




\def\ah#1#2{a_h(#1,#2)}
\def\bh#1#2{b_h(#1,#2)}
\def\ch#1#2#3#4{j_h(#1,#2,#3,#4)}
\def\Ah#1#2#3#4{A_h(#1,#2,#3,#4)}
\def\Lh#1#2{L_h(#1,#2)}
\def\bal{\begin{align}}
\def\eal{\end{align}}
\def\norm#1{\|#1\|}
\def\normh#1{\|#1\|_{*}}
\def\bcell{\partial \cell}
\def\normreg#1#2{\norm{#1}_{#2}}
\def\brackets#1{\left(#1\right)}
\def\interp#1{i_h(#1)}
\def\szint#1{\pi_h(#1)}
\def\bs#1{\boldsymbol{#1}}
\def\uh{\bs{u}_h}
\def\u{\bs{u}}
\def\ush{u_h}
\def\vsh{v_h}

\def\H{\bs{H}}
\def\L{\bs{L}}
\def\v{\bs{v}}
\def\u{\bs{u}}
\def\bbh{\bs{b}_\cell}
\def\bsh{{b}_\cell}
\def\vh{\bs{v}_h}
\def\wh{\bs{w}_h}
\def\ph{p_h}
\def\qh{q_h}
\def\ip#1#2{\brackets{#1,#2}_\Dom}
\def\ipreg#1#2#3{\brackets{#1,#2}_{#3}}
\def\n{{\bs{n}}}
\def\f{{\bs{f}}}
\def\g{{\bs{g}}}
\def\dn{{\partial_\n}}
\def\bou{{\Gamma}}
\def\bouD{{\bou_{\rm D}}}
\def\bouN{{\bou_{\rm N}}}
\def\pen{{\tau}}
\def\h{{h}}
\def\div{\grad \cdot}
\def\hcell{\h_\cell}
\def\kproj#1{\kappa_h^\perp{\brackets{#1}}}

\def\Vh{\boldsymbol{V}_{h}}
\def\Qh{Q_h}

\def\normregl2#1#2{{\| #1 \|}_{{#2}}}
\def\normregh1#1#2{{\| #1 \|}_{1,#2}}
\def\seminormregh1#1#2{{| #1 |}_{1,#2}}
\def\tnorm#1{\opnormh{#1}}
\def\half{\frac{1}{2}}

\def\pisz{\pi_h^{SZ}}

\section{Introduction}\label{sec:int}

Unfitted \ac{fe} techniques are receiving increasing attention since they are very appealing in many practical situations. Such techniques avoid the generation of \emph{body-fitted} meshes, which is a serious bottleneck in large scale simulations. They are particularly well-suited to multi-phase and multi-physics applications with moving interfaces (e.g., fracture mechanics, fluid-structure interaction \cite{badia_fluidstructure_2008}, or free surface flows), and in applications with varying domains (e.g., shape or topology optimization frameworks, additive manufacturing and 3D printing simulations \cite{chiumenti_numerical_2017}, stochastic geometry problems). Unfitted \ac{fe} methods have been named in different ways. When designed for capturing interfaces, they are usually denoted as \ac{xfem} \cite{belytschko_arbitrary_2001}, whereas they are denoted as { \emph{embedded}, \emph{immersed}, or \emph{unfitted} methods} when the motivation is to simulate a problem using a (usually simple) background mesh { (see, e.g., the {cutFEM} method \cite{burman_cutfem:_2015}).

Yet useful, unfitted \ac{fe} methods have known drawbacks. They pose problems to numerical integration, imposition of Dirichlet boundary conditions, and lead to ill conditioned problems { \citep{de_prenter_condition_2017}}. For most of the unfitted \ac{fe} techniques, the condition number of the discrete linear system does not only depend on the characteristic element size of the background mesh, but also on the ratios for all cut cells of the total cell volume and the cell volume inside the physical domain, which can be arbitrarily small, leading to the so-called \emph{small cut cell problem}. Methods based on fictitious material \cite{schillinger_finite_2014} require a penalty term that goes to zero with a power of the mesh size for optimal convergence and thus, are also affected by these problems. Preconditioned iterative linear solvers suitable for standard \ac{fe} methods are not robust for these formulations. Recently, a robust domain decomposition preconditioner able to deal with cut cells has been proposed in \cite{badia_robust_2017} for first order methods, but these preconditioners still require some special treatment for the robust direct solution of local-to-subdomain systems. 

The authors have recently proposed in \cite{Badia2017} an unfitted \ac{fe} formulation, referred to as the \ac{agfem}, that fixes the ill conditioning issues associated with cut cells for elliptic \acp{pde}. This novel method relies on the so-called \ac{agfe} spaces, grounded on cell aggregation techniques and judiciously chosen linear constraints for conflictive \acp{dof} with respect to interior ones. This approach can be applied to grad-conforming (globally continuous) spaces and discontinuous \ac{fe} spaces of arbitrary order. The \ac{agfem} leads to a well-posed Galerkin formulation of elliptic problems, viz., no stabilization terms are needed and the method is thus consistent. Furthermore, the resulting linear system have condition numbers that scale only with the element size of the background mesh in the same way as in standard \ac{fe} methods for body-fitted meshes. These methods have been implemented in \FEMPAR{}, a large scale \ac{fe} software package \cite{badia_fempar:_2017,_fempar_????}.

  Compared to other existing approaches, the most salient one is the \emph{ghost penalty} formulation used in the CutFEM method \cite{burman_cutfem:_2015,burman_fictitious_2012}. In any case, this approach  leads to weakly non-consistent algorithms and requires to compute high order derivatives on faces for high order \acp{fe}, which are not at our disposal in general \ac{fe} codes and are expensive to compute, certainly complicating the implementation of the methods and harming code performance. For B-spline approximations, one can consider the so-called \emph{extension} or \emph{extrapolation} techniques (see, e.g., \citep{hollig_weighted_2001,ruberg_subdivision-stabilised_2012,ruberg_fixed-grid_2014}). These works are close to the \ac{agfem} \cite{Badia2017} in the sense that the  problematic \acp{dof} associated with B-splines with small support inside the physical domain are eliminated by constraining them as a linear combination of \emph{well-posed} \acp{dof}. Such aggregation approaches are not new in \ac{dg} methods (see, e.g., \cite{helzel_high-resolution_2005,johansson_high_2013,kummer_extended_2013}), for which the situation is much easier, since no conformity must be kept. In fact, some aggregation techniques in \ac{dg} \cite{helzel_high-resolution_2005,johansson_high_2013,kummer_extended_2013} can be casted as discontinuous \acp{agfem}.
  
The use of mixed \ac{fe} methods on unfitted meshes has been explored in previous works. The combination of ghost penalty stabilization and inf-sup stable elements for the unfitted \ac{fe} approximation of the Stokes problem was originally addressed in \cite{Burman2014} for triangular meshes in two dimensions. The analysis therein relies on the continuous inf-sup condition on the interior domain, viz., the union of interior cells (not intersecting the boundary), in order to prove pressure stability in interior cells, whereas cut cell pressure stability relies on ghost penalty stabilization. The extension of this work to interface Stokes problems for the MINI element has been proposed in \cite{hansbo_cut_2014} (see also \cite{Cattaneo2015} for a similar strategy). It has been observed in \cite{Guzman2017} that the analysis of unfitted and \ac{xfem} \ac{fe} methods, e.g., in \cite{Burman2014,hansbo_cut_2014,Cattaneo2015}, is not fully satisfactory, because it relies on the inf-sup condition of the interior domain, which has an inf-sup constant that depends on the mesh refinement and can tend to zero. Guzm\'an and Olshanskii follow a different approach in \cite{Guzman2017}, proving stability and error estimates for some families of inf-sup stable elements on triangles and tetrahedra. We refer the reader to \cite{Burman2014,Guzman2017} for more references on this subject in the frame of \ac{xfem}. As an alternative to mixed \ac{fe} methods, globally stabilized residual-based and pressure jump first order schemes combined with ghost penalty stabilization have been used in \cite{massing_stabilized_2014}. Global residual-based stabilization has also been used in \cite{ruberg_subdivision-stabilised_2012,ruberg_fixed-grid_2014} for B-spline approximations.

In this work, we propose to combine the \ac{agfem} approach, which fixes the small cut cell problem for the numerical approximation of elliptic \acp{pde}, with mixed \ac{fe} spaces. Unsurprisingly, the development of mixed \ac{agfe} spaces that satisfy a discrete version of the inf-sup condition is not straightforward. The discrete inf-sup condition requires a perfect balance of the velocity and pressure spaces, whereas the boundary-cell intersections can be arbitrary, leading to a large set of possible cell aggregates geometries. In this work, we consider hexahedral meshes and arbitrary order mixed \ac{fe} spaces with discontinuous pressures, and analyze the potential deficiencies of the \emph{unfitted} inf-sup in terms of a set of \emph{improper} aggregates and interfaces that will require additional stabilization. An abstract stability analysis under some assumptions about such stabilization allows us to define effective stabilization terms. We propose two algorithms. The first one combines a standard aggregated tensor-product Lagrangian \ac{fe} with interior residual-based and pressure jump face stabilization on improper aggregates and faces, respectively. The second one makes use of an \ac{agfe} space in terms of a serendipity-based extension of tensor-product Lagrangian \acp{fe} combined with pressure jump stabilization on improper faces. The resulting schemes can be used in quadrilateral/hexahedral meshes, the order of approximation can be selected by the user, the algorithm does not require to compute (higher than order one) derivatives on cell boundaries (unlike ghost penalty/cutFEM approaches), and it involves minimal stabilization (e.g., only pressure jump stabilization on a very small subset of faces \emph{close} to the interface). A complete numerical analysis shows the uniform stability (that does not rely on the potentially ill inf-sup condition on the union of interior cells), optimal \emph{a priori} error estimates, and condition number bounds with respect to the mesh size and cell boundary intersection. Another remarkable feature of our approach is that it exposes a high degree of message-passing parallelism, and thus it is suitable for the development of a highly scalable parallel unfitted \ac{fe} framework on distributed memory computers, so far still missing in the literature. In fact, a highly scalable parallel implementation of \ac{agfem}, grounded on \texttt{p4est} for efficient octree handling \cite{BursteddeWilcoxGhattas11}, is under development in \FEMPAR{} \cite{badia_fempar:_2017,_fempar_????}. Apart from their ability of controlling geometry approximation errors by local adaptation in regions of high geometric variability,  octree meshes can be very efficiently generated, refined and coarsened, partitioned, and 2:1 balanced on hundreds of thousands of processors \cite{BursteddeWilcoxGhattas11}, being the latter the main reason why we favour this sort of meshes in our approach.

The outline of this work is as follows. In Sect. \ref{sec:pro_sta}, we introduce the Stokes problem and, in Sect. \ref{sec:fe_spa}, a brief introduction to \ac{fe} spaces and some notation follows. Sect. \ref{sec:agg_fe_spa} is devoted to the definition of \ac{agfe} spaces and their mathematical properties. A discrete \ac{agfem} for the approximation of the Stokes problem is proposed in Sect. \ref{sec:app_sto}, in which the stabilization terms are not defined yet. Sect. \ref{sec:num_ana} is devoted to a complete numerical analysis of mixed \acp{agfem}. More specifically, in Sect. \ref{ssec:abs_sta_ana}, we perform an abstract stability analysis under some assumptions over the mixed \ac{agfe} space and the stabilization terms. Two different algorithms that satisfy these assumptions are proposed in Sect. \ref{ssec:sta_mix_fe_pre_sta}. \emph{A priori} error estimates and condition number bounds that are independent of the cut cell intersection with the boundary are proved in  Sect. \ref{ssec:a_pri_err_est} and Sect. \ref{ssec:cond_nums}, respectively. A complete set of numerical experiments can be found in \ref{sec:num_exp}}. To close this work, some conclusions are drawn in Sect. \ref{sec:concl}.

\section{Problem statement}
\label{sec:pro_sta}

Let us consider an open and bounded physical domain $\Dom \subset \mathbb{R}^d$ (where $d = 2, \, 3$ is the physical space dimension) with Lipschitz boundary $\bou$, occupied by a viscous fluid. We consider Dirichlet boundary conditions on $\bou$ for brevity in the exposition; the introduction of Neumann boundary conditions is straightforward. The Stokes problem, after scaling the pressure with the inverse of the diffusion coefficient, reads as: find the velocity field $\u: \Dom \rightarrow \mathbb{R}^d$ and the pressure field $p : \Dom \rightarrow \mathbb{R}$ such that 
\begin{align}
-\Delta \u + \gradient p = \f  \quad \text{in } \ \Omega, \qquad \gradient\cdot \u = 0  \quad \text{in } \ \Omega, \qquad \u=\g  \quad&\text{on } \ \Gamma,
\label{eq:sto_eqs}
\end{align}
where 
$\f$ is the body force and $\g$ is the prescribed Dirichlet data, which must satisfy $\int_\bou \g \cdot \n = 0$, where $\n$ stands for the outward normal. In order to uniquely determine the pressure, we additionally enforce that $\int_\Dom p = 0$.

We use standard notation for Sobolev spaces (see \cite{brezis_functional_2010}). In particular, the $L^2(\omega)$ scalar product will be denoted by $(\cdot,\cdot)_\omega$ for some $\omega \subset \mathbb{R}^d$. Making abuse of notation, we represent the $H^1(\omega)$ duality pairing the same way. $L^2_0(\omega)$ is the subspace of functions in $L^2(\omega)$ with zero mean value.
For a Sobolev space $X$, we denote its norm by $\| \cdot \|_{X}$. In particular, the $L^2(\omega)$ norm is denoted by $\normregl2{\cdot}{\omega}$, whereas the $H^1(\omega)$ norm as $\normregh1{\cdot}{\omega}$. The seminorm on the Sobolev space $W^{k,p}(\omega)$ is denoted by $| \cdot |_{W^{k,p}(\omega)}$, or simply $| \cdot |_{1,\omega}$ for $H^1(\omega)$. Given a function $g \in H^{\frac{1}{2}}(\partial \omega)$, the subspace of functions in $H^1(\omega)$ with trace equal to $g$ is represented with $H^1_g(\omega)$. Vector-valued Sobolev spaces are represented with boldface letters. 

Let us assume that $\f\in \L^{2}(\Dom)$ and $\g \in \H^{\frac{1}{2}}(\bou)$. The weak form of the Stokes problem \eqref{eq:sto_eqs} reads as follows: find $(\u, \, p) \in \H^1_{\g}(\Dom) \times L_0^2(\Dom)$ such that
\begin{align}\label{eq:wea_for}
\ip{\grad \u}{\grad \v} -
\ip{p}{\grad \cdot \v} - \ip{q}{\grad \cdot \u} = \ip{\f}{\v},
\end{align}
for any $(\v, q) \in  \H^1_{\bs{0}}(\Dom) \times L_0^2(\Dom)$. The well-posedness of this linear problem relies on the fact that the divergence operator on $\H^1_{\bs{0}}(\Dom)$ is surjective in $L^2_0(\Dom)$. There exists a constant $\beta$ that depends on $\Omega$ such that
\begin{align}
\inf_{p \in L^2_0(\Omega)} \sup_{\v \in \H^1_{\bs 0}(\Dom)} \frac{\ip{p}{\grad \cdot \v}}{\normregl2{p}{\Dom} \normregh1{\v}{\Dom}} \geq \beta > 0.\label{eq:inf_sup_con}
\end{align}
In the following exposition, we consider the numerical approximation of this problem by using \ac{fe} methods. In particular, we are interested in the discretization of the Stokes problem when using unfitted \ac{fe} methods, i.e., the mesh is not fitted to $\Dom$. 

\section{Finite element spaces}\label{sec:fe_spa}

Let us consider an open polyhedral domain $\omega$ and its partition $\mesh(\omega)$ into a set of cells. We may consider the case in which all cells are hexahedra/quadrilaterals (hex mesh) or all cells are tetrahedra/triangles (tet mesh). At any cell $\cell \in \mesh(\omega)$, we define the local \ac{fe} spaces as follows. Using the abstract definition of Ciarlet, a \ac{fe} is represented by the triplet $\{ \cell, \lfesp, \moments \}$, where $\cell$ is a compact, connected, Lipschitz subset of $\mathbb{R}^d$, $\lfesp$ is a vector space of functions, and $\moments$ is a set of linear functionals that form a basis for the dual space $\lfesp'$. The elements of $\moments$ are the so-called \acp{dof} of the FE; we denote the number of \acp{dof} as $n_\Sigma$. The \acp{dof} can be written as $\sigma^a$ for $a \in \mathcal{N}_\Sigma \doteq \{ 1, \ldots,n_\Sigma\}$. We can also define the basis $\{\shpf{a}\}_{a \in \mathcal{N}_\moments}$ for $\lfesp$ such that $\sigma^a(\shpf{b}) = \delta_{a b}$ for $a,\, b \in \mathcal{N}_\moments$. These functions are the so-called \emph{shape functions} of the FE, and there is a one-to-one mapping between shape functions and \acp{dof}.

In this work, we consider three different concretizations of the vector space $\lfesp$: (1) the space $\Pq(\cell)$ of polynomials of degree less or equal to $q$;
(2) the space $\Qq(\cell)$ of polynomials of degree less or equal to $q$ with  respect to each reference space coordinate; (3) the space $\Qqm(\cell)$ of polynomials of superlinear degree less or equal to $q$ (see \cite{Arnold2011} for more details). $\Qqm(\cell)$ on hex meshes leads to the \emph{serendipity} \ac{fe}. For the sake of simplicity, we assume that all cells in the mesh have the same topology and (for a given field) the same polynomial order.\footnote{The polynomial spaces are defined in the physical space cell, instead of relying on a reference cell and a map from the reference to the physical space. Both approaches are equivalent for affine maps, whereas the second one is more appealing due to lower computational cost. The convergence properties of serendipity \acp{fe} are deteriorated if the map is not affine \cite{Arnold2000}. Fortunately, the equivalence holds for the Cartesian hex meshes below.} 

In order to build globally continuous \ac{fe} spaces, we denote by $\nodes{\cell}$ the set of $n_{\moments}$ Lagrangian nodes of order $q$ of cell $\cell$ for $\Pq(\cell)$ in tets and $\Qq(\cell)$ in hexs. The set of nodal values, i.e., $\sigma^a(v) \doteq v(\x^a)$ for $a \in \nodes{\cell}$, is a basis for the dual space $\lfesp'$. By definition, it holds $\shpf{a}(\x^b) = \delta_{ab}$, where $\x^b$ are the space coordinates of node $b$ in the corresponding set of nodes.  Next, we assume that there is a local-to-global \ac{dof} map such that the resulting global space is $\mathcal{C}^0$ continuous. It leads to the following $\mathcal{C}^0(\omega)$ global \ac{fe} spaces: (1) the space $\Pqh(\omega)$ of functions such that its cell restriction belongs to $\Pq(\cell)$ for a tet mesh; (2) the space $\Qqh(\omega)$ (resp. $\Qqhm(\omega)$) of functions such that its cell restriction belongs to $\Qq(\cell)$ (resp. $\Qqm(\cell)$) for a hex mesh. We note that for discontinuous \ac{fe} spaces, the definition of \ac{dof} is flexible, since no inter-cell continuity must be enforced. We will make use of the global space  $\Pqhd(\omega)$ of piecewise discontinuous functions that belong to $\Pq(\cell)$, for an \emph{arbitrary} cell topology.
 The spaces of vector-valued functions with components in these spaces are represented with boldface letters.

Given a function $v$, we define the \emph{local interpolator} for nodal Lagrangian \acp{fe}, as
\begin{align}\label{eq:loc_int}
  \pi^I_K(v) \doteq \sum_{a \in \nodes{\cell} } \sigma^a (v) \shpf{a} =  \sum_{a \in \nodes{\cell} } v(\x^a) \shpf{a},  \qquad K \in \mesh(\omega).
\end{align}
It is easy to check that the interpolation operator is in fact a projection. The global interpolator $\pi_h^I(\cdot)$ is defined as the sum over the cells of the corresponding local interpolators, i.e., $\pi_h^I(v) = \sum_{\cell \in \mesh(\omega)} \pi_\cell^I(v)$. 

\section{Aggregated finite element spaces}\label{sec:agg_fe_spa}

In this section, we define \ac{agfe} spaces. We refer to \cite{Badia2017} for more details. First, we introduce some geometrical concepts related to the use of embedded boundary methods, the cell aggregation algorithm, and the map between \acp{vef} on cut cells and aggregates. Next, we use the geometrical aggregation to define \ac{agfe} spaces on unfitted meshes. Finally, we provide some trace and inverse inequalities, together with approximability properties that will be used in the following sections to analyze the stability and to obtain \emph{a priori} error estimates. In the following, we assume that hex meshes are being used. In practice, we are interested in Cartesian hex meshes, where all the cells can be represented as the scaling of a $d$-cube. This restriction simplifies implementation issues, since polynomial bases in the physical space can be obtained as the mapped reference cell polynomial bases, a fact that does not hold for general (first order) hex meshes. However, \ac{agfe} spaces can readily be obtained for tet meshes using the ideas below. 


\subsection{Embedded boundary setup and cell aggregation} \label{sec:bac_mes}

As usual for embedded boundary methods, we consider an \emph{artificial} domain $\Domart$ with a simple shape that can easily be meshed using a conforming Cartesian grid $\meshart \doteq \mesh(\Domart)$ of characteristic size $h$ that includes the \emph{physical} domain $\Dom \subset\Domart$  (see Fig. ~\ref{fig:immersed-setup-a}). Let us assume for the sake of simplicity that the domain boundary is implicitly defined as the zero level-set of a given scalar function $\lset^\mathrm{ls}$, i.e., $\bou \doteq\{ \x\in\mathbb{R}^d:\lset^\mathrm{ls}(\x)=0\}$. In practice, we consider an approximation $\Omega_h$ of $\Omega$, e.g., using a marching cubes-like algorithm, which also leads to an approximated boundary $\Gamma_h$. Even though the actual computational domain is $\Dom_h$, we will omit the subscript for the sake of conciseness in the notation, unless the distinction is important.

Cells in $\meshart$ can be classified as follows: a cell
$\cell \in \meshart$ such that $\cell \subset \Omega$ is an
\emph{internal cell}; if $\cell \cap \Omega = \emptyset$, $\cell$ is
an \emph{external cell}; otherwise, $\cell$ is a \emph{cut cell} (see
Fig.~\ref{fig:immersed-setup-b}).  The set of interior (resp.,
external and cut) cells is represented with $\meshin$ and its union
$\Domin \subset \Omega$ (resp., $(\meshext,\Domext)$ and
$(\meshcut, \Domcut))$. Furthermore, we define the set of \emph{active
  cells} as $\meshact \doteq \meshin \cup \meshcut$ and its union
$\Domex$. We assume that the background mesh
is \emph{quasi-uniform} (see, e.g.,
\cite[p.107]{brenner_mathematical_2010}) to reduce technicalities, and
define a characteristic mesh size $\h$. 

\begin{figure}[ht!]
  \centering
  \begin{subfigure}{0.24\textwidth}
    \centering
    \includegraphics[width=0.9\textwidth]{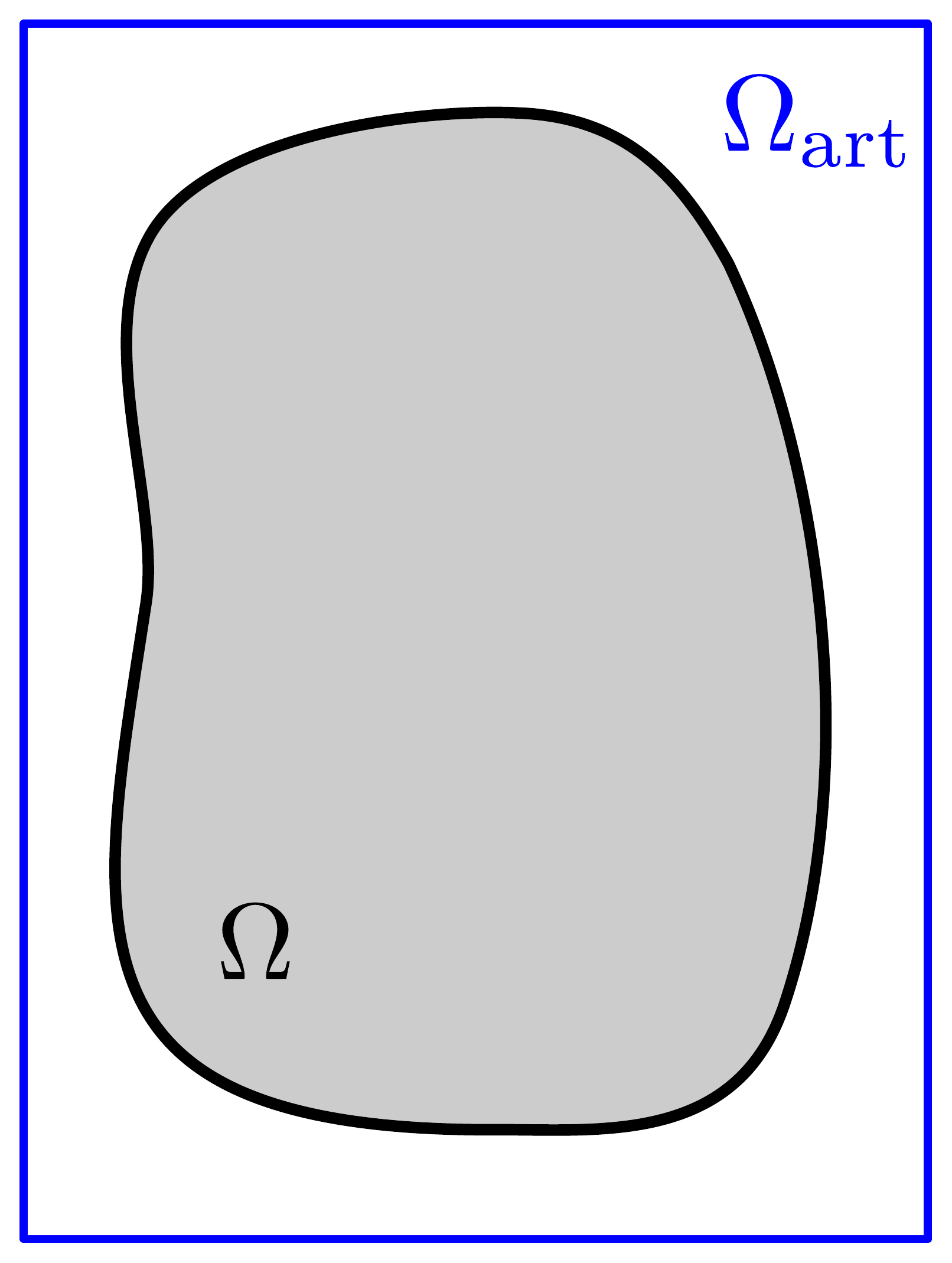}
    \caption{}
    \label{fig:immersed-setup-a}
  \end{subfigure}
  \begin{subfigure}{0.24\textwidth}
    \centering
    \includegraphics[width=0.9\textwidth]{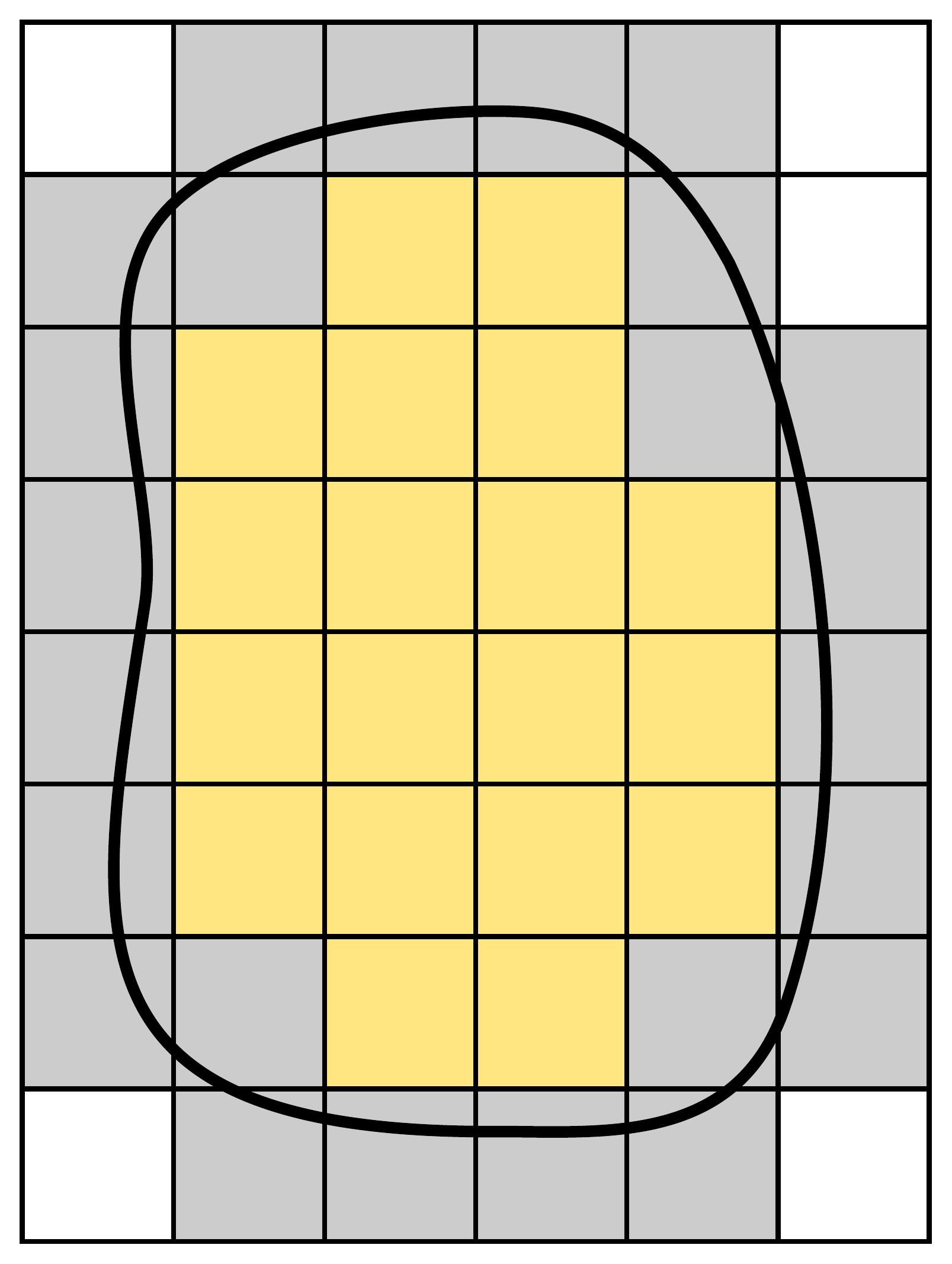}
    \caption{}
    \label{fig:immersed-setup-b}
  \end{subfigure}
    \begin{subfigure}{0.1\textwidth}
    \begin{tabular}{l}
    \tikz{\draw[fill=myellow]  (0,0) rectangle (1.4em,1.4em);} \emph{internal} cells
    \\
    \tikz{\draw[fill=gray20]  (0,0) rectangle (1.4em,1.4em);} \emph{cut} cells
    \\
    \tikz{\draw  (0,0) rectangle (1.4em,1.4em);} \emph{external} cells
    \end{tabular}
  \end{subfigure}
  \caption{Embedded boundary setup.}
  \label{fig:immersed-setup}
\end{figure}
We can also consider a partition  of $\Omega$ into non-overlapping cell aggregates composed of cut cells and \emph{only} one interior cell such that each aggregate is connected, using, e.g., the strategy described in Algorithm \ref{alg:agg_sch} below.

\begin{method}[Cell aggregation algorithm]\

\begin{enumerate}
\item Mark all interior cells as touched and all cut cells as untouched.
\item For each untouched cell,  if there is at least one touched cell connected
to it through a facet  $F$ such that $F \cap \Omega \neq \emptyset$, we aggregate the cell to the touched cell belonging to the aggregate containing the closest interior cell. If more than one touched cell fulfills this requirement, we choose one arbitrarily, e.g., {the cell connected via the facet with more area inside the physical domain, or the one with smaller global label.}
\item Mark as touched all the cells aggregated in step 2. 
\item Repeat steps 2. and 3. until all cells are aggregated.
\end{enumerate}
\label{alg:agg_sch}
\end{method}

Fig. \ref{fig:aggr-steps} shows an illustration of each step in Alg.
\ref{fig:aggr-steps}. The black thin lines represent the boundaries of the
aggregates. Note that from step 1 to step 2, some of the lines between adjacent
cells are removed, meaning that the two adjacent cells have been merged in the
same aggregate. The aggregation schemes can be easily applied to arbitrary
spatial dimensions.

\begin{figure}[ht!]
  \centering
  \begin{subfigure}{0.99\textwidth}
    \centering
    \begin{small}
      \begin{tabular}{llll}
         \tikz{\fill[fill=myellow]  (0,0) rectangle (1.4em,1.4em);} touched 
         &
         \tikz{\fill[fill=gray20]    (0,0) rectangle (1.4em,1.4em);} untouched
         &
         \tikz{ \draw[line width=0.5pt] (0,0) -- (2em,0);} Aggregates' boundary
         &
         \tikz{ \draw[line width=2pt] (0,0) -- (2em,0);} $\partial \Omega$
      \end{tabular}
    \end{small}
  \end{subfigure}
  \par
  \begin{subfigure}{0.24\textwidth}
    \centering
    \includegraphics[width=0.9\textwidth]{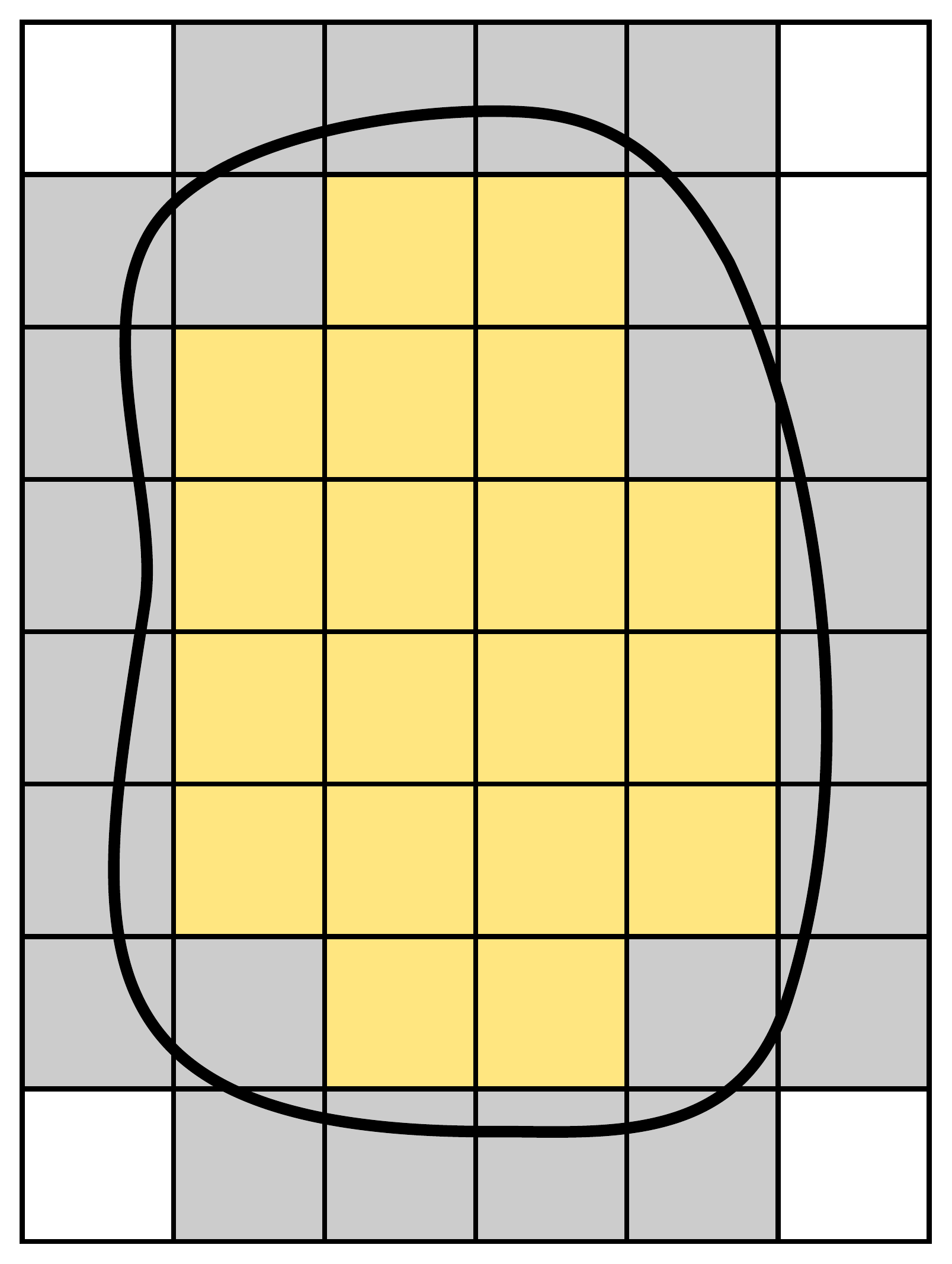}
    \caption{Step 1.}
    \label{fig:aggr-steps-a}
  \end{subfigure}
  \begin{subfigure}{0.24\textwidth}
    \centering
    \includegraphics[width=0.9\textwidth]{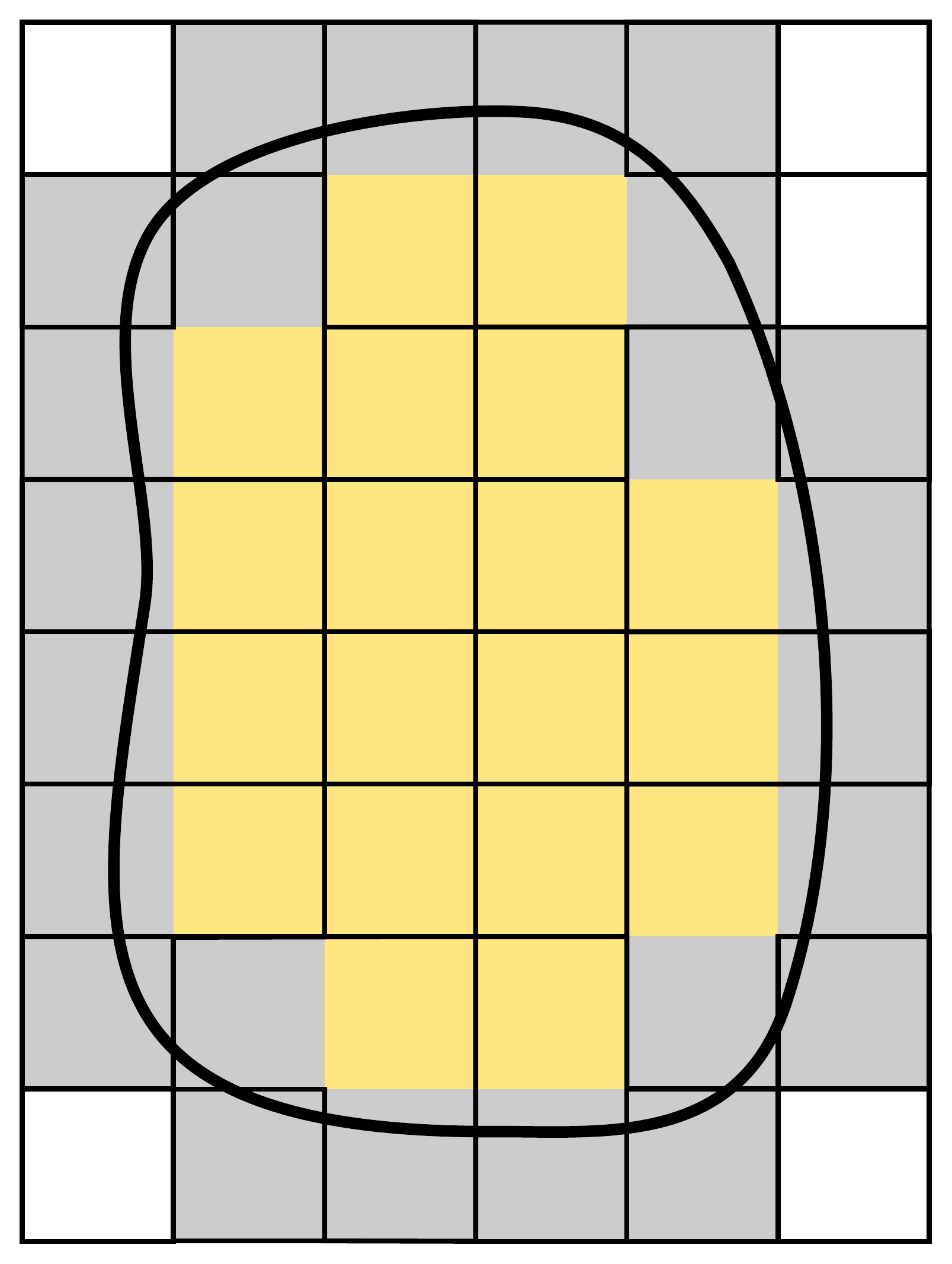}
    \caption{Step 2.}
    \label{fig:aggr-steps-b}
  \end{subfigure}
  \begin{subfigure}{0.24\textwidth}
    \centering
    \includegraphics[width=0.9\textwidth]{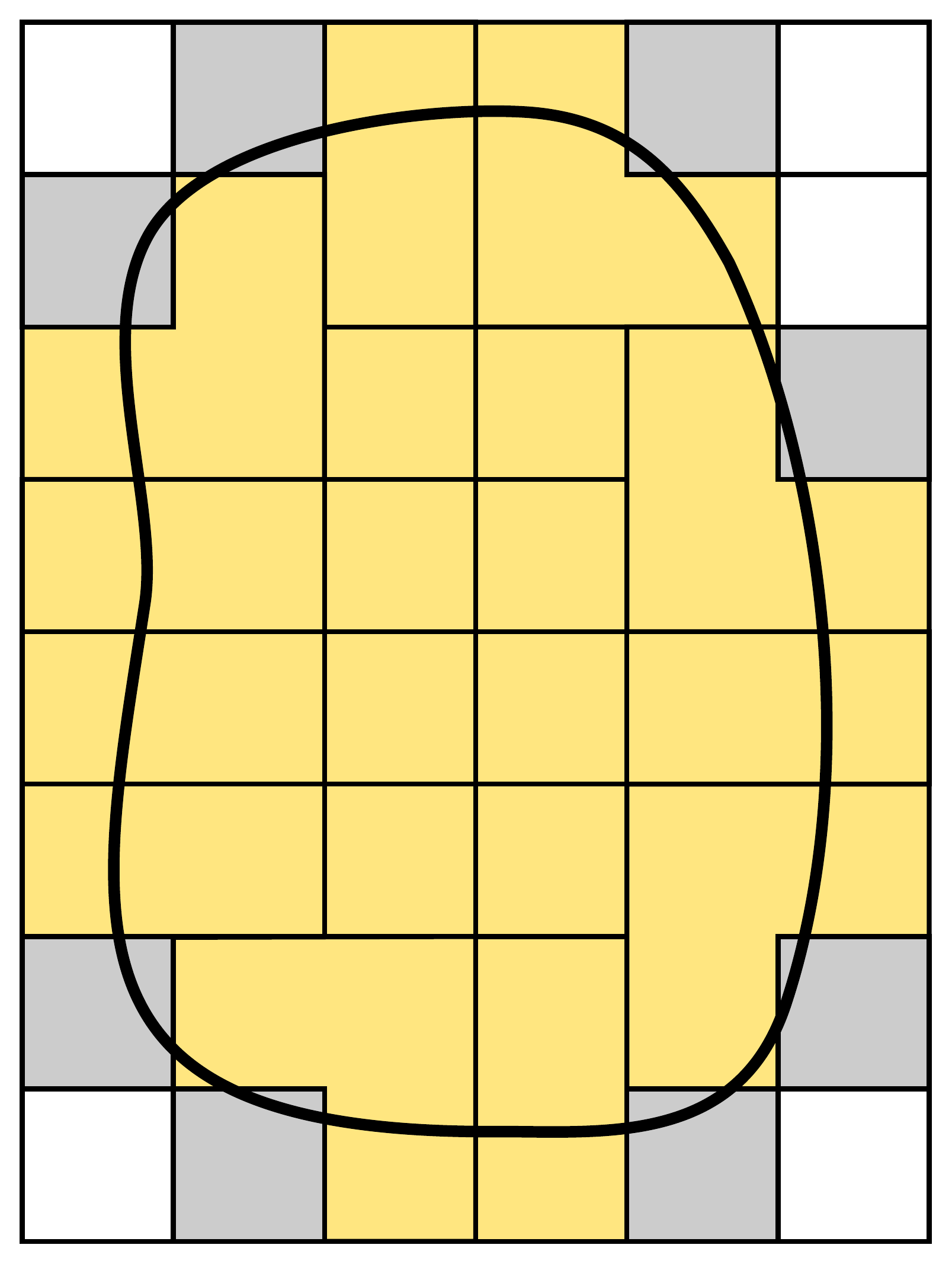}
    \caption{Step 3.}
    \label{fig:aggr-steps-c}
  \end{subfigure}
  \begin{subfigure}{0.24\textwidth}
    \centering
    \includegraphics[width=0.9\textwidth]{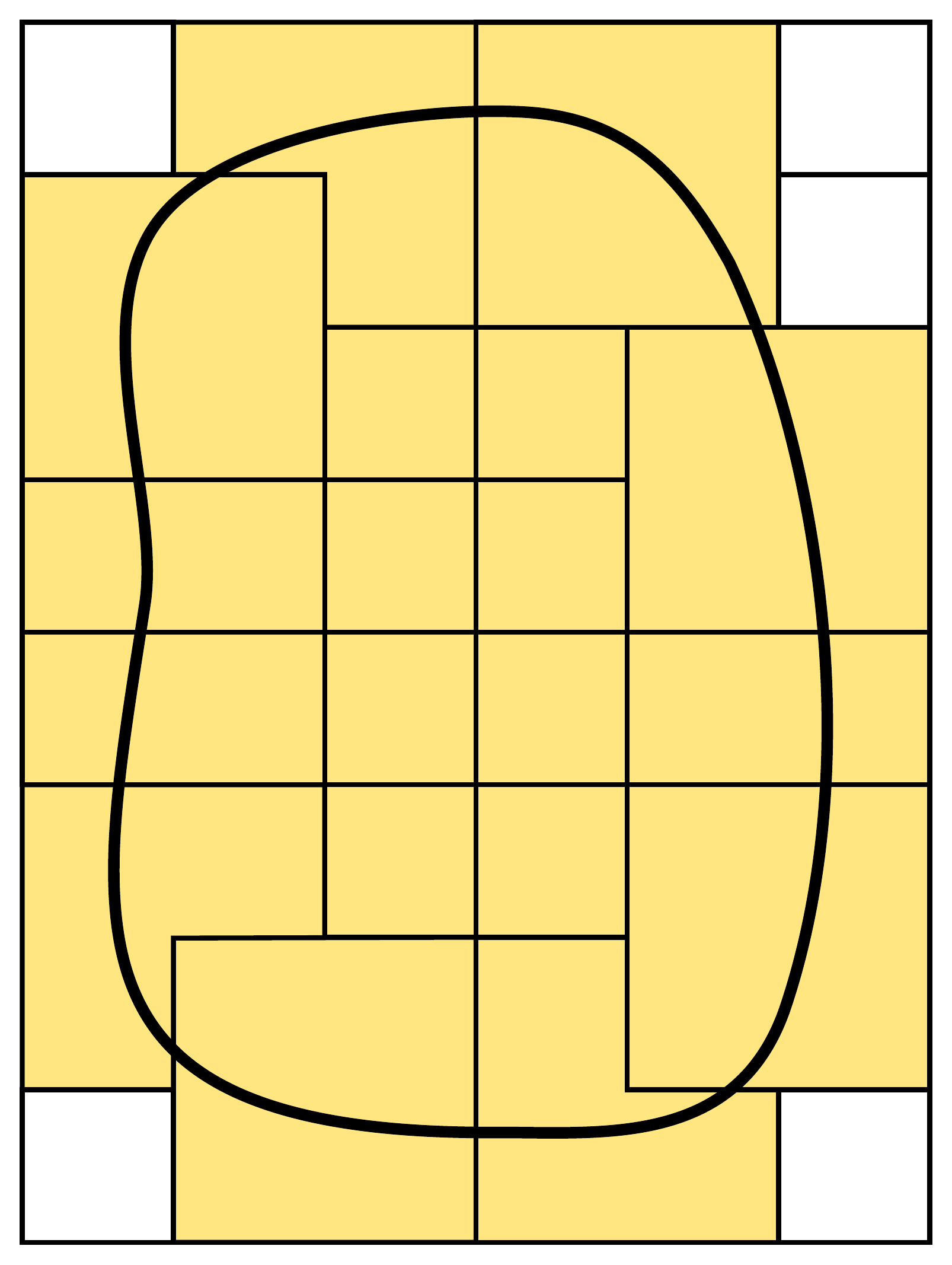}
    \caption{Step 4.}
    \label{fig:aggr-steps-d}
  \end{subfigure} 
  \caption{Illustration of the cell aggregation scheme defined in Algorithm \ref{alg:agg_sch}. We note that the definition of an aggregate in \eqref{eq:def_agg} is such that it only considers the part of the aggregated cells inside $\Omega$ as this simplifies the notation in the numerical analysis.}
  \label{fig:aggr-steps}
\end{figure}

In the forthcoming sections, we need an upper bound of the size of the aggregates generated with Algorithm \ref{alg:agg_sch} in terms of the cell mesh size $h$, i.e., the characteristic size of an aggregate is bounded by $\gamma h$ for some $\gamma$ independent of $h$ and the cut cell intersection with the boundary. We refer to \cite[Lem. 2.2]{Badia2017} and the subsequent discusion for a bound of this quantity, supported with the numerical experiments in \cite[Sect. 6.3]{Badia2017}.

Alg. \ref{fig:aggr-steps} leads to another partition $\meshag$ into aggregates, where an aggregate is defined in terms of a set of cells as follows:
\begin{equation} A \doteq \{ \cup_{i=0}^{n_A} K_i \cap \Dom : \ K_i \in \mesh \}, \qquad \forall A \in \meshag,\label{eq:def_agg}\end{equation}
where (without loss of generality) $K_0 \in \meshin$ is the owner interior cell, also represented with $\mathcal{O}(A)$. By construction of Algorithm \ref{alg:agg_sch}, it holds: 1) $n_A \geq 0$; 2) interior cells that have no aggregated cut cells ($n_A = 0$) remain the same; 3) there is only one interior cell per aggregate, i.e., $K_i \not\subset \Omega$ for $i>0$; 4) every cut cell belongs to one and only one aggregate.

For a interior/cut cell $K \in \meshact$, we define its owner (interior) cell $\mathcal{O}(K)$ as the owner $\mathcal{O}(A)$ of the only aggregate $A \in \meshag$ that contains the cell, i.e., $K \cup A$ has non-zero measure in dimension $d$. Thus, the owner of an interior cell is the cell itself. 
\footnote{Other aggregation algorithms could be considered, e.g., touching in the first step of the algorithm not only
the interior cells, but also cut cells \emph{without} the small cut cell problem. It can be implemented by defining the quantity $\eta_\cell \doteq \frac{| \cell \cap \Omega|}{ |\cell|}$ and touch in the first step not only the interior cells but also any cut cell with $\eta_\cell > \eta_0 > 0$ for a fixed value $\eta_0$.}


We can also construct a map that, given an \emph{outer} \ac{vef}, i.e., a \ac{vef} that belongs to at least one cut cell in $\meshcut$ but does not belong to any interior cell in $\meshin$, provides its aggregate owner among all the aggregates that contain it (see Fig.  \ref{fig:def-outnodmap}). This map can be arbitrarily built, e.g., we can consider the \emph{smallest} aggregate that contain the \ac{vef}. 
The map between the outer \ac{vef} $b$ and the interior cell owner is \emph{also} represented with $\owner{b}$.\footnote{After the cell aggregation and the \ac{vef} owner definition, we have defined a map $\owner{\cdot}$ such that, given any outer \ac{vef} or cut cell, provides its owner (interior) cell.} 


\begin{figure}[ht!]
  \centering
    \begin{subfigure}{0.33\textwidth}
    \centering    
    \includegraphics[width=0.8\textwidth]{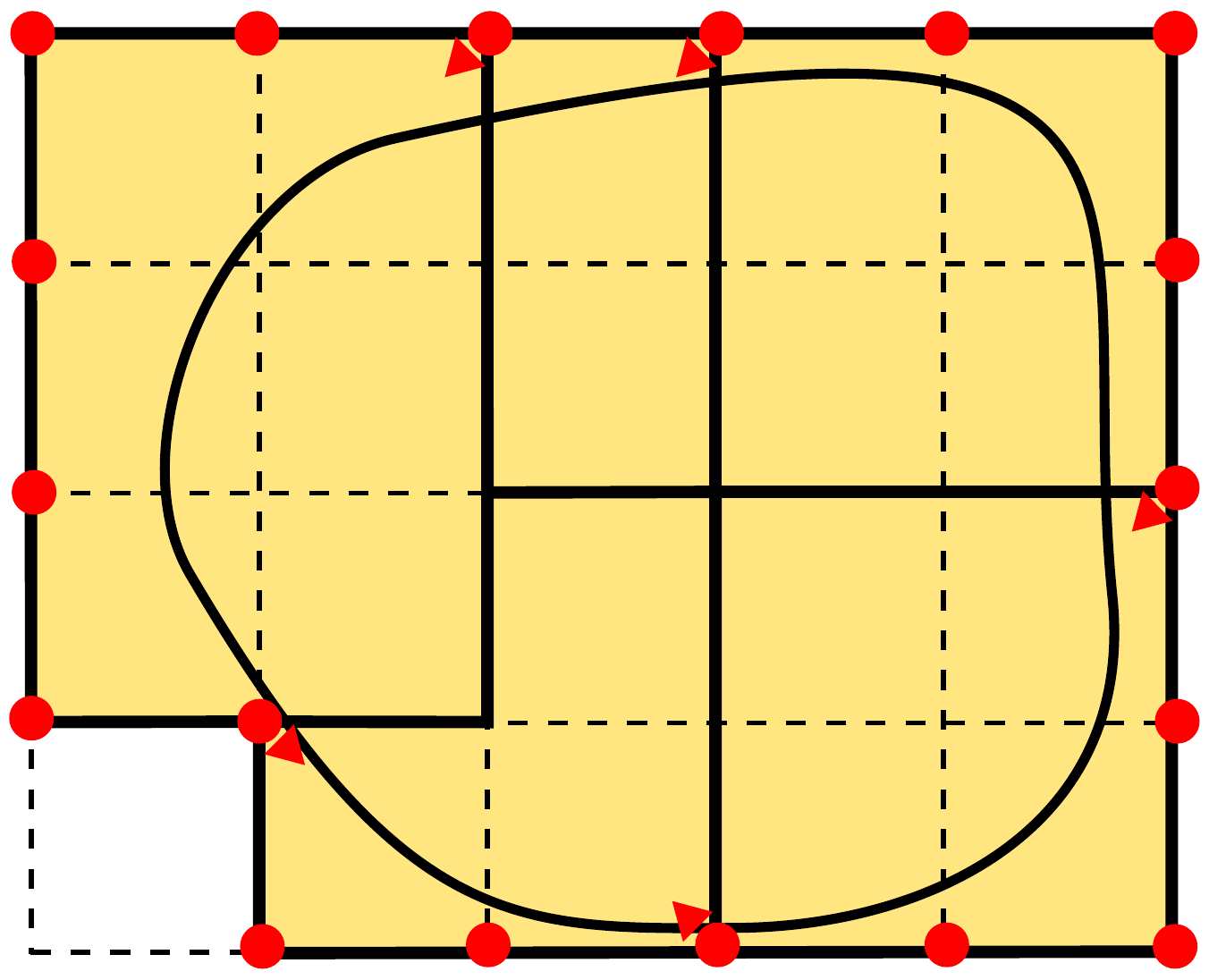}
    \caption{Outer vertex to aggregate map}
    \label{fig:def-outnodmap-a}
  \end{subfigure}
  \begin{subfigure}{0.33\textwidth}
    \centering
    \includegraphics[width=0.8\textwidth]{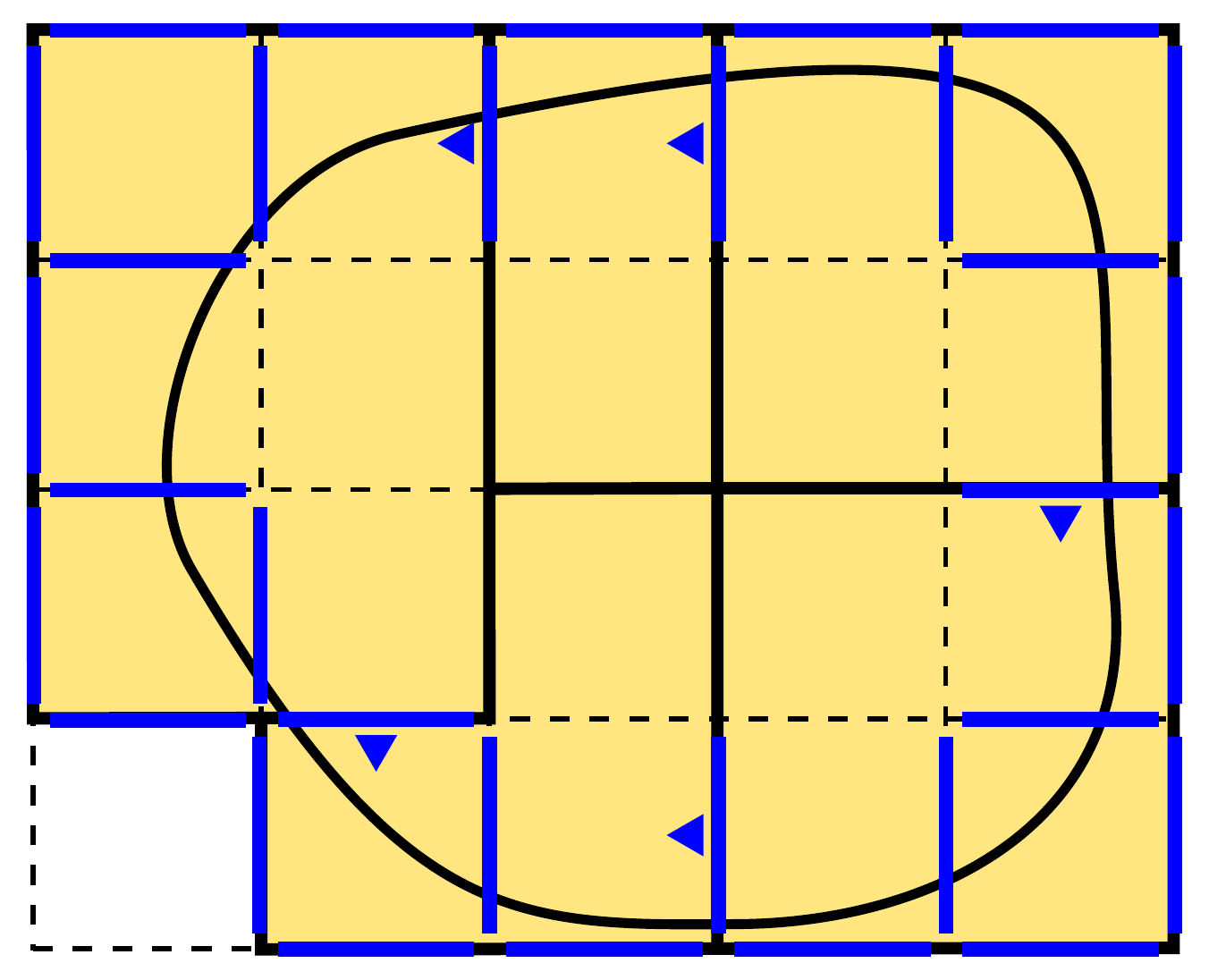}
    \caption{Outer face to aggregate map}
    \label{fig:def-outnodmap-b}
  \end{subfigure}
  \begin{subfigure}{0.18\textwidth}
  \begin{tabular}{cl}
    \includegraphics[width=3.2em]{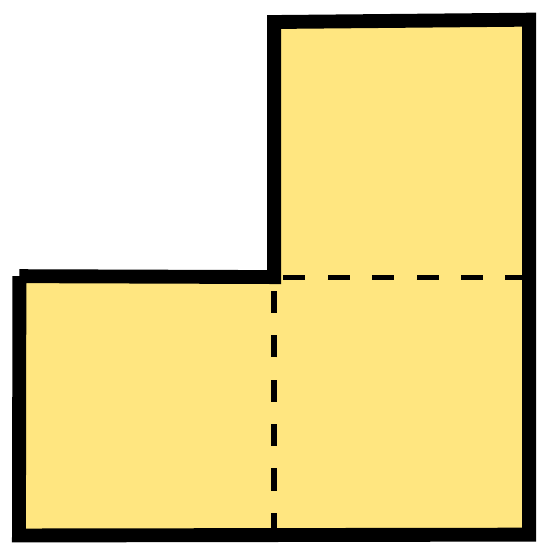} & Aggregate\\
    \includegraphics[width=1.6em]{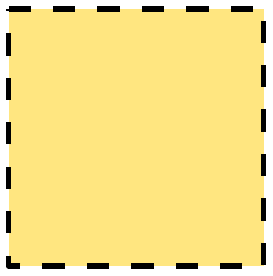} & Cell\\
        \tikz{ \draw[blue,line width=2pt] (0,0) -- (1.4em,0);} & Outer face\\
    \tikz{ \fill[red] (0,0) circle (0.2em);} & Outer vertex
  \end{tabular}
  \end{subfigure}
  \caption{Map from outer faces and vertex to aggregate owner.  The small pointers denote the aggregate owner. 
   Pointers are not used for vertex and faces belonging to only one aggregate  since the owner is obvious. { Aggregates in ({\sc a}) and ({\sc b}) are the same, but the aggregate bounds are clearer in ({\sc a}.}}
     \label{fig:def-outnodmap}
\end{figure}

\subsection{Aggregated finite element spaces} \label{sec:agg_unf}

Our goal is to define \ac{fe} spaces using the cell aggregates introduced above, in order to end up with unfitted \ac{fe} spaces on the domain $\Dom$, with optimal approximability properties not affected by the small cut cell problem. In this work, the spaces will eventually be used for the interpolation of every velocity component and the pressure in the Stokes problem. Thus, it is enough to define the  \ac{agfe} spaces for a generic scalar-valued field. 

Let us represent with $\fespst(\omega)$ a generic global and continuous Lagrangian \ac{fe} space, i.e., it can be $\Qqh$ for hex meshes and $\Pqh$ for tet meshes, for an arbitrary order $q$. We  introduce the \emph{active} \ac{fe} space associated with the active portion of the background mesh $\fespex \doteq \fespst(\meshex)$ and the \emph{interior} \ac{fe} space $\fespin \doteq \fespst(\meshin)$. 
The active \ac{fe} space $\fespex$ (see Fig. \ref{fig:def-spaces-c})
is the functional space typically used in unfitted \ac{fe} methods
(see, e.g.,
\cite{badia_robust_2017,de_prenter_condition_2017,schillinger_finite_2014}).
It is well known that $\fespex$ leads to arbitrary ill conditioned
systems when integrating the \ac{fe} weak form on the physical domain
$\Dom$ only (if no stabilization technique is used to remedy it).  It is
obvious that the interior \ac{fe} space $\fespin$ (see
Fig. \ref{fig:def-spaces-a}) is not affected by this problem, but it
is not usable since it is not defined on $\Dom$.

Herein, we propose an alternative \ac{agfe} space
$\fespag$ that is defined on $\Domap$ but does not present the
ill-conditioning issues related to $\fespex$. To this end, we can
define the set of nodes of $\fespin$ and $\fespex$ as $\nodesgin$ and
$\nodesgex$, respectively (see Fig. \ref{fig:def-spaces}). We define
the set of \emph{outer} nodes as
$\nodesgou \doteq \nodesgex \setminus \nodesgin$ (e.g., the nodes that belong to outer \acp{vef} in Fig.
\ref{fig:def-outnodmap}). The outer
nodes are the ones that can lead to conditioning problems due to the
small cut cell problem (see, e.g.,
\cite{de_prenter_condition_2017}). The space of global shape functions
of $\fespin$ and $\fespex$ can be represented as
$\{ \shpf{b} \, : \, b \in \nodesgin\}$ and
$\{ \shpf{b} \, : \, b \in \nodesgex\}$, respectively. Any function
$\ush \in \fespin$ can be written as
$\ush = \sum_{a \in \nodesgin} u^a_h \shpf{a}$; analogously for
functions in $\fespex$. The space $\fespag$ is defined taking as
starting point $\fespex$, and adding judiciously defined constraints
for the nodes in $\nodesgou$.

\begin{figure}[ht!]
  \centering
  \begin{subfigure}{0.24\textwidth}
    \centering
    \includegraphics[width=0.9\textwidth]{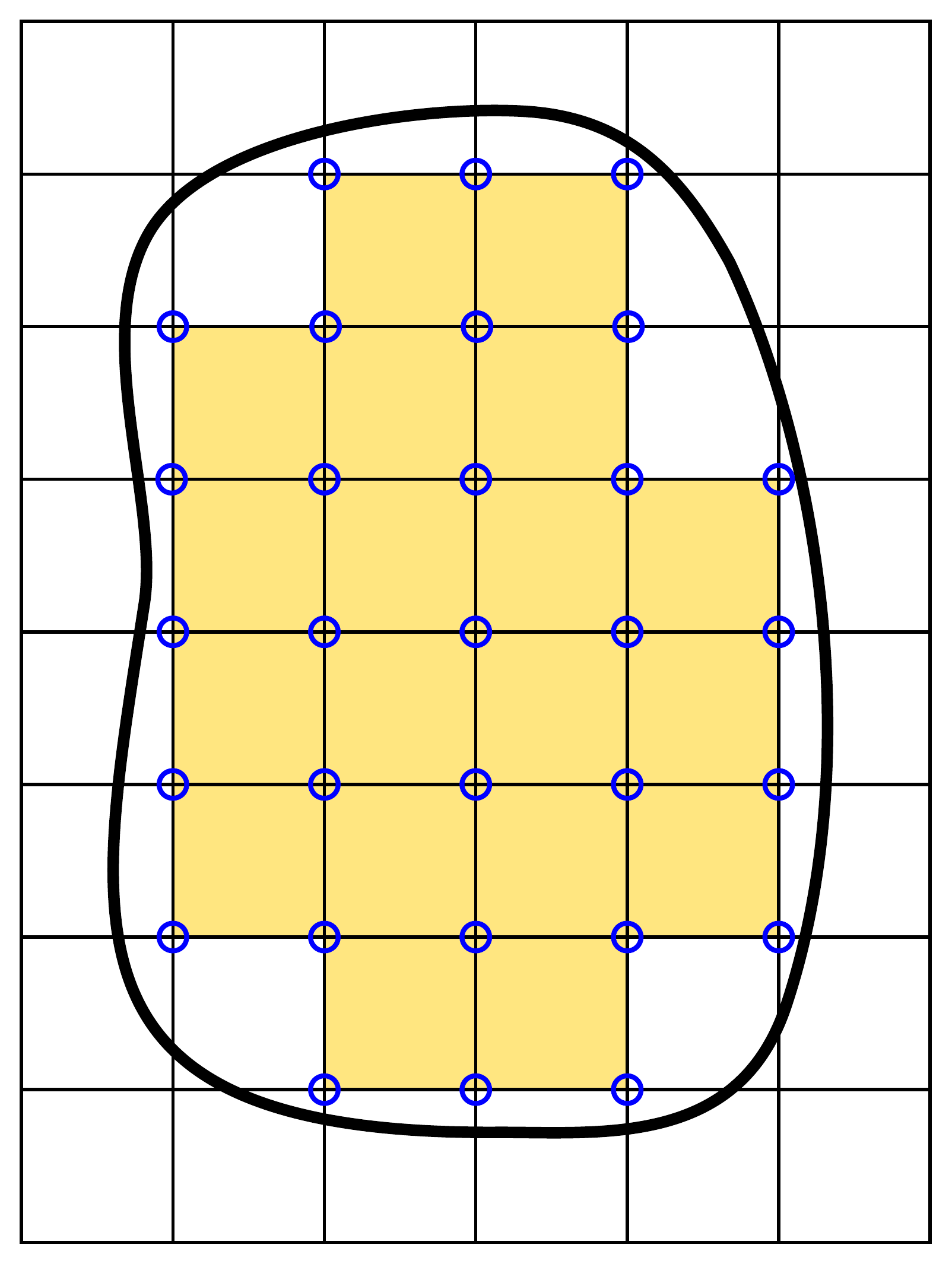}
    \caption{$\fespin$}
    \label{fig:def-spaces-a}
  \end{subfigure}
  \begin{subfigure}{0.24\textwidth}
    \centering
    \includegraphics[width=0.9\textwidth]{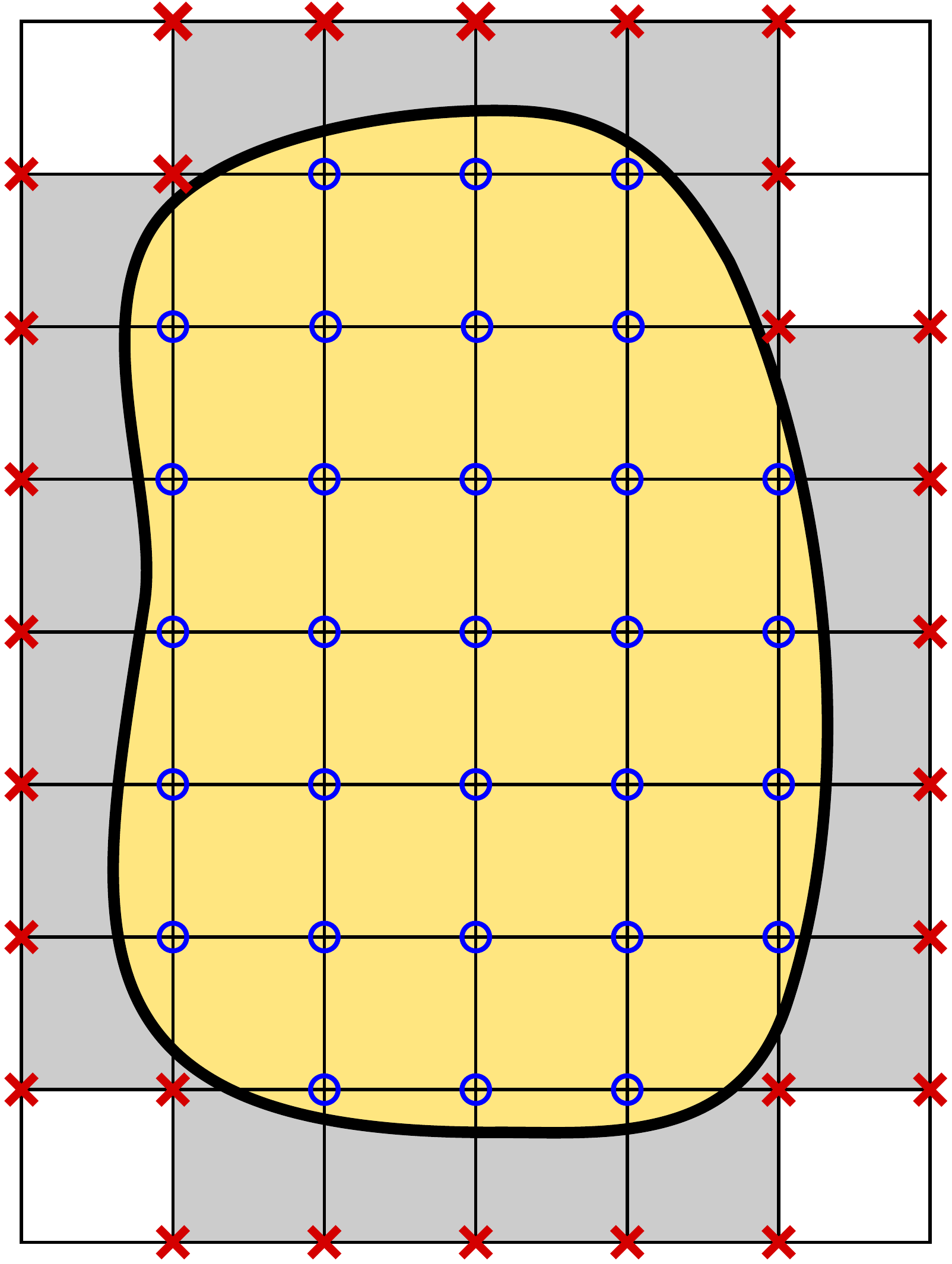}
    \caption{$\fespag$}
    \label{fig:def-spaces-b}
  \end{subfigure}
  \begin{subfigure}{0.24\textwidth}
    \centering
    \includegraphics[width=0.9\textwidth]{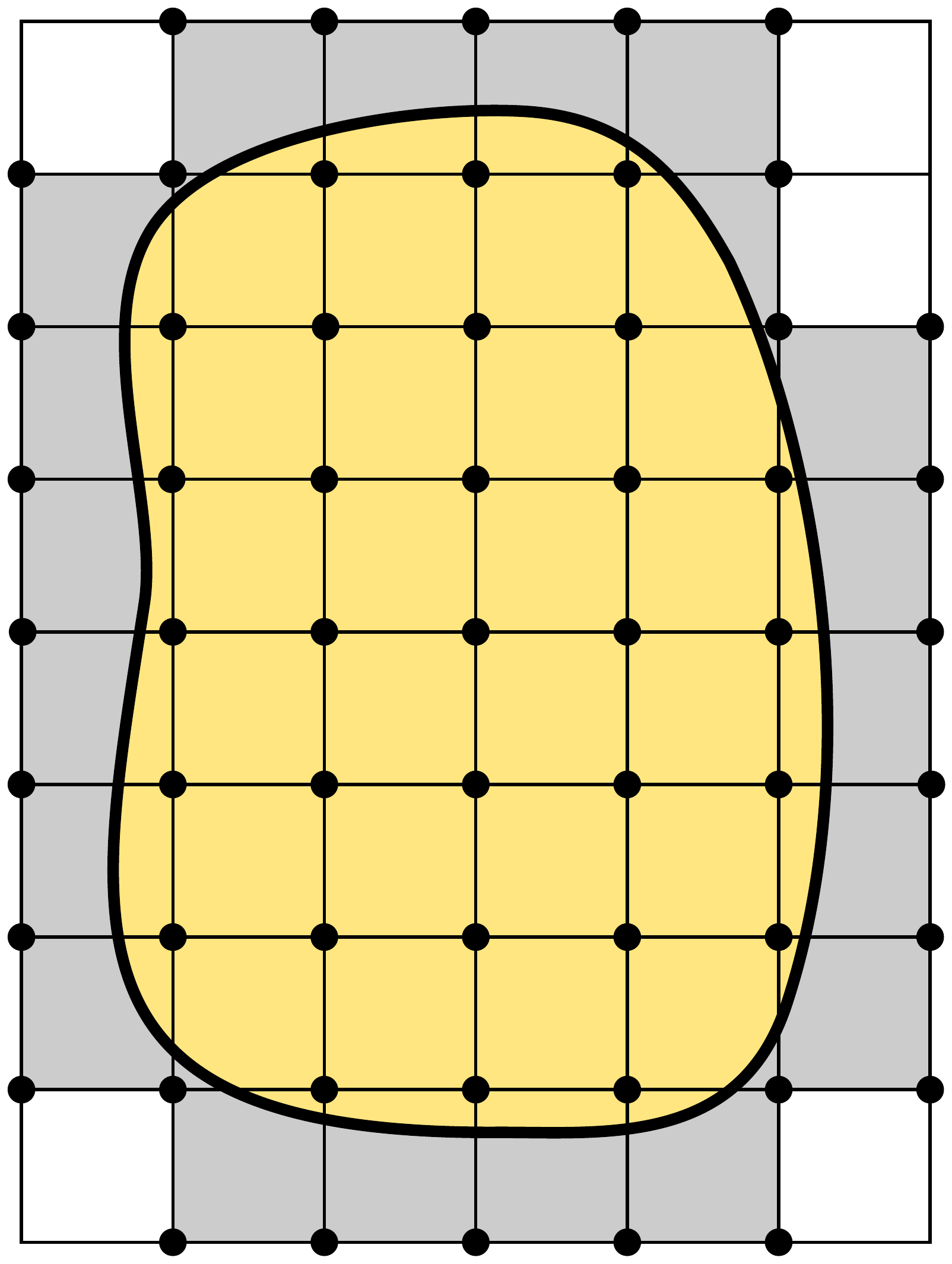}
    \caption{$\fespex$}
    \label{fig:def-spaces-c}
  \end{subfigure}
  \begin{subfigure}{0.2\textwidth}
    \begin{tabular}{l}
      {\color{blue} $\circ$} nodes in $\nodesgin$\quad\quad
      \\      
      {\color{black} $\bullet$} nodes in $\nodesgex$
      \\      
      {\color{red} $\times$} nodes in $\nodesgou$\quad\quad
    \end{tabular}
  \end{subfigure}
  \caption{Finite Element spaces.}
  \label{fig:def-spaces}
\end{figure}

In order to define $\fespag$, we observe that, in nodal Lagrangian \ac{fe} spaces, there is a one-to-one map between \acp{dof} and nodes (points) of the \ac{fe} mesh. For globally continuous \ac{fe} spaces, we can define the owner {\ac{vef}} of a node as the lowest-dimensional {\ac{vef}} that contains the node.  As a result, the geometrical outer-\ac{vef}-to-cell-owner map above leads to an outer-\ac{dof}-to-cell-owner map too. Making abuse of notation, we also define the \ac{dof} map as $\owner{b}$ for an outer \ac{dof} $b$.

Given a function $\vsh \in \fespin$ and a cell $\cell \in \meshin$, we define the unique polynomial $\xi_h^{\cell}(v_h): \mathbb{R}^d \rightarrow \mathbb{R}$ such that its restriction  to the cell $\cell$ coincides with the \ac{fe} function, i.e., $\vsh(\x) = \xi_h^{\cell}(v_h)(\x)$, $\x \in \cell$. With these ingredients, we define $\fespag \subset \fespex$ as the subset of functions in $\fespex$ such that, for any \ac{dof} $a \in \nodesgou$,
\begin{align}
\vsh^a = \sigma^a( \xi_h^{\owner{a}}(\vsh)) =\sum_{b \in \nodes{\owner{a}}} \sigma^a (\xi_h^{\owner{a}}(\shpf{b})) \sigma^b(\vsh).
  \label{eq:con_abs}
\end{align}
By construction, functions in $\fespag$ are uniquely determined by the \acp{dof} of $\fespin$. Thus, we can define the extension operator $\mathcal{E}: \fespin \rightarrow \fespag \subset \fespex$, such that, given $\ush \in \fespin$ provides the \ac{fe} function $\ext{\ush} \in \fespex$ with outer nodal values computed as in \eqref{eq:con_abs}. Thus, the \ac{agfe} space is the range of this operator, i.e., $\fespag \doteq \ext{\fespin} \subset \fespex$. 
 Since $\fespag \subset \fespex$, if $\fespin$ and $\fespex$ are $\mathcal{C}^0$ continuous, so it is $\fespag$. We note that \eqref{eq:con_abs} has sense for continuous and discontinuous spaces, and both tensor-product and serendipity spaces for hex meshes.

\subsubsection{Nodal Lagrangian aggregated finite element spaces}\label{sec:nod_lag_agg_fe_spa}
 
In particular, for nodal-based Lagrangian \ac{fe} spaces (which include tensor-product $\Qqh(\mesh)$  and serendipity spaces $\Qqhm(\mesh)$ for hex meshes and $\Pqh(\mesh)$ for tet meshes), the previous expression is reduced to:
\begin{equation}
  \vsh(\x^a) = \sum_{b \in \mathcal{N}(\owner{a})} \vsh(\x^b) \phi^b(\x^a).
  \label{eq:con_lag}
\end{equation}
The computation of the constraint is straightforward, and simply involves to evaluate the shape function polynomials of a cell in a set of points that do not belong to the cell, viz., the nodes of an aggregated cut cell. The definition of \ac{dof} ownership is simple, the \ac{vef} or cell that contains the node related to the \ac{dof} with minimum dimension, which is uniquely defined.

\subsubsection{Discontinuous aggregated finite element spaces} \label{sec:dis_agg_fe_spa}
Let us comment on discontinuous \ac{fe} spaces, e.g., $\Pqhd(\mesh)$ on hex meshes. In this case, all the \acp{dof} belong to the cell itself, since no continuity must be enforced. Thus, the \acp{dof} owner and \acp{dof}-to-cell maps are trivial once defined the cell aggregation. Since no continuity must be enforced among cells, the \acp{dof} definition is very flexible. The definition in \eqref{eq:con_abs} is general and can be used for discontinuous spaces with \acp{dof} that are not nodal evaluations. It is easy to check that for discontinuous spaces, the \ac{agfe} space can be analogously defined as:
\begin{align}
\fespag = \{ v  : \, v|_A \in
\Pq^-(A), \, \hbox{for any} \ A \in \meshag \}.\label{eq:dg_con}
\end{align}
The equivalence between the definition based on  \eqref{eq:con_abs} and the one in \eqref{eq:dg_con} is straightforward. The use of aggregation techniques within \ac{dg} methods has already been used, e.g., in \cite{Johansson2013,kummer_extended_2013}.
  
\subsubsection{Aggregated finite elements with serendipity extension}\label{ssec:ser_ext}

Up to now, we have assumed that the constraints for the extension operator were computed using the same shape functions as the ones of the local \ac{fe} space in the owner interior cell (see \eqref{eq:con_abs}). Here, we consider a more general case in which these two shape functions bases (and the corresponding spanned spaces) can differ. In particular, we are interested in using a tensor-product Lagrangian space at all cells in a hex mesh, but to compute the constraints through the corresponding serendipity basis (preserving the order of approximation). As we will see later on, it does not affect accuracy and has positive properties when considering stable mixed \ac{agfe} spaces.

Let us introduce some notation, in order to distinguish between tensor-product and serendipity \ac{fe} spaces. For serendipity \acp{fe} and hex meshes, i.e., $\Qqm(\cell)$, we represent its unisolvent set of nodes with $\nodesm{\cell}$, i.e., the corresponding nodal values are a basis for the dual space, with cardinality $\breve{n}_{\moments}$ (see \cite[Fig. 1]{Arnold2011}). The corresponding shape functions and \acp{dof} are represented with $\{ \breve{\phi}^a \}_{a \in \nodesm{\cell}}$ and $\{ \breve{\sigma}^a \}_{a \in \nodesm{\cell}}$, respectively. For serendipity spaces, we denote its corresponding nodal interpolator in \eqref{eq:loc_int} as $\breve{\pi}^I_K(v)$.

We constrain every outer \ac{dof} $a \in \nodesgou$ of a function $\vsh \in \fespex$ as
\begin{equation}
\vsh^a = \sigma^a( \xi_h^{\owner{a}}\circ \breve{\pi}^I_{\owner{a}}(\vsh))) =\sum_{b \in \nodesm{\owner{a}}} \sigma^a (\xi_h^{\owner{a}}(\breve{\phi}^{b})) \breve{\sigma}^b(\vsh),
\label{eq:con_ans_ser}
\end{equation}
or analogously,
\begin{equation}
\vsh(\x^a) = \sum_{b \in \nodesm{\owner{a}}}  \vsh(\x^b) \breve{\phi}^{b}(\x^a).
\label{eq:con_ser}
\end{equation}
It leads to the new \ac{agfe} space $\fespagm$ and its corresponding extension operator $\breve{\mathcal{E}}_h$.
\subsection{Mathematical properties}

In the following, we list some \ac{fe} inequalities that will be used in the next sections. We use $A \lesssim B$ to say that $A < C B$ for some positive constant $C$; analogously for $\gtrsim$ and $\eqsim$. We use $C$ to denote such a constant, which can be different in different appearances. The word \emph{constant} in this work always denotes independence with respect to $h$ and the cut cell intersection, i.e., it is not affected by the small cut cell problem.

Let us consider an arbitrary \ac{fe} space $\fespst$. The following inverse inequalities hold (see, e.g., \cite{brenner_mathematical_2010}):
\begin{align}
  \normregl2{\grad \ush}{\cell} & \lesssim \hc^{-1} \normregl2{ \ush}{\cell}, \label{eq:inv_ine1} \\
  \normregl2{\dn \ush}{\bou \cap \cell} & \lesssim \hc^{-\frac{1}{2}} \normregl2{\grad \ush}{{\cell}}, \label{eq:inv_ine2}
\end{align}
where 
$\n$ is the outward normal (in this appearance, with respect to $\bou_{\rm D} \cap \cell$),
and $\dn \doteq \n \cdot \grad$. Furthermore, we have the following trace inequalities  (see \cite{hansbo_unfitted_2002}):
\begin{align}
  \normregl2{\ush}{\bcell} & \lesssim \hc^{-\half} \normregl2{\ush}{\cell} + \hc^\half \normregl2{\grad \ush}{\cell}.\label{eq:tra_ine}\\
  \normregl2{\ush}{\bou \cap \cell} & \lesssim \hc^{-\half} \normregl2{\ush}{\cell} + \hc^\half \normregl2{\grad \ush}{\cell}, \label{eq:tra_ine2}
\end{align}
(We note that \eqref{eq:tra_ine2} implies \eqref{eq:tra_ine}). The extension operators $\ext{\cdot}$ and $\extm{\cdot}$ satisfy the following stability bounds. The standard extension operator can be considered for both tet and hex meshes, whereas the serendipity extension operator only for hex meshes. 
\begin{lemma}\label{lm:sta_agg_ext}
 Given a function $\ush \in \fespin$, it holds:
\begin{align}
  & \normregl2{\ext{\ush}}{\Domex} \lesssim  \normregl2{\ush}{\Domin}, \quad && \normregl2{{\extm{\ush}}}{\Domex}  \lesssim  \normregl2{\ush}{\Domin}, \\
  & \normregl2{\grad \ext{\ush}}{\Domex}  \lesssim  \normregl2{\grad \ush}{\Domin}, \quad &&
\normregl2{\grad \extm{\ush}}{\Domex}  \lesssim  \normregl2{\grad \ush}{\Domin}.
\end{align}
\end{lemma}
\begin{proof}
The proof for $\ext{\cdot}$ can be found in \cite[Corollary 5.3]{Badia2017} for a general \ac{agfe} space, that can be either $\fespin$ or the discontinuous \ac{fe} space of its gradients. The results for $\extm{\cdot}$ can be proved analogously.
\end{proof}

Given the interior \ac{fe} space $\fespin$, we can define the standard Scott-Zhang interpolation using the standard definition in \cite{scott_finite_1990}. Let us define an \emph{extended} Scott-Zhang interpolant as follows: 1) perform the standard interior Scott-Zhang interpolator onto $\fespin$ through the assignment for every interior \ac{dof} $a \in \nodesgin$ of an arbitrary \ac{vef}/cell\footnote{Even though this choice is arbitrary, we do not permit $\tilde{K}_a \subset \Domin$ to be a vertex, since it would restrict the applicability of the interpolator to $\mathcal{C}^0(\Dom)$ functions with pointwise sense. We note that the concept of \ac{vef}/cell ownership of a \ac{dof} can be extended to non-nodal \acp{dof} (see, e.g., \cite{badia_fempar:_2017}).} $\tilde K_a \subset \Domin$ that contains the owner \ac{vef} of $a$,  and compute the mean value of the function on $\tilde{K}_a$, represented with $\sigma^{SZ,a}_{\tilde{K}_a}(\cdot)$; 2) extend the interior function to $\Omega$ using the extension operator $\ext{\cdot}$ (or $\extm{\cdot}$), leading to a function in $\fespag$ (or $\fespagm$). Thus, the extended Scott-Zhang interpolant reads:
  $$ \pi_h^{SZ}(u)(\x) \doteq \sum_{a \in \nodesgin} \sigma^{SZ,a}_{\tilde{K}_a} (u) \mathcal{E}(\phi^a(\x)).$$
The serendipity-extended interpolant, represented with $\breve{\pi}_h^{SZ}(u)$, is obtained as above, but using $\extm{\cdot}$ instead.
\def\nodalval#1{{\underline{\boldsymbol{\sigma}}(#1)}}

In the next theorem, we prove the approximability properties of the extended Scott-Zhang interpolant. In the statement of the theorem, we represent with $\omega(A)$ the union of the owner of the aggregate itself and the owners of all its neighbors, i.e., $\omega(A) \doteq \{  \mathcal{O}(B) : A \cap B \neq \emptyset, \ B \in \meshag \}$. We note that $A \not\subseteq \omega(A) \subset \Domin$ in general.

\begin{theorem}\label{th:app_sco_zha}
  Let us consider an \ac{agfe} space $\fespag$ such that $\Pq(A) \subset \fespag(A)$ for $A \in \meshag$, $q \in \mathbb{N}^+$. Let us consider a function $u \in W_p^m(\Omega)$, where 
  $1 \leq p \leq \infty$, $m \leq q + 1$, and $m \geq d$ for $p=1$ or $m > \frac{d}{p}$ for $p>1$. It holds:
  \begin{align}
    & \| u - \pi_h^{SZ}(u) \|_{W_p^s(A)} \lesssim \h^{m-s} | u |_{W_p^m(\omega(A))}, \qquad
    & \| u - \pi_h^{SZ}(u) \|_{W_p^s(S)} \lesssim \h^{m-s-\frac{1}{2}} | u |_{W_p^m(\omega(A))}, \label{eq:sco_zha_int}
\end{align}  
for $1 \leq s \leq m$, $A \in \meshag$, and $S$ being the intersection between a plane in $\mathbb{R}^d$ and $A$. The same results apply for the serendipity-extended \ac{agfe} space $\fespagm$ and its corresponding interpolant $\breve{\pi}_h^{SZ}(\cdot)$.
\end{theorem}
\begin{proof}
  The standard and serendipity interpolants can be analyzed analogously. The Scott-Zhang moments $\sigma^{SZ,a}_{\tilde{K}_a}(\cdot)$ are bounded in $W^m_p(\Omega)$ owing to the trace theorem, i.e., $W_p^l(\Omega) \subset L^1(\tilde K_a)$ for $\tilde K_a$ being a facet or cell (see \cite{scott_finite_1990}). On the other hand, $\mathcal{E}(\phi^a(\x)) \in W_\infty^m(\Omega) \subset W_p^m(\Omega)$, since it is a combination of shape function with bounded nodal values (see \eqref{eq:con_abs} and Lem. \ref{lm:sta_agg_ext}). Moreover, from the definition of the extension operator, the nodal values of $\pi_h^{SZ}(\cdot)|_A$ are constrained from the \acp{dof} of the owner interior cell of $A$ or the \acp{dof} of the owner cell of a neighbor of $A$. Thus, 
  we readily obtain that  $\| \pi_h^{SZ}(u) \|_{W_p^m(A)} \leq C \| u \|_{W_p^m(\omega(A))}$. Next, we consider an arbitrary function $\pi(u) \in W_p^m(\Dom)$ such that $\pi(u)|_\cell \in \Pq(\omega(A)) \subset \fespag(\omega(A))$ (note that the inclusion also holds for the serendipity extension). The fact that $\pi_h^{SZ}(\cdot)$ is a projection onto $\fespag$ by construction yields $\pi(u)|_A = \pi_h^{SZ}(\pi(u))|_A$. Thus, we have:
  \def\int#1{\pi_h^{SZ}(#1)}
\begin{align}
 \| u - \int{u} \|_{W_p^m(A)}  & \leq 
  \| u - \pi(u)  \|_{W_p^m(A)}  +  \| \int{\pi(u) - u} \|_{W_p^m(A)} \\
& \lesssim
  \| u - \pi(u)  \|_{W_p^m(A)}  +   \| \pi(u) - {u} \|_{W_p^m(\omega(A))} \lesssim \| u - \pi(u)  \|_{W_p^m(\omega(A))}.
\end{align}
Since $\omega(A)$ is an open bounded domain with Lipschitz boundary by definition, one can use the Deny-Lions lemma (see, e.g., \cite{ern_theory_2004}). As a result, using the $\pi(u)$ that minimizes the right-hand side, it holds:
\begin{align}
\| u - \int{u} \|_{W_p^m(A)} \lesssim | u |_{W_p^m(\omega(A))}.
\end{align}
The Sobolev embedding theorem and the trace theorem yield:
\begin{align}
& \| u - \pi_h^{SZ}(u) \|_{W_p^s(A)} \leq C(A) | u |_{W_p^m(\omega(A))}, \qquad \\
&\| u - \pi_h^{SZ}(u) \|_{W_p^{s-\frac{1}{2}}(S)} \leq \| u - \pi_h^{SZ}(u) \|_{W_p^{s}(A)} \leq C(A) | u |_{W_p^m(\omega(A))}.
\end{align}
Using standard scaling arguments, we prove the lemma. 
\end{proof}

\section{Approximation of the Stokes problem} \label{sec:app_sto}
In this section, we consider the \ac{fe} approximation of the Stokes problem \eqref{eq:sto_eqs} using \ac{agfe} spaces on unfitted meshes. We focus on extended by aggregation inf-sup stable spaces (velocity-pressure pairs of \ac{fe} spaces that satisfy a discrete version of the inf-sup condition on body-fitted meshes) with additional stabilizing terms to cure the potential deficiencies of the \emph{unfitted} inf-sup condition. In this section, the velocity and pressure spaces are represented with $\Vh$ and $\Qh$, respectively. As usual in unfitted \ac{fe} methods, the Dirichlet boundary conditions cannot be enforced strongly. Instead, we consider a Nitsche-type weak imposition of the Dirichlet data \cite{becker_mesh_2002,nitsche_uber_1971}. It provides a consistent numerical scheme with optimal converge rates (also for high-order elements) that is commonly used in the embedded boundary community (see, e.g., \cite{massing_stabilized_2014} for its application in unfitted discretizations of the Stokes problem). Another important ingredient in unfitted \ac{fe} approximations is the integration on cut cells. We refer to \cite{badia_robust_2017} for a detailed exposition of the particular technique used in this paper.
With these ingredients, we define the Stokes operator:
\begin{align}
  \Ah{\uh}{\ph}{\vh}{\qh} \doteq \ah{\uh}{\vh} + \bh{\vh}{\ph} + \bh{\uh}{\qh} - \ch{\uh}{\ph}{\vh}{\qh}, \label{eq:sto_ope}
\end{align}
where
\begin{align}
  \ah{\uh}{\vh} & \doteq  \ip{ \grad \uh}{\grad \vh} - \ipreg{\dn \uh}{\vh}{\bou} - \ipreg{\dn \vh}{\uh}{\bou} + \pen \ipreg{\h\uh}{\vh}{\bou}, \label{eq:ah_for} \\
  \bh{\vh}{\ph} & \doteq - \ip{\div \vh}{ \ph} + \ipreg{\n \cdot \vh}{\ph}{\bou}, \label{eq:bh_for}
\end{align}
with $\pen$ a large enough positive constant, for stability purposes. The right-hand side reads:
\begin{align}
\Lh{\vh}{\qh} \doteq \ip{\f}{\vh} + g_h(\f,\vh).
\end{align}
The pressure stabilization term $j_h$ and the corresponding potential modification of the right-hand side $g_h$ to keep consistency will be defined in Sect. \ref{ssec:sta_mix_fe_pre_sta}, motivated from the numerical analysis. The discrete Stokes problem finally reads: find $(\uh,\, \ph) \in \Vh \times \Qh$ such that
\begin{align}
\Ah{\uh}{\ph}{\vh}{\qh} = \Lh{\vh}{\ph}, \qquad \forall \, (\vh, \, \qh) \in \Vh \times \Qh. \label{eq:dis_sto}
\end{align}

In the following analysis, we restrict ourselves to hexahedral meshes and discontinuous pressures. Similar ideas can be applied to inf-sup stable mixed \acp{fe} on tetrahedral meshes and continuous pressures, but we do not consider these cases for the sake of conciseness. Thus, using the notation in Sect. \ref{sec:agg_unf}, we will make use of the following global \ac{agfe} spaces: the space $\Qqh$, for $q \geq 1$, in which the local \ac{fe} space is the tensor-product Lagrangian $\Qq(\cell)$  in all cells $K \in \meshact$, and the constraints are defined using the standard expression in \eqref{eq:con_lag}; the space $\Qqhm$, for $q \geq 1$,  which only differs from the previous one in the constraint definition, based now on the serendipity extension in \eqref{eq:con_ser}; the discontinuous space $\Pqhd$, for $q \geq 0$, defined in \eqref{eq:dg_con}.

\section{Numerical analysis}\label{sec:num_ana}

In this section, we perform the stability analysis of \ac{fe} methods for \eqref{eq:sto_ope}. First, in Sect. \ref{ssec:abs_sta_ana}, we consider an abstract stability analysis , i.e., we prove an inf-sup condition under some assumptions over the mixed \ac{agfe} space and the stabilization terms. Two different algorithms that satisfy these assumptions, and thus are stable, are proposed in Sect. \ref{ssec:sta_mix_fe_pre_sta}. \emph{A priori} error estimates for these methods are obtained in Sect. \ref{ssec:a_pri_err_est}. Finally, in Sect. \ref{ssec:cond_nums}, we prove condition number bounds that are independent of the cut cell intersection with the boundary, i.e., the small cut cell problem.

The analysis of the discrete problem obviously relies on the well-posedness of
the continuous problem, i.e., the inf-sup condition in \eqref{eq:inf_sup_con}.
For the sake of conciseness in notation, we have not distinguished between the
actual computational domain $\Dom_h$ and the physical domain $\Dom$. However,
it is important to distinguish between these two in the definition of the
inf-sup constant, i.e., $\beta(\Dom)$ vs. $\beta(\Dom_h)$. In general,
$\beta(\Dom_h)$ can tend to zero as $h \to 0$. The lower bound for
$\beta(\cdot)$ relies on a decomposition of the domain into a finite number of
strictly star shaped domains. $\beta(\Dom_h)$ could tend to zero as $h \to 0$
unless one can prove that this number is bounded away from zero for $\Dom_h$.
It is in fact a problem for methods that rely on inf-sup conditions for
$\Domin$ (see \cite{Guzman2017}). Even though it is hard to imagine that a reasonable smooth approximation
$\Dom_h$ of $\Dom$ would require a decomposition into a number of star shaped
domains that blows up as $h \to 0$, there are constructions of $\Dom_h$ for
which one can prove that in fact $\beta(\Dom_h)$ is bounded below, or even
more, converges to $\beta(\Dom)$. In particular, if $\Dom_h$ is a polygonal
$h$-approximation of $\Dom$ in the sense of \cite[Def. 4.5]{Bernardi2016}, it
holds $|\beta(\Dom) - \beta(\Dom_h)| \leq c(\Dom)h$. In what follows, we simply
consider $\beta \doteq \inf_{h < h_0} \beta(\Omega_h)$, for $h_0$ a fine enough
mesh size to represent the topology of the geometry at hand.

\subsection{An abstract stability analysis}\label{ssec:abs_sta_ana}

In this section, we analyze the well-posedness of the discretization of the Stokes problem \eqref{eq:sto_ope} in an abstract setting, in which the \ac{fe} spaces and stabilization terms are not explicitly stated. Instead, we do the analysis under some assumptions of these ingredients.

We define the following norms:
\begin{align}
  \tnorm{ \u }^2 \doteq \normregl2{\grad \u}{\Dom}^2 +  \normregl2{h^{-\frac{1}{2}} \u }{\bou}^2, \qquad \tnorm{ \u, \, p}^2 \doteq \tnorm{ \u}^2 + \normregl2{p}{\Dom}^2.\label{eq:mes_nor}
\end{align}
In the following lemma, we prove some stability and continuity properties of the different terms that compose the Stokes operator in \eqref{eq:sto_ope}. 

\begin{lemma}
  It holds for any $\uh, \, \vh \in \Vh$
\begin{align} \label{eq:ah_bh_sta_con}
\ah{\uh}{\uh} \geq \gamma_a \tnorm{\uh}^2, \qquad \ah{\uh}{\vh} \leq \xi_a \tnorm{\uh} \tnorm{\vh}, \qquad \bh{\vh}{\qh} \leq \xi_b \tnorm{\vh} \tnorm{\qh}.
\end{align}
for $\pen$ a large enough positive constant in \eqref{eq:ah_for}.
\end{lemma}
\begin{proof}
  The continuity and stability of $a_h$ can be found, e.g., in \cite[Th. 5.7]{Badia2017}. The continuity of $b_h$ is obtained by using in its two terms the Cauchy-Schwarz inequality and in the second one the inequalities \eqref{eq:inv_ine1} and  \eqref{eq:tra_ine2} (see also \cite{massing_stabilized_2014}). 
\end{proof}

Next, we prove the unfitted inf-sup condition for \ac{agfe} spaces, using the following strategy. First, we introduce some definitions for the concept of improper facets and aggregates (Defs. \ref{def:agg_int_bub}-\ref{def:imp_int_set}  and \ref{def:agg_bub}-\ref{def:imp_agg_set}) which are the ones that will require some type of stabilization, due to the potential deficiency of the discrete inf-sup condition. Second, we prove a weak inf-sup condition for a particular type of mixed \ac{fe} spaces in Th. \ref{th:inf_sup_hig_ord}. Finally, using an abstract definition of the pressure stabilization term that satisfies Asm. \ref{ass:pre_sta}, we prove the stability of the \ac{agfe} method for the Stokes problem in Th. \ref{th:fin_inf_sup}. In the subsequent sections, we will consider different realizations of mixed \ac{fe} spaces, analyze how to determine a superset of improper facets/aggregates, and define a pressure stabilization fulfilling Asm. \ref{ass:pre_sta} for \eqref{eq:sto_ope} to be well-posed.

Let us define the set of aggregate interfaces:
$$ \qquad F_{AB} \doteq \partial A \cap \partial B, \ \ A, \, B \in \meshag,  \qquad 
\agints \doteq \{ F_{AB} \, : \, A, \, B \in \meshag \}.
$$
We note that, since  $A \subset \Omega$ for any aggregate $A \in \meshag$ by its definition in \eqref{eq:def_agg}, $F_{AB}$ can include a cut facet of a cut cell.

\begin{definition}[Aggregate interface bubble] \label{def:agg_int_bub}
Given an aggregate interface $F_{AB} \in \agints$ shared by $A, \, B \in \meshag$, an \emph{aggregate interface bubble } is a function $\bsphih^{F_{AB}} \in \Vh$ with $\sup(\bsphih^{F_{AB}}) \subseteq AB \doteq A \cup B$ such that 
  $$
  \int_{F_{AB}} \bsphih^{F_{AB}} \cdot \n > C | F_{AB} |> 0, \quad \| \bsphih^{F_{AB}} \|_{\infty} \leq C, \quad \bsphih^{F_{AB}} \cdot \n = {0} \ \ \hbox{on} \ \ \partial AB.$$
\end{definition}

\begin{definition}[Improper interface set of $\Vh$] \label{def:imp_int_set}
Given an aggregate interface $F_{AB}$ shared by $A, \, B \in \meshag$, it is improper with respect to $\Vh$ if there not exists any interface bubble satisfying the properties of Def. \ref{def:agg_int_bub}. The set of improper interfaces is denoted by $\agints^-$. Its complement is represented with $\agints^+ \doteq \agints \setminus \agints^-$.
\end{definition}

\begin{definition}[Aggregate  bubble] \label{def:agg_bub}
Given an aggregate $A \in \meshag$, an \emph{aggregate bubble} is a function ${\phi}^{A}_h \in \Qbord{2}(A) \cap H_0^1(A)$ with $\sup({\phi}^{A}_h) \subseteq A$ such that:
  \begin{align}\label{eq:agg_bub_pro}
    \bsphih^{A}(\x)  \geq 0,  \quad \int_{A} \bsphih^{A}  > C | A |> 0, \quad \| \bsphih^{A} \|_{\infty} \leq C.
    \end{align}
\end{definition}

\begin{definition}[Improper aggregate set of $\Vh$]  \label{def:imp_agg_set}
Given an aggregate  $A  \in \meshag$, it is improper with respect to $\Vh$  if there not exists any aggregate bubble satisfying the properties of Def. \ref{def:agg_bub}. The set of improper aggregates is denoted by $\meshag^-$. Its complement is represented with $\meshag^+ \doteq \meshag \setminus \meshag^-$.
\end{definition}

In the next lemma we propose a weak version of the standard Fortin interpolator (see, e.g., \cite{ern_theory_2004}), which we denote as \emph{quasi-}Fortin interpolator.
\begin{lemma}[Quasi-Fortin interpolant] \label{lm:qua_For_int}
For any $\v \in \H^1_{\bs{0}}(\Omega)$, there exists a function $\pi^{qF}_h(\v) \in \Vh$ such that:
\begin{equation} \label{eq:qua_For_pro}
\int_{F} \v \cdot \n = 
\int_{F} \pi_h^{qF}(\v) \cdot \n \quad \forall  F \in \agints^+,  \qquad \tnorm{\pi_h^{qF}(\v)} \leq \xi_{qF} \normregh1{\v}{\Dom},
\end{equation}
for a positive constant $\xi_{qF} > 0$.
\end{lemma}
\begin{proof}
Given a function $\v \in \H^1_{\bs{0}}(\Omega)$, let us consider, e.g., the extended Scott-Zhang interpolant $\pi_h^{SZ}(\vh)$ with the optimal approximability properties in Th. \ref{th:app_sco_zha}. For \ac{agfe} spaces with serendity extensions, we would consider $\breve{\pi}_h^{SZ}(\vh)$ instead. From Def. \ref{def:agg_int_bub}, at every proper interface $F \in \agints^+$, we can compute $\zeta_{F}(\v) \in \mathbb{R}$ such that:
\begin{equation} \label{eq:fac_val}
\zeta_{F}(\v) \int_F  \bsphih^F \cdot \n = \int_{F}  (\v - \pi_h^{SZ}(\v) )\cdot \n,
\end{equation}
and define $\bszetah(\v) = \sum_{F \in \agints^+} \bsphih^F \zeta_F(\v)$. Thus, taking $\pi_h^{qF}(\v) \doteq \pi_h^{SZ}(\v) + \bszetah(\v)$, one readily checks the equality in \eqref{eq:qua_For_pro}. Next, we prove the stability of the quasi-Fortin interpolant. We can bound $\zeta_F(\v)$, since $F \in \agints^+$, as follows. Let us represent with $A_F \in \meshag$ one of the two aggregates sharing $F$. The definition of the aggregate bubble in Def. \ref{def:agg_bub}, the extended Scott-Zhang approximability properties in \eqref{eq:sco_zha_int}, the inverse inequality \eqref{eq:tra_ine2}, and the Cauchy-Schwarz inequality yield:
\begin{align}
  \zeta_F & = \frac{\int_F  (\v - \pi_h^{SZ}(\v) ) \cdot \n }{\int_F  \bsphih^F\cdot \n}  \lesssim \frac{\int_F (\v - \pi_h^{SZ}(\v) )\cdot \n}{| F |} \lesssim \frac{ \normregl2{\v - \pi_h^{SZ}(\v)}{F}}{|F|^{\frac{1}{2}}}
  \label{eq:bou_fac_com1} \\  
&\lesssim  h^{-\frac{d-1}{2}} \normregl2{\v - \pi_h^{SZ}(\v)}{\partial A_F}
\lesssim h^{-\frac{d-1}{2}} ( h^{-\frac{1}{2}} \normregl2{\v - \pi_h(\v)}{A_F} + h^{\frac{1}{2}}\seminormregh1{\v - \pi_h(\v)}{A_F}) \lesssim  h^{-\frac{d-2}{2}} \normregh1{\v}{ \omega(A_F)}. \nonumber  
\end{align}
Using scaling arguments and the properties of the interface bubbles in Def. \ref{def:agg_int_bub} and \eqref{eq:bou_fac_com1}, and the fact that for any interior cell $K \in \meshin$,  the cardinality of the set $\{ A \in \meshag \ : \ K \subseteq \omega(A) \}$ is bounded independently of $h$, we get:
\begin{align}\label{eq:bou_fac_com2}
\tnorm{\bszetah(\v)}^2 = \sum_{F \in \agints^+} \normregh1{\zeta_F \bsphih^F}{AB_F}^2 \lesssim \sum_{F \in \agints^+} \zeta_F^2  h^{d-2}\| \bsphih^F \|^2_{L^\infty(AB_F)} \lesssim \sum_{F \in \agints^+} \normregh1{\v}{\omega(A_F)}^2 \lesssim \normregh1{\v}{\Dom}^2.
\end{align}
This result, combined with the stability and approximability of the Scott-Zhang projector and the triangle inequality, leads to the stability of the quasi-Fortin interpolant in \eqref{eq:qua_For_pro}:
\begin{align}\label{eq:sta_qua_For}
\tnorm{ \pi_h^{qF} (\v) + \bszetah(\v) } \lesssim \tnorm{ \pi_h^{qF}(\v)} + \tnorm{ \bszetah(\v)} \lesssim \normregh1{\v}{\Dom}.
\end{align}
It proves the lemma.
\end{proof}

\begin{lemma} \label{lm:bub_sta}
  Given an aggregate $A \in \meshag^+$, we consider the aggregate bubble function ${\phi}^A_h \in H^1_{{0}}(A)$ that satisfies the properties in Def. \ref{def:imp_agg_set}. For any $\ph \in \Qh$, the function
  $
\vph{\ph} \doteq  \sum_{A \in \meshag^+} {\phi}^A_h \h^2 \grad \ph
$
satisfies the following properties:
\begin{align}\label{eq:bub_sta}
\frac{1}{\beta_0'} \sum_{A \in \meshag^+} \normregl2{\ph - \pi_h^{-,0}(\ph)}{A}^2 \leq b_h(\bs{\varphi}_h(\ph),\ph), \qquad \tnorm{\vph{\ph}}^2 \leq \sum_{A \in \meshag^+} \normregl2{\ph - \pi_h^{-,0}(\ph)}{A}^2.
\end{align}
\end{lemma}
\begin{proof}
The fact that $\vph{\ph}$ vanishes on $\bou$, the definition of the norms in \eqref{eq:mes_nor}, and the inverse inequality \eqref{eq:inv_ine1}, yield the continuity bound in \eqref{eq:bub_sta}:
\begin{align}
  \tnorm{\vph{\ph}}^2 =  \normregh1{\vph{\ph}}{\Dom}^2 \lesssim \sum_{A \in \meshag^+} h^2 \| \grad \ph \|^2  = \sum_{A \in \meshag^+} h^2 \| \grad (\ph - \pi_h^{-,0}(\ph))\|^2 \lesssim \sum_{A \in \meshag^+} \normregl2{\ph - \pi_h^{-,0}(\ph)}{A}^2.
\end{align}
Next, we note that given a proper aggregate $A \subset \meshag^+$ and \ac{fe} function $\vh$, using scaling arguments, the first two properties in \eqref{eq:agg_bub_pro}, and the equivalence of norms in finite dimension, we have:
\begin{align} \label{eq:equ_nor}
C_- \ipreg{\vh}{\vh}{A} \leq
\ipreg{{\phi}^A_h \vh}{\vh}{A} \leq 
C_+ \ipreg{\vh}{\vh}{A},
\end{align}
for positive constants independent of $h$ and cut cell intersections. Thus, integrating by parts the first term in \eqref{eq:bh_for}, using the definition of $\vph{\ph}$ in the statement of Lem. \ref{lm:bub_sta}, the fact that aggregate bubbles vanish on aggregate boundaries (see Def. \ref{def:agg_bub}), and the equivalence of norms in \eqref{eq:equ_nor}, we obtain:
\begin{align}
 b_h(\vph{\ph},\ph) &=  \sum_{A \in \meshag^+} \ipreg{ \vph{\ph} }{\grad \ph}{A} = \sum_{A \in \meshag^+} h^2 \ipreg{{\phi}^A_h}{|\grad \ph|^2}{A} \gtrsim \sum_{A \in \meshag^+} h^2 \normregl2{ \grad \ph}{A}^2. \label{eq:bub_sta_bou2}
\end{align}
On the other hand, since $(\ph - \pi_h^{-,0}(\ph))|_A \in \Qh(A) \cap L_0^2(A)$, it holds from the Poincar\'e-Wirtinger inequality with a scaling argument:
\begin{align}\label{eq:bub_sta_bou3}
  \normregl2{\ph - \pi_h^{-,0}(\ph)}{A} \lesssim h \normregl2{\grad (\ph - \pi_h^{-,0}(\ph))}{A} = h \normregl2{\grad \ph}{A}.\end{align}
Combining \eqref{eq:bub_sta_bou2} and \eqref{eq:bub_sta_bou3}, we prove the lemma.
\end{proof}

In what follows, we will make use of the jump operator over facets:
\begin{align}
  \jump{p}(\x) =  \lim_{\epsilon \to 0^+} \left( p(\x + \epsilon \n) - p(\x-\epsilon \n) \right), \qquad \forall \x \in F, \ \ \forall F \in \agints,
\end{align}
where $\n$ is a normal to the facet.

\begin{lemma} \label{lm:inf_sup_con_pre}
Let us consider the mixed \ac{fe} space $\Vh \times \Qh$ for $\Qh \doteq \Pordhd{0} \cap L^2_0(\Dom)$. Then, for any $p_h \in \Pordhd{0} \cap L^2_0(\Dom)$, there exists a $\vh \in \Vh$ such that:
\begin{align}\label{eq:inf_sup_con_pre}
  \frac{1}{\beta_0}\normregl2{\ph}{\Dom}^2 \leq \bh{\vh}{\ph} +  \sum_{F \in \agints^-}h \normregl2{\jump{p}}{F}^2,  \qquad \tnorm{\vh} \leq \normregl2{\ph}{\Dom},
\end{align}
for a positive constant $\beta_0$.
\end{lemma}
\begin{proof}
Relying on the continuous inf-sup condition \eqref{eq:inf_sup_con}, for any $\ph \in \Pordhd{0} \cap L^2_0(\Dom)$ there exists a $\v \in  \H_{\bs{0}}^1(\Omega)$ such that: 
\begin{align}
 b_h(\v,p_h)= & -\ipreg{\div \v}{\ph}{\Dom}
  + \ipreg{ \v \cdot \n}{\ph}{\bou}  
   = - \sum_{F \in \agints} \ipreg{\v \cdot \n}{\jump{\ph}}{F} 
  \geq \frac{1}{\beta} \normregl2{ p_h }{\Dom}^2,  \qquad \normregh1{\v}{\Dom} \lesssim \normregl2{\ph}{\Dom}, \label{eq:con_inf_sup_con_pre}
\end{align}
where we have used integration by parts and added up the contributions from both cells sharing an interior interface. Using the properties of the quasi-Fortin interpolant in \eqref{eq:qua_For_pro}, after some algebraic manipulation, we obtain:
\begin{align}
  b(\pi_h^{qF}(\v),\ph) &= -\ipreg{\div \pi_h^{qF}(\v)}{\ph}{\Dom}
                          + \ipreg{\pi_h^{qF}(\v) \cdot \n}{\ph}{\bou}
                          = -\sum_{F \in \agints} \ipreg{\pi_h^{qF}(\v) \cdot \n}{\jump{\ph}}{F} \\
                        & = -\sum_{F \in \agints^+} \ipreg{\pi_h^{qF}(\v) \cdot \n}{\jump{\ph}}{F}
                          - \sum_{F \in \agints^-} \ipreg{\pi_h^{qF}(\v) \cdot \n}{\jump{\ph}}{F}\\
                        &= -\sum_{F \in \agints^+} \ipreg{\v \cdot \n}{\jump{\ph}}{F}
                          - \sum_{F \in \agints^-} \ipreg{\pi_h^{qF}(\v) \cdot \n}{\jump{\ph}}{F}\\
                        &= -\sum_{F \in \agints} \ipreg{\v \cdot \n}{\jump{\ph}}{F}
                          + \sum_{F \in \agints^-} \ipreg{(\v - \pi_h^{qF}(\v)) \cdot \n}{\jump{\ph}}{F}.\label{eq:inf_sup_con_pre_bou1}
\end{align}
We can bound the last term in \eqref{eq:inf_sup_con_pre_bou1} using the trace inequality \eqref{eq:tra_ine2}, the local Scott-Zhang interpolant error estimate in Th. \ref{th:app_sco_zha} for $\v \in H_{\bs{0}}^1(\Dom)$, the second bound in \eqref{eq:con_inf_sup_con_pre}, and Young's and Cauchy-Schwarz inequalities as follows:
\begin{align}
 \sum_{F \in \agints^-} \ipreg{(\v - \pi_h^{qF}(\v)) \cdot \n}{\jump{\ph}}{F} & \leq \sum_{F \in \agints^-} \normregl2{\v - \pi_h^{qF}(\v)}{F} \normregl2{\jump{\ph}}{F} \\
                                                        & \lesssim \sum_{F \in \agints^-} (h^{-\frac{1}{2}} \normregl2{\v - \pi_h^{qF}(\v)}{A_F} + h^{\frac{1}{2}} \normregh1{\v - \pi_h^{qF}(\v)}{A_F})  \normregl2{\jump{\ph}}{F} \\
                                                                                  & \lesssim  \sum_{F \in \agints^-} \normregh1{\v}{{\omega(A_F)}} h^{\frac{1}{2}} \normregl2{\jump{\ph}}{F} \lesssim \alpha \normregh1{\v}{\Dom}^2 + \frac{1}{\alpha} \sum_{F \in \agints^-} h \normregl2{\jump{\ph}}{F}^2 \\
  & \lesssim \alpha \normregl2{\ph}{\Dom}^2 + \frac{1}{\alpha}\sum_{F \in \agints^-} h \normregl2{\jump{\ph}}{F}^2,  \label{eq:inf_sup_con_pre_bou2}
\end{align}
for any $\alpha > 0$. Combining \eqref{eq:con_inf_sup_con_pre}, \eqref{eq:inf_sup_con_pre_bou1}, and \eqref{eq:inf_sup_con_pre_bou2} with $\alpha$ large enough, we readily get:
\begin{align}
  b(\pi_h^{qF}(\v),\ph) \geq \frac{1}{\beta}  \normregl2{\ph}{\Dom}^2 - C \alpha \normregl2{\ph}{\Dom}^2 - \frac{C}{\alpha}\sum_{F \in \agints^-} h \normregl2{\jump{\ph}}{F}^2 \geq \frac{1}{\beta_0}  \normregl2{\ph}{\Dom}^2 - \frac{C}{\alpha}\sum_{F \in \agints^-} h \normregl2{\jump{\ph}}{F}^2,  \label{eq:inf_sup_con_pre_bou3}
\end{align}
for $\beta_0 > 0$. It proves the lemma.
\end{proof}
Let us define the $L^2$ interpolant for extended discontinuous Lagrangian spaces as follows. Given $\ph \in \Pqhd$ and $0 \leq r < q$, $\pi_h^{-,r}(\ph) \in \Pqhd$ is such that 
\begin{align}
\ipreg{\pi_h^{-,r}(\ph)}{\qh}{A} = 
\ipreg{\ph}{\qh}{A},  \qquad \forall \qh \in \Pordhd{r}.
\end{align}
\begin{theorem} \label{th:inf_sup_hig_ord}
Let us assume that there exists a $q \in \mathbb{Z}^+$, $q \geq 2$,  such that $\Qbq(A) \subset \Vh(A)$ and that the mixed \ac{fe} space $\Vh \times \Pordhd{0} \cap L_0^2(\Dom)$ satisfies the inf-sup condition \eqref{eq:inf_sup_con_pre} in Lem. \ref{lm:inf_sup_con_pre}. Then, for any $p_h \in \Pordhd{q-1}$, there exists a $\vh \in \Vh$ such that:
\begin{align}\label{eq:inf_sup_hig_ord}
  \frac{1}{\beta_q}\normregl2{\ph}{\Dom}^2 \leq
  \bh{\vh}{\ph} +   \sum_{F \in \agints^-}h \normregl2{\jump{\ph}}{F}^2 -  \sum_{A \in \meshag^-} \normregl2{\ph + \pi_h^{-,0}(\ph)}{A}^2 , \qquad
  \tnorm{\vh} \leq \normregl2{\ph},
\end{align}
for a positive constant $\beta_q$.
\end{theorem}
\begin{proof}
  Let us decompose $b_h(\vh,\ph)$ as follows: 
\begin{align}
  b_h(\vh,\ph) &= b_h(\vh,\pi_h^{-,0}(\ph)) + b_h(\vh,\pi_h^{-,0}(\ph) - \ph). \label{eq:inf_sup_hig_ord_bou1}
\end{align}
Since $\Vh \times ( \Pordhd{0} \cap L_0^2(\Dom))$ is weakly inf-sup stable by the statement of the theorem, i.e., it satisfies \eqref{eq:inf_sup_con_pre},  there exists a function $\vh$ such that
\begin{align}\label{eq:inf_sup_q0}
\frac{1}{\beta_0}\normregl2{\pi_h^{-,0}(\ph)}{\Dom}^2 \leq
  \bh{\vh}{\pi_h^{-,0}(\ph)} +  \sum_{F \in \agints^-}h \normregl2{\jump{\pi_h^{-,0}(\ph)}}{F}^2, \qquad \tnorm{\vh} \leq \normregl2{\pi_h^{-,0}(\ph)}{\Dom}.
\end{align}
Using the trace inequality \eqref{eq:tra_ine}, the inverse inequality \eqref{eq:inv_ine1}, the stability of $\vh$ in the weak inf-sup condition \eqref{eq:inf_sup_q0}, and Young's and Cauchy-Schwarz inequalities, the first term in \eqref{eq:inf_sup_hig_ord_bou1} can be bounded as follows:
\begin{align}
  b_h(\vh,\pi_h^{-,0}(\ph) - \ph)  & \lesssim \normregh1{\vh}{\Dom} \normregl2{\ph - \pi_h^{-,0}(\ph)}{\Dom} + h^{-\frac{1}{2}} \normregl2{\vh}{\bou} h^{\frac{1}{2}} \normregl2{\ph - \pi_h^{-,0}(\ph)}{\bou}  \\ & \lesssim \tnorm{\vh}  \normregl2{\ph - \pi_h^{-,0}(\ph)}{\Dom} \\ & \lesssim \alpha \normregl2{ \pi_h^{-,0}(\ph) }{\Dom}^2  + \frac{1}{\alpha} \normregl2{\ph + \pi_h^{-,0}(\ph)}{\Dom}^2.  \label{eq:inf_sup_hig_ord_bou2}
\end{align}
Combining \eqref{eq:inf_sup_hig_ord_bou1}, \eqref{eq:inf_sup_q0}, and \eqref{eq:inf_sup_hig_ord_bou2} for $\alpha$ small enough, we get:
\begin{align}
  \frac{1}{\beta^*_0} \normregl2{\pi_h^{-,0}(\ph)}{\Dom}^2
  \leq b_h(\vh,\ph) + \sum_{F \in \agints^-}h \normregl2{\jump{\pi_h^{-,0}(\ph)}}{F}^2 + C \normregl2{\ph - \pi_h^{-,0}(\ph)}{\Dom}^2,
             \label{eq:inf_sup_hig_ord_bou3}
\end{align}
where $\beta_0^* > 0$. On the other hand,  we have that $\grad \ph \in \Pbordhd{q-2}$ and $\vph{\ph} \in \Qbq(A) \cap \H_0^1(A) \subset \Vh$ for any $A \in \agints^+$. Thus, combining the first inequality in \eqref{eq:bub_sta} from  Lem. \ref{lm:bub_sta} with \eqref{eq:inf_sup_hig_ord_bou3}, we obtain,
for an arbitrary positive constant $\alpha'$:
\begin{align}
   \bh{\vph{\ph} + \alpha' \vh}{\ph}   \geq &
  \frac{1}{\beta'_0}\sum_{A \in \meshag^+} \normregl2{\ph - \pi_h^{-,0}(\ph)}{A}^2 + \frac{\alpha'}{\beta^*_0}\normregl2{\pi_h^{-,0}(\ph)}{\Dom}^2 \\ &-  \alpha' \sum_{F \in \agints^-}h \normregl2{\jump{\pi_h^{-,0}(\ph)}}{F}^2- \alpha' C \normregl2{\ph - \pi_h^{-,0}(\ph)}{\Dom}^2 \\
  \geq & \frac{1}{\beta'_0}\sum_{A \in \meshag^+} \normregl2{\ph - \pi_h^{-,0}(\ph)}{A}^2 + \frac{\alpha'}{\beta^*_0}\normregl2{\pi_h^{-,0}(\ph)}{\Dom}^2 -  \alpha' \sum_{F \in \agints^-}h \normregl2{\jump{\pi_h^{-,0}(\ph)}}{F}^2 \\ & - \alpha' C \sum_{A \in \meshag^+} \normregl2{\ph - \pi_h^{-,0}(\ph)}{A}^2 - \alpha' C \sum_{A \in \meshag^-} \normregl2{\ph - \pi_h^{-,0}(\ph)}{A}^2 \\
  \geq & \frac{1 - \alpha' C \beta_0'}{\beta'_0}\sum_{A \in \meshag^+} \normregl2{\ph - \pi_h^{-,0}(\ph)}{A}^2 + \frac{\alpha'}{\beta^*_0}\normregl2{\pi_h^{-,0}(\ph)}{\Dom}^2 \\ &  -  \alpha' \sum_{F \in \agints^-}h \normregl2{\jump{\pi_h^{-,0}(\ph)}}{F}^2 - \alpha' C \sum_{A \in \meshag^-} \normregl2{\ph - \pi_h^{-,0}(\ph)}{A}^2 \\
  \geq & \frac{1 - \alpha' C\beta'_0}{\beta'_0}\sum_{A \in \meshag} \normregl2{\ph - \pi_h^{-,0}(\ph)}{A}^2 + \frac{\alpha'}{\beta^*_0}\normregl2{\pi_h^{-,0}(\ph)}{\Dom}^2 \\ &  -  \alpha' \sum_{F \in \agints^-}h \normregl2{\jump{\pi_h^{-,0}(\ph)}}{F}^2 - \frac{1 - 2\alpha' C \beta_0'}{\beta_0'} \sum_{A \in \meshag^-} \normregl2{\ph - \pi_h^{-,0}(\ph)}{A}^2. \label{eq:inf_sup_hig_ord_bou5}
\end{align}
Furthermore, using the fact that $\normregl2{\pi_h^{-,0}(\ph)}{\Dom} \leq \normregl2{\ph}{\Dom}$, the stability in \eqref{eq:bub_sta}, and the triangle inequality, we get:
\begin{align}
\tnorm{\vph{\ph} + \alpha' \vh}^2 \lesssim 
  \tnorm{\vph{\ph}}^2 + \tnorm{\alpha' \vh}^2 \leq  \sum_{A \in \meshag^+} \normregl2{\ph - \pi_h^{-,0}(\ph)}{A}^2
  + \normregl2{\alpha' \pi_h^{-,0}(\ph)}{\Dom}^2 \lesssim \normregl2{\ph}{\Dom}^2.\label{eq:inf_sup_hig_ord_bou6}
\end{align}
On the other hand, the trace inequality \eqref{eq:tra_ine2} and the triangle inequality yield:
\begin{align}
  \sum_{F \in \agints^-}h \normregl2{\jump{\pi_h^{-,0}(\ph)}}{F}^2 & \leq \sum_{F \in \agints^-}h \normregl2{\jump{\ph - \pi_h^{-,0}(\ph)}}{F}^2 + \sum_{F \in \agints^-}h \normregl2{\jump{\ph}}{F}^2 \\ & \lesssim \sum_{A \in \meshag} \normregl2{\ph - \pi_h^{-,0}(\ph)}{A}^2  + \sum_{F \in \agints^-}h \normregl2{\jump{\ph}}{F}^2. \label{eq:inf_sup_hig_ord_bou7} 
\end{align}
Bounds \eqref{eq:inf_sup_hig_ord_bou5}-\eqref{eq:inf_sup_hig_ord_bou7}  yield \eqref{eq:inf_sup_hig_ord} for $\alpha'$ small enough, after a proper scaling of $\vph{\ph} + \alpha' \vh$.  
\end{proof}

\begin{assumption}[Pressure stabilization] \label{ass:pre_sta}
  For a mixed \ac{fe} space $\Vh \times \Qh$, we consider a pressure stabilization that is positive semidefinite and holds: 
  \begin{align}
    \frac{1}{\gamma_j} j_h(\uh,\ph,\uh,\ph) &\geq  \sum_{A \in \meshag^-} \normregl2{\ph - \pi_h^{-,0}(\ph)}{A}^2 +  \sum_{F \in \agints^-}h \normregl2{\jump{\ph}}{F}^2 - \frac{\gamma_a}{2 \gamma_j} \tnorm{\uh}^2, \label{eq:pre_sta1} \\ j_h(\uh,p_h,\vh,q_h) & \leq \xi_j \tnorm{\uh,\ph} \tnorm{\vh,q_h},\label{eq:pre_sta2}
  \end{align}
  for any $(\uh,\ph), \ (\vh,\qh) \in \Vh \times \Qh$.
\end{assumption}

\begin{theorem} \label{th:fin_inf_sup}
Let us assume that the mixed \ac{fe} space  $\Vh \times \Qh$ satisfies the inf-sup condition \eqref{eq:inf_sup_hig_ord} and that the pressure stabilization $j_h$ satisfies Ass. \ref{ass:pre_sta}. It holds:
\begin{align}\label{eq:fin_inf_sup}
  \frac{1}{\beta_d} \tnorm{\uh,\ph} \leq
   \sup_{(\vh,\qh) \in \Vh \times \Qh} \frac{A_h(\uh, \ph, \vh, \qh)}{\tnorm{\vh,\qh}},
\end{align}
for a positive constant $\beta_d$.
\end{theorem}
\begin{proof}
  First, we take as test function $(\uh,-\ph)$. Using the first inequality in \eqref{eq:ah_bh_sta_con}, we get: 
\begin{align}\label{eq:fin_inf_sup_bou1}
  A_h(\uh, \ph, \uh, \ph) & =  a_h(\uh,\uh) + j_h(\uh,\ph,\uh,\ph)  \geq {\gamma_a} \tnorm{\uh}^2 + j_h(\uh,\ph,\uh,\ph).
\end{align}
Next, taking as test function $(\vh,0)$, where $\vh$ satisfies the weak inf-sup \eqref{eq:inf_sup_hig_ord} in Th. \ref{th:inf_sup_hig_ord}, we get:
\begin{align}
  A_h(\uh, \ph, \vh, 0)  = &  a_h(\uh,\vh) + b_h(\vh,\ph) - j_h(\uh,\ph,\vh,0) \\
                          \geq & \frac{1}{\beta_q}\normregl2{\ph}{\Dom}^2 -  \sum_{F \in \agints^-}h \normregl2{\jump{\ph}}{F}^2 -  \sum_{A \in \meshag^-} \normregl2{\ph - \pi_h^{-,0}\ph}{A}^2 \\ & + a_h(\uh,\vh) - j_h(\uh,\ph,\vh,0).   \label{eq:fin_inf_sup_bou2}   
\end{align}
On one side, the second inequality in \eqref{eq:ah_bh_sta_con} together with Young's and Cauchy-Schwarz inequalities yield:
\begin{align}\label{eq:fin_inf_sup_bou21}
   a_h(\uh,\vh) \leq \frac{4\xi_a^2}{\alpha} \tnorm{\uh}^2 +  \alpha \tnorm{\vh}^2 \leq \frac{4\xi_a^2}{\alpha}  \tnorm{\uh}^2 + \alpha \normregl2{\ph}{\Dom}^2,                       
\end{align}
for an arbitrary constant $\alpha$. 
On the other side, using the fact that the pressure stabilization is positive semidefinite, Cauchy-Schwarz and Young's inequalities, the continuity in \eqref{eq:pre_sta2}, and the stability for $\vh$ in \eqref{eq:inf_sup_hig_ord}, we get:
\begin{align}
j_h(\uh,\ph,\vh,0) \leq & \frac{4}{\alpha} j_h(\uh,\ph,\uh,\ph) + \alpha j_h(\vh,0,\vh,0)  \nonumber \\ 
\leq & \frac{4}{\alpha} j_h(\uh,\ph,\uh,\ph) + \alpha \xi_j \tnorm{\vh}^2 \leq \frac{4}{\alpha} j_h(\uh,\ph,\uh,\ph) + \alpha \xi_j\normregl2{\ph}{\Dom}^2. \label{eq:fin_inf_sup_bou22}
\end{align}
As a result, combining \eqref{eq:fin_inf_sup_bou2}-\eqref{eq:fin_inf_sup_bou22}, and taking $\alpha$ small enough, we obtain:
\begin{align}
  A_h(\uh, \ph, \vh, 0)  \gtrsim & \normregl2{\ph}{\Dom}^2 -  \sum_{F \in \agints^-}h \normregl2{\jump{\ph}}{F}^2 -  \sum_{A \in \meshag^-} \normregl2{\ph - \pi_h^{-,0}(\ph)}{A}^2 \nonumber \\ 
& - j_h(\uh,\ph,\uh,\ph) -  \tnorm{\uh}^2.  \label{eq:fin_inf_sup_bou23}                     
\end{align}
By taking $(\uh + \alpha' \vh,\ph)$ as a test function with $\alpha'$ small enough, using \eqref{eq:fin_inf_sup_bou22}, \eqref{eq:fin_inf_sup_bou23}, and the assumption over the pressure stability in \eqref{eq:pre_sta1}, we finally get:
\begin{align}\label{eq:fin_inf_sup_bou3}
  A_h(\uh, \ph, \uh + \alpha' \vh, \ph)  \gtrsim & \tnorm{\uh}^2 + \alpha' C \normregl2{\ph}{\Dom}^2 + j_h(\uh,\ph,\uh,\ph) \\ & - \alpha' C \sum_{F \in \agints^-}h \normregl2{\jump{\ph}}{F}^2 - \alpha' C \sum_{A \in \meshag^-} \normregl2{\ph - \pi_h^{-,0}(\ph)}{A}^2 \\ \gtrsim &\tnorm{\uh}^2 + \normregl2{\ph}{\Dom}^2 + j_h(\uh,\ph,\uh,\ph).  
\end{align}
On the other hand, the stability for $\vh$ in \eqref{eq:inf_sup_hig_ord} and the triangle inequality yield:
\begin{align}\label{eq:fin_inf_sup_bou4}
\tnorm{ \u_h+\alpha' \vh, \ph } \lesssim \tnorm{ \u_h, \ph } + \tnorm{ \alpha' \vh } \lesssim \tnorm{ \u_h, \ph } + \normregl2{\alpha' \ph }{\Dom}  \lesssim \tnorm{ \u_h, \ph }.
\end{align}
It proves the theorem.
\end{proof}

\subsection{Stable mixed \acp{fe} and pressure stabilization} \label{ssec:sta_mix_fe_pre_sta}
We propose below two different algorithms that satisfy  Ass. \ref{ass:pre_sta} and thus, the stability results in Th. \ref{th:fin_inf_sup}.
\begin{method} \label{alg:mix_lag}
  We consider a hex mesh, the velocity space $\Vh \doteq \Qbqh$, and the pressure space $\Qh \doteq \Pordhd{q-1}$ for an integer $ q \geq 2$. On the other hand, for any facet $F \in \agints$, and the two aggregates $A_F, B_F \in \meshag$, we include the facet in the subset of facets $\agints^*$ if $A_F$ or $B_F$ belong to the set $\meshag \setminus \meshag \cap \meshact$. On the other hand, we define the set of aggregates to be stabilized as $\meshag^* \doteq \meshag \setminus \meshag \cap \meshact$. The pressure stabilization term is taken as:
  \begin{align}
    j_h(\uh,\ph,\vh,\qh) & \doteq \sum_{F \in \agints^*} \tau_{j1} h \ipreg{\jump{\ph}}{\jump{\qh}}{F} +  \sum_{A \in \meshag^*} \tau_{j2} h^2 \ipreg{-\Delta \uh + \grad \ph}{-\Delta \vh + \grad \qh}{A}, \\
    g_h(\f,\vh) & \doteq \sum_{A \in \meshag^*} h^2 \ipreg{\f}{-\Delta \vh + \grad \qh}{A}, \label{eq:mix_lag_sta}
    \end{align}
for positive algorithmic constants $\tau_{j1}$ and $\tau_{j2}$.
\end{method} 

The \ac{agfe} space thus relies on the popular \ac{fe} space $\Qbqh \times \Pordhd{q-1}$ for the interior cells. The velocity field is extended to cut cells by the standard extension operator in Sect. \ref{ssec:ser_ext}, and the discontinuous pressure field is extended by the standard (discontinuous) one. This choice has been motivated by the proof of the abstract discrete inf-sup condition.
\begin{theorem}\label{th:mix_lag}
The method proposed in Alg.  \ref{alg:mix_lag} has a pressure stabilization term that satisfies Ass. \ref{ass:pre_sta} and thus, it satisfies Th. \ref{th:fin_inf_sup}. As a result, the discrete problem \eqref{eq:dis_sto} is well-posed for $\f \in \L^2(\Omega)$. 
  \end{theorem}
  \begin{proof}
    It is clear that for aggregates that are interior cells, i.e., in $\meshag \cap \meshact$, the extension of their quadratic bubble functions is the zero extension outside the cell. Thus, for these cells, there exists a cell bubble satisfying the requirements in Def. \ref{def:agg_bub}. Analogously, for facets that are being shared by two aggregates that are interior cells, the extension of the corresponding quadratic bubble function is also the zero extension. These facet bubbles satisfy the requirements in Def. \ref{def:agg_int_bub}. Thus, $\agints^- \subset \agints^*$ and $\meshag^- \subset \meshag^*$.

    As required in Ass. \ref{ass:pre_sta}, the pressure stabilization is positive semidefinite. In order to prove that \eqref{eq:pre_sta1} holds, we use the following inequality. Given three functions $v,p$ in a Hilbert space $X$ and $u$ in a Banach space $Y$, defining $\gamma^{\frac{1}{2}} \doteq \frac{\| u \|_{Y}}{\| v \|_{X}}$, we have, using Young's inequality for an arbitrary constant $\alpha > 1$:
    \begin{align}
      2 \| p+v \|_{X}^2 
      & = \| p+v \|_{X}^2 + \| p \|_{X}^2 + \| v \|_{X}^2 - 2(p,v)_{X}   \\ &\geq
       \| p+v \|_{X}^2 + \left( 1-\frac{1}{\alpha} \right) \| p \|_{X}^2 + (\alpha-1) \| v \|_{X}^2 \\ & = \| p+v \|_{X}^2 + \left( 1-\frac{1}{\alpha} \right)\| p \|_{X}^2 - \frac{\alpha-1}{\gamma} \| u \|_{Y}^2.
    \end{align} 
    Taking $\alpha = 1 + \frac{\gamma}{2} > 1$, we obtain
    \begin{align}
      2 \| p+v \|_{X}^2 & \geq  \| p+v \|_{X}^2 + \frac{1}{1+ \frac{2}{\gamma}}\| p \|_{X}^2 - \| u \|_{Y}^2.
    \end{align}
    Let us consider $X = L^2(A)$, $v = - h \Delta \uh$, $p = h \grad \ph$, $Y= \H^1(A)$, and $u = \omega^{\frac{1}{2}} \uh$, for $A \in \meshag^-$ and an arbitrary positive constant $\omega$. Using the inverse inequality \eqref{eq:inv_ine1}, we have that $ h \normregl2{- \Delta \uh }{A} \leq C  \normregh1{\uh}{A}$, thus $\gamma \geq C^{-2} \omega^{-1}$. The previous bound leads to:
    \begin{align}\label{eq:mix_lag_bou1}
      \sum_{A \in \meshag^*} h^2 \normregl2{-\Delta \uh + \grad \ph}{A}^2 \geq C \sum_{A \in \meshag^*} h^2 \normregl2{\grad \ph}{\Dom}^2 - \sum_{A \in \meshag^*} \frac{\omega}{2} \normregh1{\uh}{A}^2.
    \end{align}
    The Poincar\'e-Wirtinger inequality with a scaling argument yields:
    \begin{align}\label{eq:mix_lag_bou2}
    \| \ph - \pi_h^{-,0}(\ph) \|_A \lesssim h \normregl2{\grad (\ph - \pi_h^{-,0}(\ph)}{\Dom} =  h \normregl2{\grad \ph}{\Dom}.
    \end{align}
    Combining \eqref{eq:mix_lag_bou1} and \eqref{eq:mix_lag_bou2}, and adjusting $\omega$ accordingly, we find
 \begin{align}\label{eq:mix_lag_sta2}
j_h(\uh,\ph,\uh,\ph) \geq  \gamma_j \sum_{A \in \meshag^*} \| \ph - \pi_h^{-,0}(\ph) \|_A^2 + \gamma_j \sum_{F \in \agints^*} h \normreg{\jump{\ph}}{F}^2 - \frac{\gamma_a}{2 \gamma_j} \normregh1{\uh}{\Dom}^2,
 \end{align}
 for a positive constant $\gamma_j$. Thus, the stabilization term satisfies Ass. \ref{ass:pre_sta}. Its continuity  in \eqref{eq:pre_sta2} is obtained from the trace inequalities \eqref{eq:tra_ine}-\eqref{eq:tra_ine2} and the inverse inequality \eqref{eq:inv_ine1}. This result, together with \eqref{eq:fin_inf_sup}, proves the well-posedness of the discrete operator. Furthermore, for $\f \in L^2_0(\Omega)$, we can easily prove that $g_h(\f,\vh) \leq \xi_g \normregl2{\f}{\Dom} \tnorm{\vh}$. \end{proof}
\begin{method} \label{alg:mix_ser}
  We consider a hex mesh, the velocity space $\Vh \doteq \Qbqhm$, and the pressure space $\Qh \doteq \Pordhd{q-1}$ for an integer $2 \leq q \leq 2d-2$. On the other hand, for any facet $F \in \agints$, and the two aggregates $A_F, B_F \in \meshag$, we include the facet in the subset of facets $\agints^*$ if $A_F$ or $B_F$ belong to the set $\meshag \setminus \meshag \cap \meshact$. In 3D, if $q \leq 2d-3$, $F \in \agints^*$ can be restricted further, by considering only those before that also satisfy that their corresponding owner interior cells $K_{A_F}$ and $K_{B_F}$ do not share a \ac{fe} facet, i.e., $|K_{A_F} \cap K_{B_F}| = \emptyset$ in $d-1$ sense. The pressure stabilization term is taken as:
  \begin{align}
j_h(\ph,\qh) \doteq \sum_{F \in \agints^*} \tau_{j1} h \ipreg{\jump{\ph}}{\jump{\qh}}{F}, \label{eq:mix_ser_sta}
  \end{align}
for a positive algorithmic constant $\tau_{j1}$, whereas $g_h(\f,\vh) \doteq 0$. 
  \end{method}
The \ac{agfe} space again relies on $\Qbqhm \times \Pordhd{q-1}$ for the interior cells. The velocity field is extended to cut cells by the serendipity extension operator in Sect. \ref{ssec:ser_ext}, and the discontinuous pressure field is extended by the standard (discontinuous) one. This choice has also been motivated by the proof of the abstract discrete inf-sup condition.
  \begin{theorem}\label{th:mix_ser}
The method proposed in Alg. \ref{alg:mix_ser} has a pressure stabilization term that satisfies Ass. \ref{ass:pre_sta} and thus, it satisfies Th. \ref{th:fin_inf_sup}. As a result, the discrete problem \eqref{eq:dis_sto} is well-posed for $\f \in \H^{-1}(\Omega)$. 
  \end{theorem}
  \begin{proof}
    First, we note that for the serendipity \ac{fe} up to order $2d-2$, a unisolvent set of  \acp{dof} are nodal values on the cell boundary only, and thus, zero for bubble functions (see \cite{Arnold2011} for more details). Thus, the serendipity extension of the quadratic bubble function of the owner cell of the aggregate is zero in $\Dom \setminus \Domin$ for $q \leq 2d-2$. Thus, all the aggregates are proper, and no cell interior stabilization is needed. On the other hand, for facets in $\agints \setminus \agints^*$, i.e., facets between interior cells that have not been aggregated to other cut cells, the quadratic facet function belongs to the space $\Vh$ (we note that it leads to a zero extension outside of the two interior cells that share the facet), and satisfies the requirements in Def. \ref{def:agg_int_bub}. Thus, the only stabilization that is needed is on the interface between aggregates that are not simply interior cells. In 3D, since serendipity \ac{fe} up to order $2d-3$ do not include the \acp{dof} corresponding to the quadratic facet bubbles, i.e., the facet bubbles of interior cells that are the owner of an aggregate are extended by zero. Thus, the subset $ \agints^*$ can be restricted as stated in the definition of the algorithm. Thus, $\agints^- \subset \agints^*$ and $\meshag^- = \emptyset$. It is obvious to check that
    \begin{align}
      j_h(\uh,\ph,\uh,\ph) =  \sum_{F \in \agints^*} \tau_{j1} h \normregl2{\jump{\ph}}{F}^2.
    \end{align}
    We can readily check that the stabilization term satisfies \eqref{eq:pre_sta1}. The continuity result in \eqref{eq:pre_sta2} is readily obtained from the trace inequalities \eqref{eq:tra_ine}-\eqref{eq:tra_ine2}. As a result, the stabilization term satisfies Ass. \ref{ass:pre_sta}. This result, together with \eqref{eq:fin_inf_sup}, proves the theorem.
  \end{proof}

  \subsection{\emph{A priori} error estimates} \label{ssec:a_pri_err_est}

At this point, we have already checked that Algs. \ref{alg:mix_lag} and \ref{alg:mix_ser} are well-posed. Next, we want to prove \emph{a priori} error estimates for these algorithms. The proof of these results is fairly straightforward, since the pressure stabilization terms are consistent for pressure fields in $H^1(\Omega)$. As usual, Galerkin orthogonality, the stability in Ths. \ref{th:mix_lag} and \ref{th:mix_ser}, and the approximability properties in Th. \ref{th:app_sco_zha}, lead to the desired results.
  
Let us note that the jump stabilization in \eqref{eq:mix_ser_sta} (also in \eqref{eq:mix_lag_sta}) can be modified by integrating not only on (potentially) cut facets $F \in \agints$ but in the corresponding whole facets. Such modification does provide more stabilization and does not affect the consistency of the method in the error analysis of Th. \ref{th:a_pri_est} below.

  \begin{theorem}\label{th:a_pri_est}
    Let us assume that the solution $(\u,p)$ of the Stokes problem \eqref{eq:wea_for} belongs to $\H^{k+1}(\Omega) \times H^{k}(\Omega)$ for some $k \geq 1$. Then, the discrete solution $(\uh,\ph) \in \vh \times \Qh$ in Algs. \ref{alg:mix_lag} and \ref{alg:mix_ser} satisfy the following \emph{a priori} error estimate: 
\begin{align}
\tnorm{\u - \uh, p - \ph} \lesssim h^{k} \| \u \|_{H^{k+1}(\Omega)} +h^{k} \| p \|_{H^{k}(\Omega)}. 
\end{align}
\end{theorem}
\begin{proof}
  First, let us note that the bilinear form $A_h$ in  Algs. \ref{alg:mix_lag} and \ref{alg:mix_ser} is consistent. Since $p \in H^1(\Omega)$, the pressure jump stabilization vanishes. It is obvious to check that the interior residual-based stabilization vanishes too. Let us consider the extended Scott-Zhang projector for every component of the velocity $\pi_h^{SZ}(\u)$ and for the pressure $\pi_h^{SZ}(p)$. The Galerkin orthogonality and the continuity of $A_h$ readily yield:
\begin{align}
  \Ah{\uh - \pisz(\u)}{\ph - \pisz(p)}{\vh}{\qh} & = \Ah{\u - \pisz(\u)}{p-\pisz(p)}{\vh}{\qh} \\
  & \leq \xi_A \tnorm{\u - \pisz(\u), p - \pisz{p}} \tnorm{\vh , \ph}.
\end{align}
Taking as test function the $(\vh,\qh)$ for which the global inf-sup condition in Th. \ref{th:fin_inf_sup} is satisfied, we readily get:
\begin{align}
 \tnorm{\uh - \pisz(\u), \ph - \pisz(p)}^2 & \leq  \beta_d \Ah{\uh - \pisz(\u)}{\ph - \pisz(p)}{\vh}{\qh} \\ & \leq \beta_d  \xi_A \tnorm{\u - \pisz(\u), p - \pisz(p)} 
 \tnorm{\uh - \pisz(\u), \ph - \pisz(p)}.
\end{align}
Finally, the approximability properties of the extended Scott-Zhang projector in Th. \ref{th:app_sco_zha} yields:
\begin{align}
\tnorm{\u - \pisz(\u), p - \pisz{p}}^2 & = \normregl2{\grad (\u - \pisz(\u))}{\Dom}^2 +  \normregl2{h^{-\frac{1}{2}} (\u - \pisz(\u)) }{\bou}^2 + \normregl2{p- \pisz{p}}{\Dom}^2 \\ & \lesssim
h^{2k} \| \u \|^2_{H^{k+1}(\Omega)} + h^{2k} \| p \|^2_{H^{k}(\Omega)}. 
\end{align}
It proves the theorem.
\end{proof}

\subsection{Condition number bounds} \label{ssec:cond_nums}

It is well-known that extended \ac{fe} spaces without aggregation lead to arbitrary ill-conditioned systems, due to the small cut cell problem, i.e., when the ratio $\eta_K$ tends to zero (see \cite{de_prenter_condition_2017}  for details). Thus, arbitrarily high condition numbers are expected in practice since the position of the interface cannot be controlled and the value  $\eta_\cell$ can be arbitrarily close to
zero. It has motivated the \ac{agfem} in \cite{Badia2017}. We prove in the following theorem that the \ac{agfem} proposed herein for the Stokes problem lead to the same condition number bounds as for body-fitted methods, i.e., they do not depend on the cut cell intersection. We represent with $| \cdot |_{\ell^2}$ the Euclidean norm of vectors and matrices.

\begin{theorem}
The condition number of the matrices that arise from Algs. \ref{alg:mix_lag} and \ref{alg:mix_ser}, i.e., $
\kappa(A_h) \doteq | A_h |_{\ell^2} | A_h^{-1} |_{\ell^2}$, satisfies $\kappa(A_h) \leq C_\kappa h^{-2}$, for a positive constant $C_\kappa$.
\end{theorem}
\begin{proof}
First, we note that $\uh \in \Vh$ can be stated in terms of a global basis of \ac{fe} shape functions as $\sum_{a=1}^{N_{u}} U_{i} \bs{\phi}_u^{a}$. We define the Cartesian norm for the vector of \ac{dof} values of $\uh$ as $|\uh|_{\ell^2}$. We proceed analogously for the pressure, e.g., $\ph = \sum_{a=1}^{N_{p}} P_{i} {\phi}_p^{a} \in \Qh$; we note that the pressure space has dimension $N_p-1$ due to the zero mean restriction, i.e., $\Qh \subset L^2_0(\Omega)$. Let us represent velocity-pressure functions in $\Vh \times \Qh$ with bold capital Greek letters. Given $\bsphi \doteq (\uh,\ph)$, we define $| \bsphi |_{\ell^2}^2 \doteq | \uh |_{\ell^2}^2 + | \ph |_{\ell^2}^2$. For any velocity component and pressure, we have from the fact that the eigenvalues of the local mass matrix in every interior cell are bounded (see, e.g., \cite{elman_finite_2005}), that $
C_m^- h^{d}   | U |_{\ell^2}^2 \leq  \normregl2{\ush}{\Domin}^2 \leq C_m^+ h^{d}   | U |_{\ell^2}^2$. This result, combined with the stability of the extension operator in Lem. \ref{lm:sta_agg_ext} yields:
\begin{align}
C_M^-  h^d  | \ush |_{\ell^2}^2  \leq  \normregl2{\ush}{\Dom}^2 \leq  C_M^+ h^{d}   | \ush |_{\ell^2}^2.\label{eq:nor_equ_l2}
\end{align}
Now, we can bound the following velocity norm using the inverse inequality \eqref{eq:inv_ine1}, the trace inequality \eqref{eq:tra_ine}, and the norm relation in \eqref{eq:nor_equ_l2}, as follows:
\begin{align}
  \tnorm{ \uh }^2 = \normregl2{\grad \uh}{\Dom}^2 +  \normregl2{h^{-\frac{1}{2}} \uh }{\bou}^2
  \lesssim h^{-2} \normregl2{ \uh}{\Dom}^2 \lesssim h^{d-2} | \uh |_{\ell^2}^2.
\label{eq:nor_bou_l2_vel}
\end{align}
Thus, we have $
  \tnorm{ \bsphi }^2 \lesssim   h^{d-2} | \bsphi |_{\ell^2}^2
$ for any $\bsphi \in \Vh \times \Qh$. The Friedrichs  inequality  and \eqref{eq:nor_equ_l2} yield $| \uh |_{\ell^2}^2 \leq C(\Omega) h^{-d} \normregl2{\uh}{\Dom}^2 \lesssim  C(\Omega) h^{-d}\tnorm{\uh}^2$. As a result: 
\begin{align}
C(\Omega)^{-1} h^{d}| \uh |_{\ell^2}^2 \lesssim  \tnorm{\uh}^2 \lesssim h^{d-2}\tnorm{\uh}^2.
\label{eq:nor_bou_l2_h}
\end{align}

We can bound the norm of $A_h$ by using its continuity (from the continuity results in \eqref{eq:ah_bh_sta_con} and \eqref{eq:pre_sta2}) and the norm equivalence in \eqref{eq:nor_bou_l2_h} as follows:
\begin{align}
| A_h |_{\ell^2}  = \max_{\bsphi \in \Vh \times \Qh} \max_{\bspsi \in \Vh \times \Qh} \frac{A_h(\bsphi,\bspsi)}{|\bsphi|_{\ell^2} |\bspsi|_{\ell^2}} \leq \xi_A \frac{\tnorm{\bsphi}\tnorm{\bspsi}}{|\bsphi|_{\ell^2} |\bspsi|_{\ell^2}} \lesssim  h^{d-2}.
\label{eq:nor_ope_l2}
\end{align}
Making abuse of notation, we use $A_h \bsphi \doteq A_h(\bsphi,\cdot)$. Next, we provide a lower bound for the norm of the operator $A_h \bsphi$, for some $\bsphi \in \Qh$. Using the inf-sup condition in Th. \ref{th:fin_inf_sup} and the norm equivalence in \eqref{eq:nor_bou_l2_h}, we obtain:
\begin{align}
| A_h \bsphi |_{\ell^2}  = \max_{\bspsi \in \Vh \times \Qh} \frac{A_h(\bsphi,\bspsi)}{|\bspsi|_{\ell^2}} =  \max_{\bspsi \in \Vh \times \Qh} \frac{A_h(\bsphi,\bspsi)}{\tnorm{\bspsi}}\frac{\tnorm{\bspsi}}{|\bspsi|_{\ell^2}} \geq \beta_d \tnorm{\bsphi} \min_{\bspsi \in \Vh \times \Qh} \frac{\tnorm{\bspsi}}{|\bspsi|_{\ell^2}}.\label{eq:con_num_bou1}
\end{align}
Combining \eqref{eq:con_num_bou1} and the lower bound in \eqref{eq:nor_bou_l2_h}, we get 
$| A_h \bsphi |_{\ell^2}  \gtrsim h^{d}| \bsphi |_{\ell^2}.
$ 
 Taking $ \bsphi = A_h^{-1} \bspsi$, we readily obtain  
$|  \bspsi |_{\ell^2}  \gtrsim h^{d} |A_h^{-1} \bspsi |_{\ell^2}
$.  Thus, $|A_h^{-1}|_{\ell^2} \lesssim h^{-d}$, which, together with \eqref{eq:nor_ope_l2}, proves the theorem.
\end{proof}

\section{Numerical experiments} \label{sec:num_exp}

The main purpose of this section is to evaluate the
performance of the \ac{agfe} spaces in several different
scenarios. We start with a convergence test (cf. Sect.~\ref{sec:conv-test}),
where we numerically validate the 	\emph{a priori} error estimates of Sect.~\ref{ssec:a_pri_err_est}
 and the condition number bounds of Sect.~\ref{ssec:cond_nums}.
Next, we consider a moving domain test (cf.
Sect.~\ref{sec:moving-dom}) in order to check the robustness  of the methods with respect to
small cuts. Finally, we provide the numerical solution of two realistic problems (cf.
Sect.~\ref{sec:misc-3d}) in order to illustrate the ability of the 
\ac{agfem} to deal with complex geometrical data.

\subsection{Setup}

In all cases, we solve the Stokes problem \eqref{eq:sto_eqs} using Galerkin
approximations with conforming Lagrangian \ac{fe} spaces as indicated in Sect.
\ref{sec:app_sto}. We consider both \ac{agfe} spaces and conventional ones in
order to evaluate the benefits of using cell aggregation.  For the conventional
(un-aggregated) case, we use $\Qqth(\meshact)$, and $\Pqthd(\meshact)$ spaces
for the approximation of velocities and pressures, respectively (e.g., in 3D,
hexahedral elements with continuous  piecewise triquadratic shape functions for
the velocity, and discontinuous piecewise linear shape functions for the
pressure). For the aggregated case, we consider the space $\Qqthm$ for
velocities (i.e., the aggregated version of $\Qqth(\meshact)$ using the
serendipity extension in the constraint definition as indicated in
Sect.~\ref{ssec:ser_ext}), whereas for pressures we use the aggregated
counterpart of $\Pqthd$. In order to fulfill inf-sup stability, we use the
facet-based stabilization given in Algorithm~\ref{alg:mix_ser} for the
aggregated spaces with $\tau_{j1}=0.01$ (the value that minimized the error for
a simple test and a set of possible constants).  The results for the usual
(un-aggregated) spaces are labeled as \emph{standard} 
throughout the numerical examples, whereas results using cell aggregation are
labeled as \emph{aggregated}. 
considered is provided in Table \ref{table:methods}.

\begin{table}[!htp]
\begin{tabular}{lp{0.75\textwidth}}
\toprule
\bf Name & \bf Description \\
\midrule
Standard & ($\Qqth(\meshact)$,$\Pqthd(\meshact)$) elements without cell aggregation.\\
Aggregated & ($\Qqthm$,$\Pqthd$) elements with cell aggregation using the
serendipity extension (cf. Sect. \ref{ssec:ser_ext}) for velocity components.\\
\bottomrule
\end{tabular}
\vspace*{1em}
\caption{\ac{fe} interpolations used in the experiments.}
\label{table:methods}
\end{table} 

The algorithms subject of study were coded using the tools provided by the object-oriented HPC code
\FEMPAR{} \cite{badia_fempar:_2017}. The underlying systems of linear equations
are solved by means of a robust sparse direct solver from the  MKL PARDISO package
\cite{_intel_????} specially designed for symmetric indefinite matrices (to which \FEMPAR{} provides
appropriate interfaces). The
condition number estimates provided below are computed outside \FEMPAR{} using
the MATLAB function {\tt condest}.\footnote{MATLAB is a trademark of THE
MATHWORKS INC.} {Numerical integration is based on local body-fitted triangulations of cut cells into triangles (in 2D) or tetrahedra (in 3D), where standard quadrature rules can be applied. The local triangulation of a cut cell is obtained by \FEMPAR{} from its nodal coordinates and the intersection points of cell edges  with the unfitted boundary via the Delaunay method available in the \texttt{QHULL} library \cite{_qhull_????,bradford_quickhull_1996}. Note that these sub-meshes are used only for integration
purposes and are completely independent from one cut cell to another (see \cite{badia_robust_2017} for details).}

\newcommand{\secttitle}{Convergence test}
\subsection{\secttitle}
\label{sec:conv-test}

We consider the Stokes problems defined in the 2D and 3D domains shown in
Fig.~\ref{fig:conv-geometries}.  The 2D domain (cf. Fig.
\ref{fig:conv-geometries-2d}) is a circular cavity defined as the set
difference of the unit square $[0,1]^2$ and the circle of radius $R=0.3$ and
center $C=(0.5,0.5)$.  The 3D domain is a complex-shaped cavity defined as the
set difference of the unit cube $[0,1]^3$ and a 3D body whose shape reminds the
one of a popcorn flake (cf.  Figs.~\ref{fig:conv-geometries-3d-1}
and~\ref{fig:conv-geometries-3d-2}). This ``popcorn-flake'' geometry is often
used in the literature to study the performance of unfitted \ac{fe} methods
(see, e.g., \cite{burman_cutfem:_2015}). The popcorn flake geometry considered
here is obtained by taking the one defined in \cite{burman_cutfem:_2015},
scaling it by a factor of $0.5$ and translating it a value of $0.5$ in each
direction such that the body fits in the unit cube $[0,1]^3$.  We consider
Dirichlet boundary conditions on the interior walls of the cavities, whereas
Neumann conditions are imposed on the facets of the unit square and unit cube
(see Fig.~\ref{fig:conv-geometries}). Dirichlet boundary conditions are imposed
using Nitsche's method as discussed in Sect. \ref{sec:app_sto}.

\begin{figure}[ht!]
  \centering

  \begin{subfigure}{0.9\textwidth}
    \tikz{\fill[mygray]  (0,0) rectangle (2.0em,1.0em);} Physical domain $\Omega$ 
    \hspace{1em}
    \tikz{\fill[myblue]  (0,0) rectangle (2.0em,1.0em);} Dirichlet boundary $\Gamma_\mathrm{D}$
    \hspace{1em}
    \tikz{\fill[mygreen]  (0,0) rectangle (2.0em,1.0em);} Neumann boundary $\Gamma_\mathrm{N}$
  \end{subfigure}

  \begin{subfigure}[b]{0.32\textwidth}
    \includegraphics[width=0.9\textwidth]{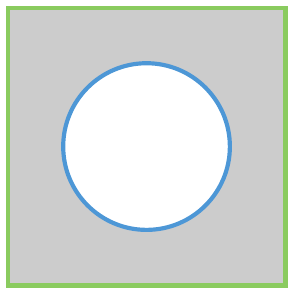}
    \vspace{1em}
    \caption{2D case.}
    \label{fig:conv-geometries-2d}
  \end{subfigure}
  \begin{subfigure}[b]{0.32\textwidth}
    \includegraphics[width=0.99\textwidth]{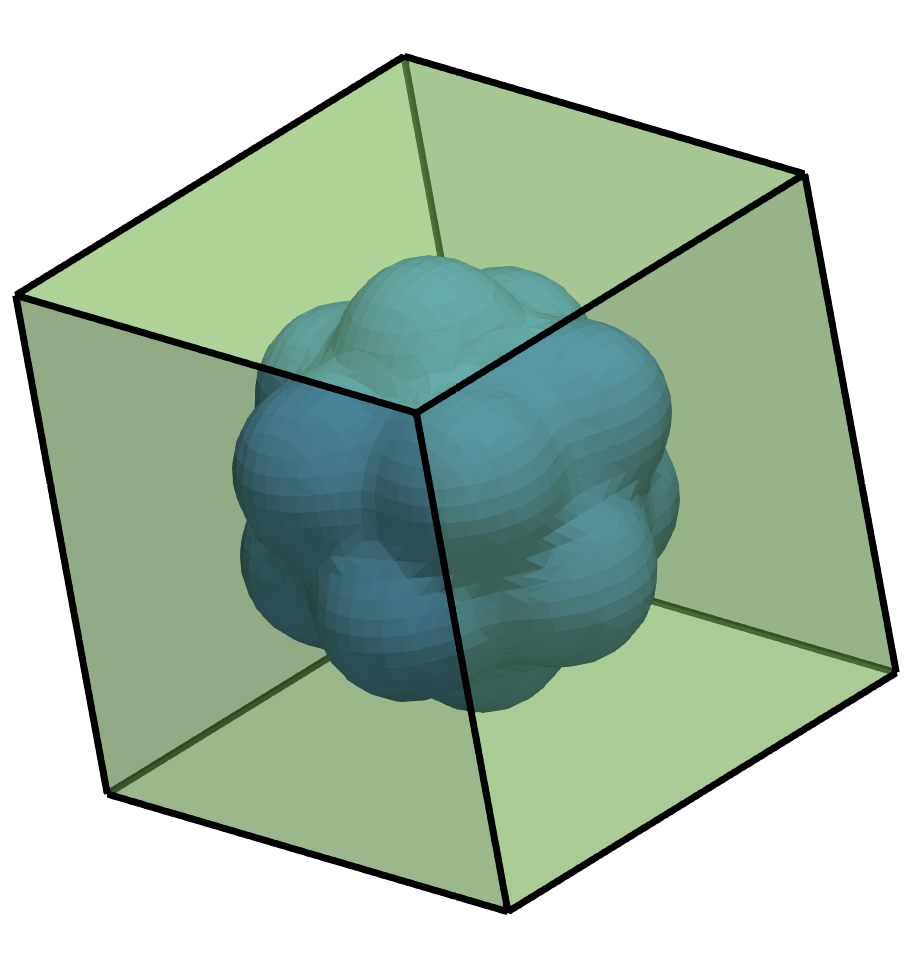}
    \caption{3D case (outer view).}
    \label{fig:conv-geometries-3d-1}
  \end{subfigure}
  \begin{subfigure}[b]{0.32\textwidth}
    \includegraphics[width=0.99\textwidth]{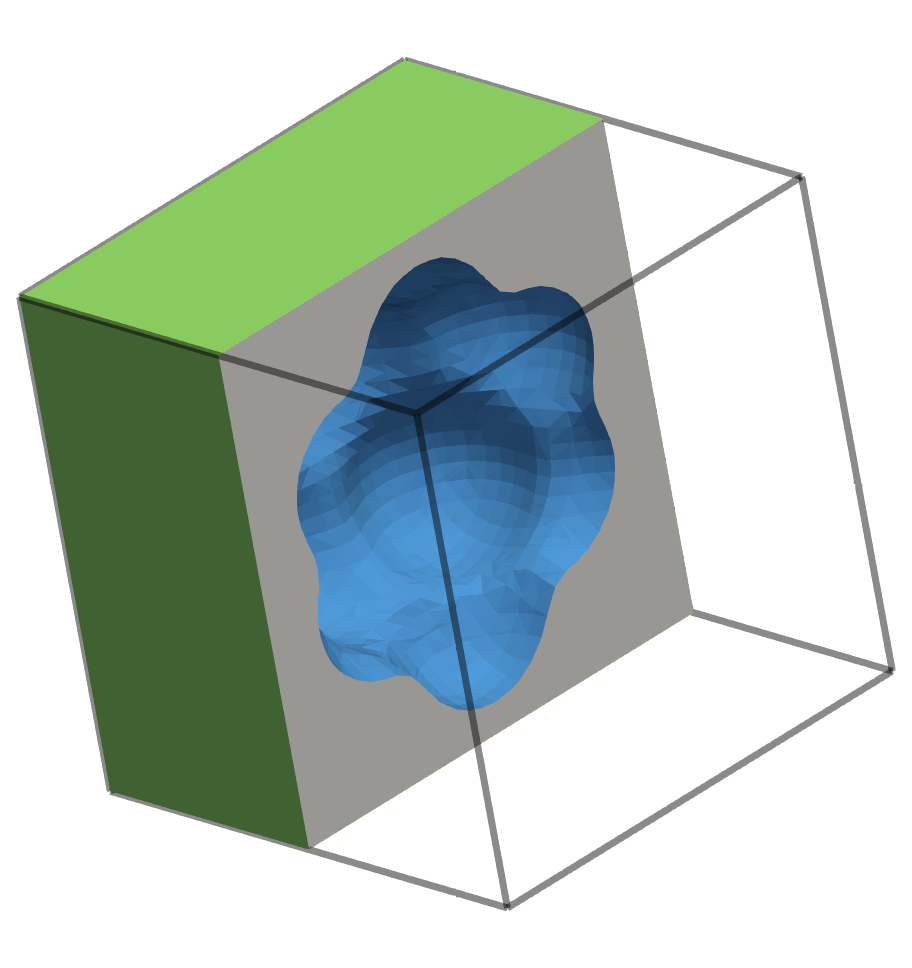}
    \caption{3D case (internal view).}
    \label{fig:conv-geometries-3d-2}
  \end{subfigure}

  \caption{\secttitle: View of the problem geometries.}
  \label{fig:conv-geometries}
\end{figure}

We use the method of manufactured solutions in order to have a problem with
known exact solution, which is used here to compute discretization errors. The
(manufactured) exact solution we have considered is
\begin{equation} \u\doteq\dfrac{\u^*}{|\u^*|},\hspace{4em} p\doteq x^{3}y^{3},
\label{eq:manuf-exact-sol} \end{equation} where \begin{equation}
  \begin{array}{l} \u^*=\left(-y+0.5,\ x +
    0.3\right)^t,\quad(x,y)\in\Omega\subset\mathbb{R}^2 \quad\text{in 2D,}\\
 \u^*=\left(y-0.5,\ -x - z -0.3,\ y-0.5
\right)^t,\quad(x,y,z)\in\Omega\subset\mathbb{R}^3\quad\text{in 3D.}
\end{array} \end{equation}
This solution corresponds to a (divergence-free) velocity field of magnitude
$1$ that spins around the point $(x,y)=(-0.3,0.5)$ for the 2D case and around
the line $(x,y,z)=(-z-0.3,0.5,z)$, $ z\in\mathbb{R}$, in 3D (see
Fig.~\ref{fig:conv-solution}).  The particular value of the boundary conditions
(both Dirichlet and Neumann) and external loads are defined such
that~\eqref{eq:manuf-exact-sol} is the exact solution of the Stokes problem
\eqref{eq:sto_eqs}. 

\begin{figure}[ht!]
  \centering
  \begin{subfigure}{0.32\textwidth}
    \centering
    \includegraphics[width=0.99\textwidth]{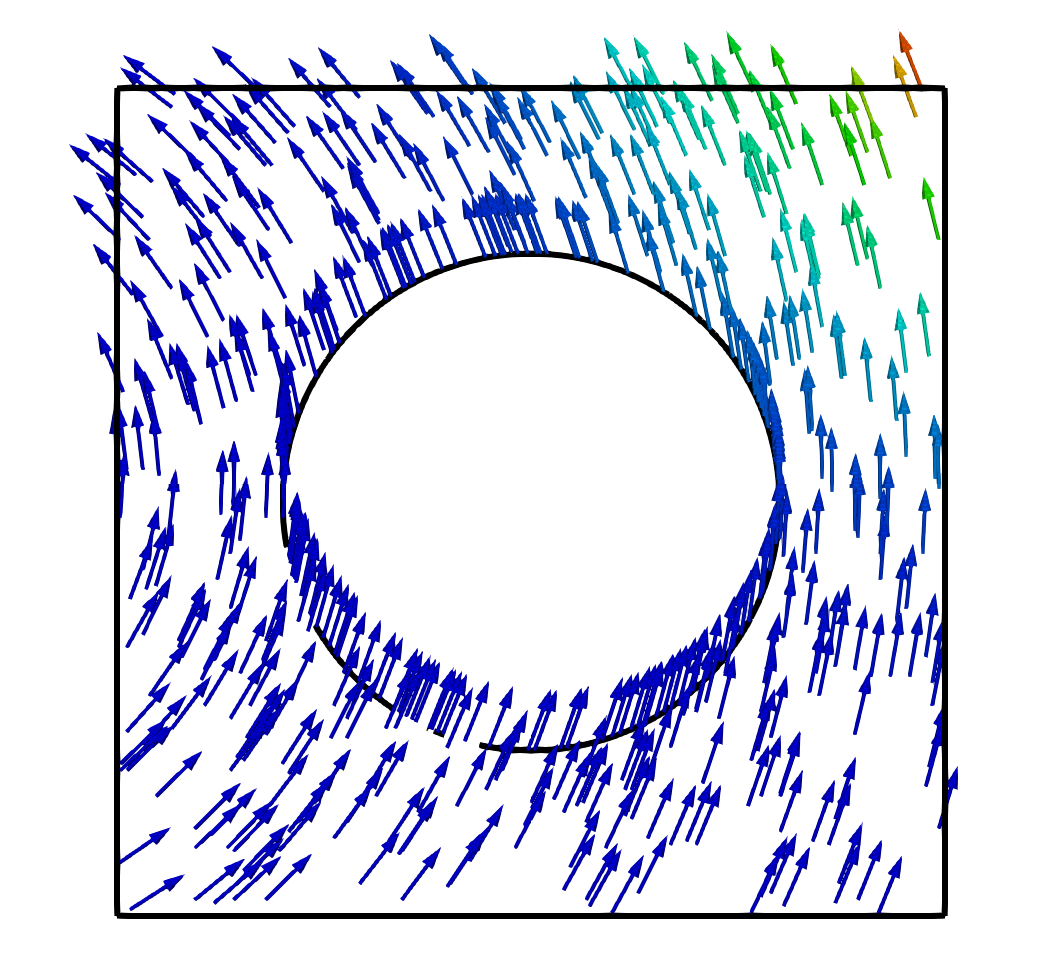}
    \caption{2D case.}
    \label{fig:conv-solution-2d}
  \end{subfigure}
  \begin{subfigure}{0.32\textwidth}
    \centering
    \includegraphics[width=0.99\textwidth]{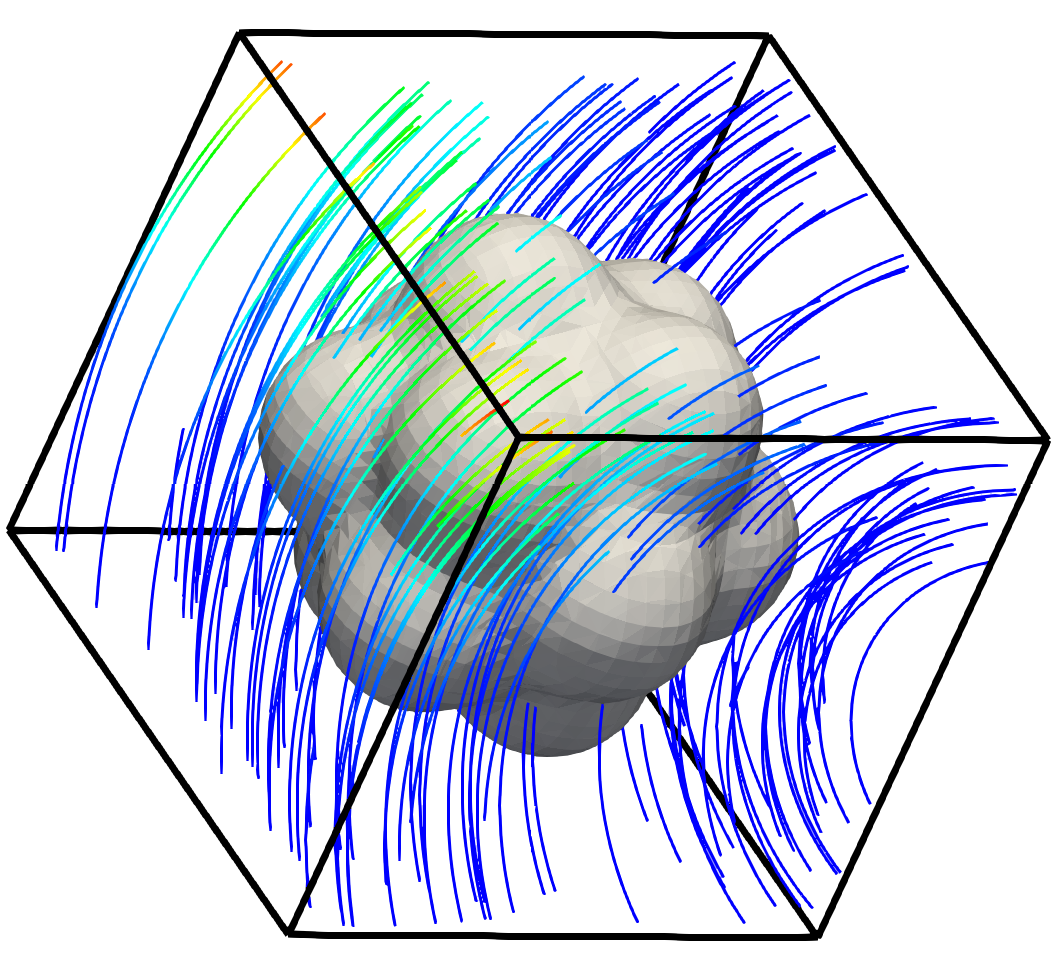}
    \caption{3D case.}
    \label{fig:conv-solution-3d}
  \end{subfigure}
  \begin{subfigure}{0.32\textwidth}
    \centering
    \begin{tikzpicture}
      \node[inner sep=0pt] (cbar) at (0,0) {\includegraphics[width=0.4\textwidth]{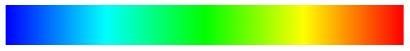}};
      \node[anchor=west] at (cbar.east) {\small $1$};
      \node[anchor=east] at (cbar.west) {\small $0$};
    \end{tikzpicture}
  \end{subfigure}
  \caption{\secttitle: View of the manufactured solution (vectors / streamlines colored by pressure field).}
  \label{fig:conv-solution}
\end{figure}

The numerical approximation is done using a family of uniform Cartesian meshes
obtained by dividing each direction of the unit square and cube into $2^m$
parts, with $m=3,4,\ldots,9$ in 2D, and $m=3,4,5$ in 3D. The obtained results
are displayed in Figs.~\ref{fig:conv-test-condest}, \ref{fig:conv-test-err-2d},
and~\ref{fig:conv-test-err-3d}.

\begin{figure}[ht!]
 \centering
 \includegraphics[scale=1]{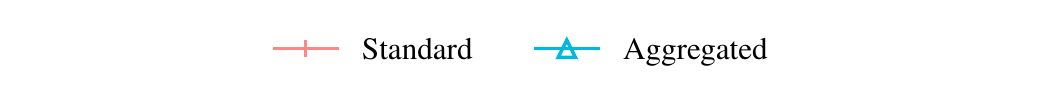}\\
  \begin{subfigure}{0.3\textwidth}
    \includegraphics[scale=1]{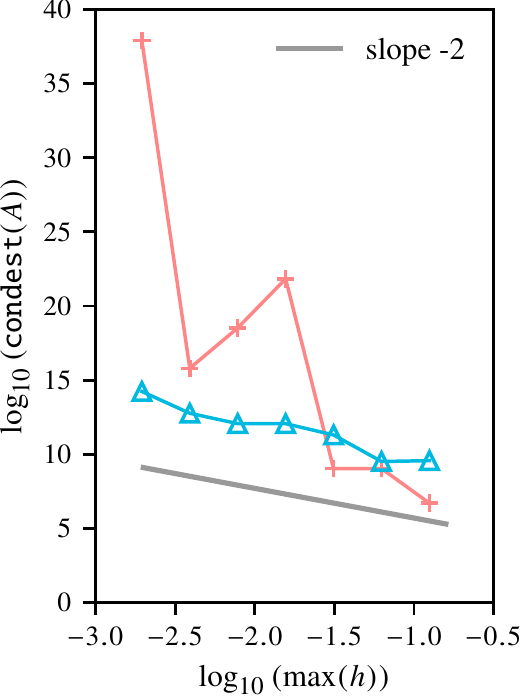}
    \caption{2D case.}
    \label{fig:conv-test-condest-a}
  \end{subfigure}
  \hspace{4em}
  \begin{subfigure}{0.3\textwidth}
    \includegraphics[scale=1]{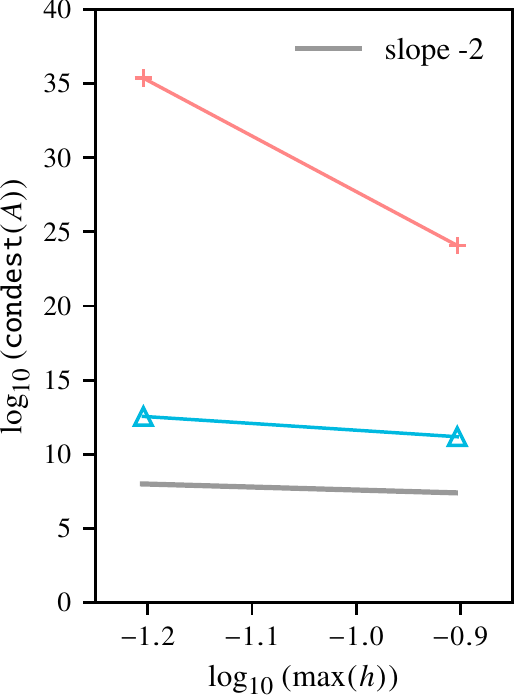}
    \caption{3D case.}
    \label{fig:conv-test-condest-b}
  \end{subfigure}
  \caption{\secttitle: Scaling of the condition number upon mesh refinement.}
  \label{fig:conv-test-condest}
\end{figure}

\begin{figure}[ht!]
 \centering
 \includegraphics[scale=1]{fig_popcorn_3d_legend}\\
  \begin{subfigure}{0.32\textwidth}
    \includegraphics[scale=1]{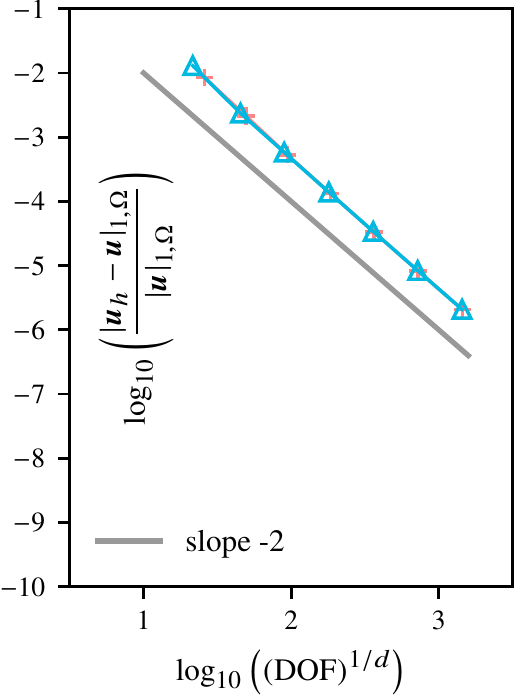}
    \caption{Velocity $H^1$ semi-norm.}
  \end{subfigure}
  \begin{subfigure}{0.32\textwidth}
    \includegraphics[scale=1]{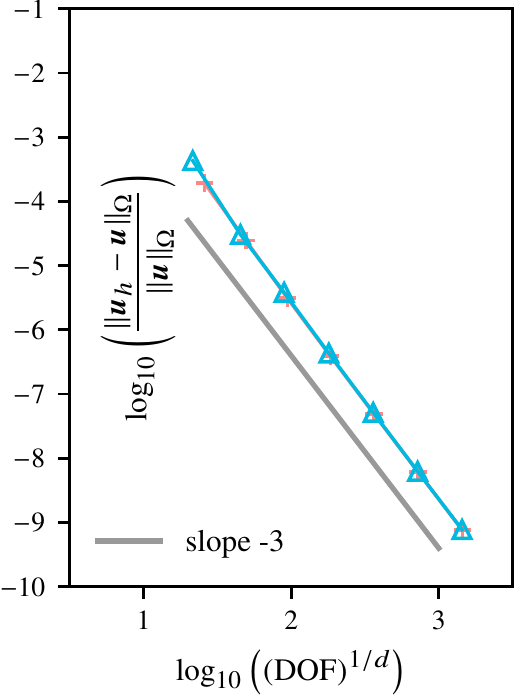}
    \caption{Velocity $L^2$ norm.}
  \end{subfigure}
  \begin{subfigure}{0.32\textwidth}
    \includegraphics[scale=1]{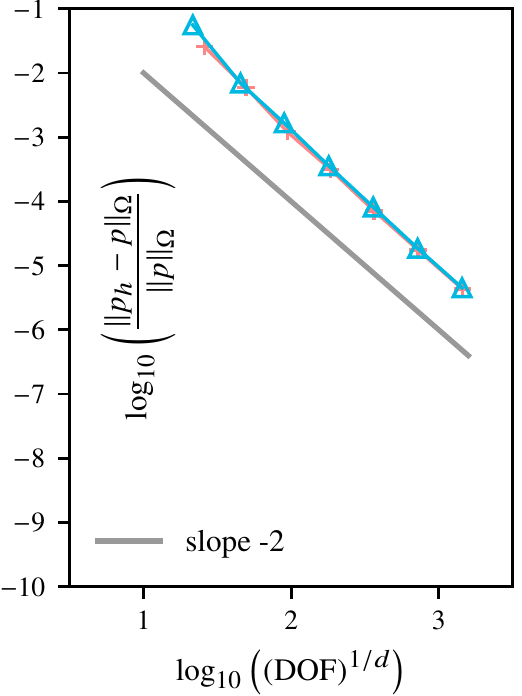}
    \caption{Pressure $L^2$ norm.}
  \end{subfigure}

  \caption{\secttitle: Convergence of the discretization error for the 2D case ($d=2$).}
  \label{fig:conv-test-err-2d}
\end{figure}

\begin{figure}[ht!]
 \centering
 \includegraphics[scale=1]{fig_popcorn_3d_legend}\\
  \begin{subfigure}{0.32\textwidth}
    \includegraphics[scale=1]{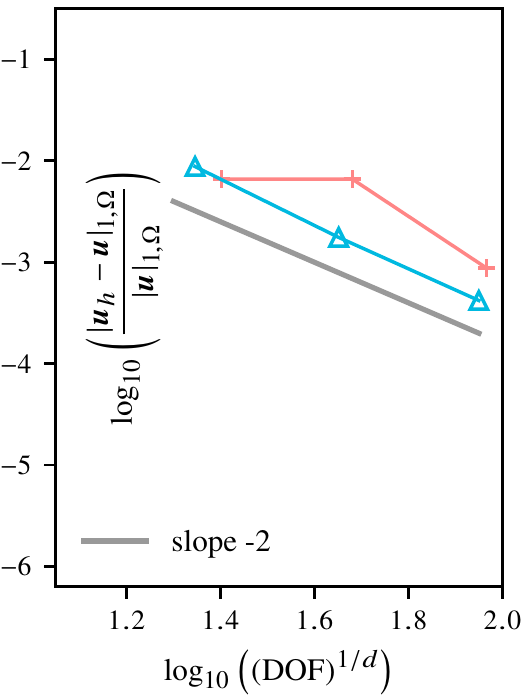}
    \caption{Velocity $H^1$ semi-norm.}
  \end{subfigure}
  \begin{subfigure}{0.32\textwidth}
    \includegraphics[scale=1]{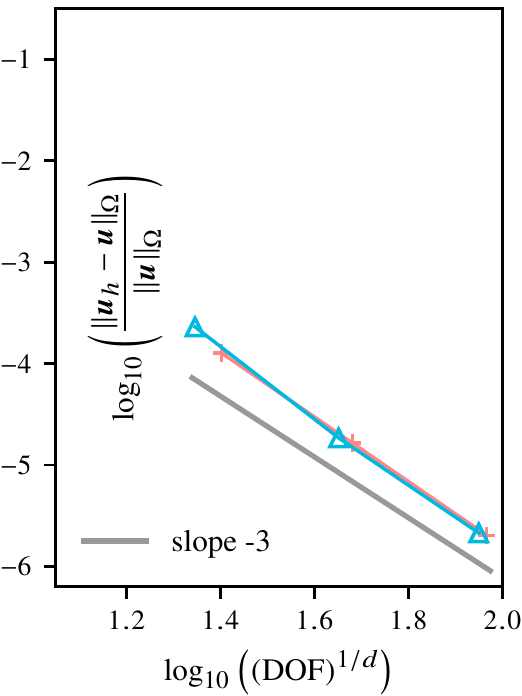}
    \caption{Velocity $L^2$ norm.}
  \end{subfigure}
  \begin{subfigure}{0.32\textwidth}
    \includegraphics[scale=1]{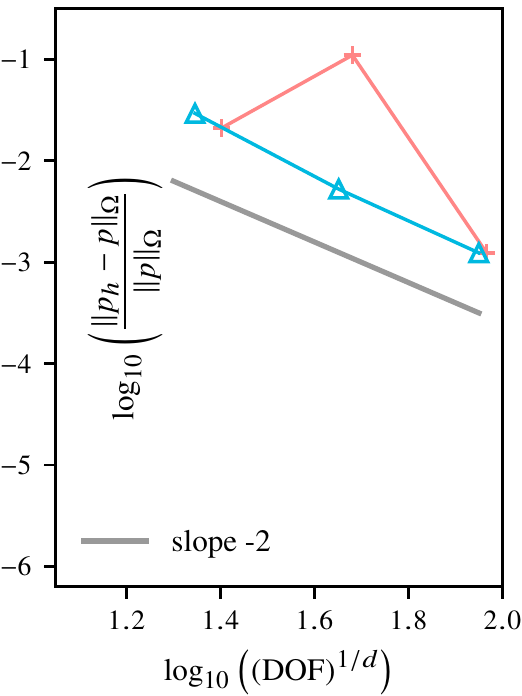}
    \caption{Pressure $L^2$ norm.}
      \label{fig:conv-test-err-3d-p}
  \end{subfigure}

  \caption{\secttitle: Convergence of the discretization error for the 3D case ($d=3$). }
  \label{fig:conv-test-err-3d}
\end{figure}

Fig.~\ref{fig:conv-test-condest}  shows the scaling of the condition number of
the underlying linear systems as the mesh is refined. For the 
\ac{agfe} spaces, the condition number scales as expected in conventional \ac{fe}
methods for body-fitted meshes (i.e., the condition number is proportional to
$h^{-2}$), which confirms the theoretical condition number bound derived in
Sect.  \ref{ssec:cond_nums}.  The same behavior is observed in 2D and 3D
cases. The lines for the 3D case in Fig. \ref{fig:conv-test-condest-b}
have only two points, since we were able to estimate the condition number only
for two of the 3D meshes due to the large amount of memory demanded by the {\tt
condest} function of MATLAB.  The benefit of using cell aggregation is clearly
illustrated in Fig.~\ref{fig:conv-test-condest}.  The standard \ac{fe} spaces
without cell aggregation lead to condition numbers that do not scale
proportional to $h^{-2}$. Theoretically, the condition number can be arbitrary
large  without cell aggregation depending on how cells are cut, which leads in
practice to an erratic scaling of the condition number that reaches large
values, as shown by the red lines in Fig.~\ref{fig:conv-test-condest}.

On the other hand, Figs.~\ref{fig:conv-test-err-2d}
and~\ref{fig:conv-test-err-3d} report the convergence of the $H^1$ semi-norm and $L^2$
norm of the discretization error for the velocity field, and the $L^2$ norm of
the discretization error for the pressure field for the 2D and 3D cases
respectively.  Since we consider standard
($\Qqth(\meshact)$,$\Pqthd(\meshact)$) and aggregated ($\Qqthm$,$\Pqthd$)
velocity-pressure
elements (which corresponds to $2$nd polynomial order for the velocities and
$1$st for the pressures), the optimal convergence orders are  $3$rd order of
convergence for the velocity error measured in the $L^2$ norm, $2$nd order for
the velocity error in the $H^1$ semi-norm, and  $2$nd order for the pressure error
in $L^2$ norm.  The plots show that the \ac{agfe} spaces lead to these
optimal \ac{fe} convergence orders, which in turn confirms the analysis of
Sect.~\ref{ssec:a_pri_err_est}.  Note that the standard (un-aggregated)
\ac{fe} spaces lead to the optimal convergence orders in the 2D case (cf. Fig.
\ref{fig:conv-test-err-2d}).  However, the underlying linear systems are so
ill-conditioned (reaching condition numbers up to $10^{35}$ as previously
showed in Fig.~\ref{fig:conv-test-condest}) that  in general one cannot rely on
the results computed by the linear solver using double precision floating point
arithmetics. { We have encountered some situations where the linear
solver was not able to provide an accurate solution for this reason, see, e.g.,
the red line in Fig. \ref{fig:conv-test-err-3d-p}.}

\renewcommand{\secttitle}{Moving domain experiment}
\subsection{\secttitle}
\label{sec:moving-dom}

In the second numerical experiment, we study the robustness of the unfitted
\ac{fe} formulation with respect to the relative position between the problem
geometry and the background mesh.  To this end, we consider two geometries
whose definition is parametrized by a scalar value $\ell$ (cf.
Fig.~\ref{fig:mov-geometries}). The 2D geometry is a circular cavity, with
radius $R=0.225$ and whose center is located at an arbitrary point on a
diagonal of the unit square (cf. Fig.~\ref{fig:mov-geometries-2d}).  The 3D
domain is again a cavity defined using the popcorn flake geometry (cf.
Fig.~\ref{fig:mov-geometries-3d}). In this case, We scale down the popcorn
flake used in the convergence test (cf. Sect.~\ref{sec:conv-test}) by a
factor or $0.5$ and place it at an arbitrary point of the diagonal of the unit
cube.   In both cases, the position of the bodies is controlled by the value of
the parameter $\ell$ (i.e., the distance between the center of the body and a
selected vertex of the square/cube). As the value of $\ell$ varies, the objects
move and their relative position with respect to the background  mesh changes.
In this process, arbitrary small cut cells can show up, leading to potential
ill conditioning problems. In this experiment, we consider a background mesh that
discretizes the unit square/cube with $2^m$ elements per direction, being $m=5$
for the 2D case and $m=4$ for the 3D case.

\begin{figure}[ht!]
  \centering

  \begin{subfigure}[b]{0.32\textwidth}
      \centering
      \includegraphics[width=0.85\textwidth]{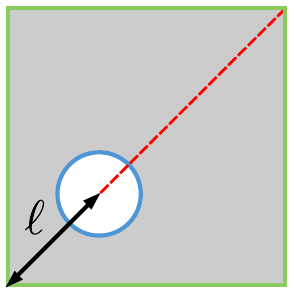}
        \caption{2D case.}
    \label{fig:mov-geometries-2d}
  \end{subfigure}
  \hspace{2em}
  \begin{subfigure}[b]{0.32\textwidth}
      \includegraphics[width=0.99\textwidth]{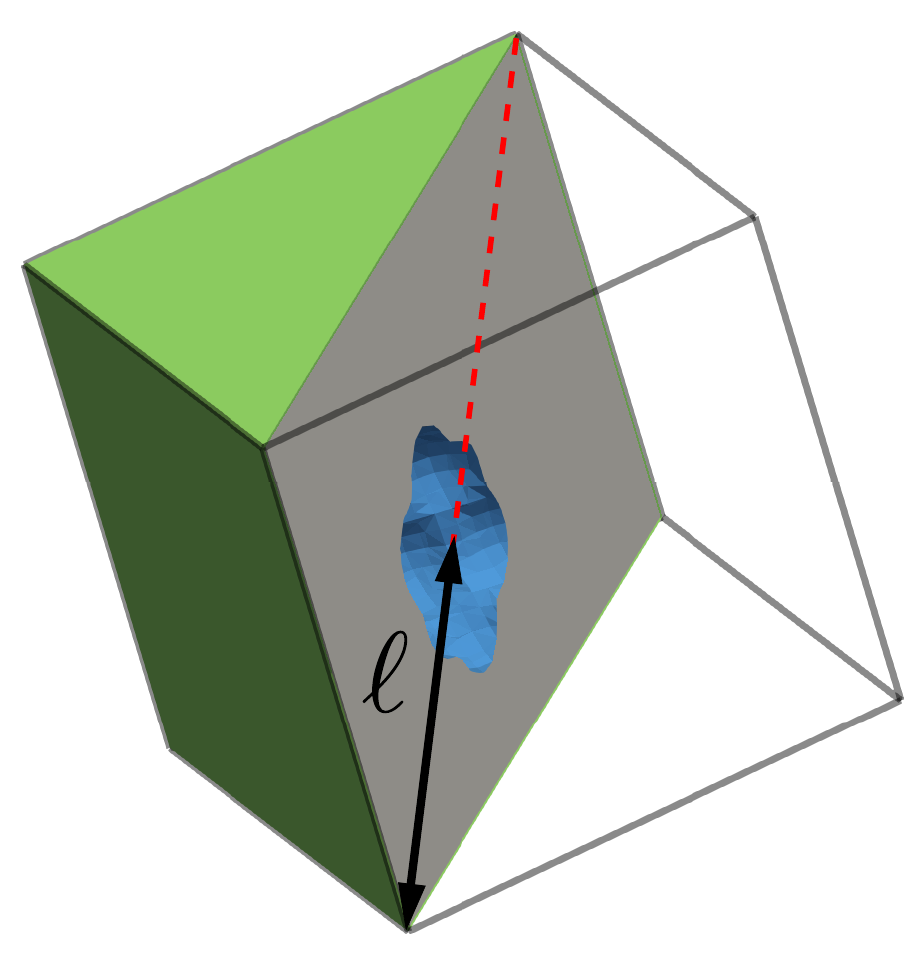}
    \caption{3D case (internal view).}
    \label{fig:mov-geometries-3d}
  \end{subfigure}
\caption{\secttitle: View of the problem geometries.}
\label{fig:mov-geometries}
\end{figure}

Fig.~\ref{fig:mov-condest} shows the condition number estimate of the
underlying linear systems versus $\ell$. The plot is generated using a sample
of 200 different values of $\ell$. It is observed  that the \ac{agfe}
spaces lead to condition numbers that are nearly independent of the value of
$\ell$, which shows that the \ac{agfem} is very robust
regardless how cells are cut. The benefit of using aggregation is clearly
demonstrated here by observing the results associated to the standard \ac{fe}
spaces. In that case, the condition numbers are very sensitive to the position
of the geometry and reach very high values (condition number greater than
$10^{35}$ in the 3D case).

\begin{figure}[ht!]
  \centering
   \includegraphics[scale=1]{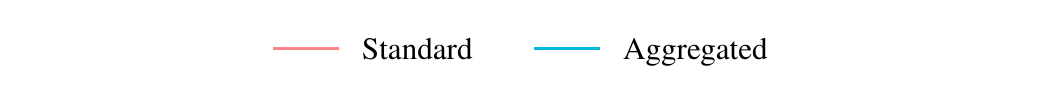}\\
  \begin{subfigure}{0.45\textwidth}
    \includegraphics[scale=1]{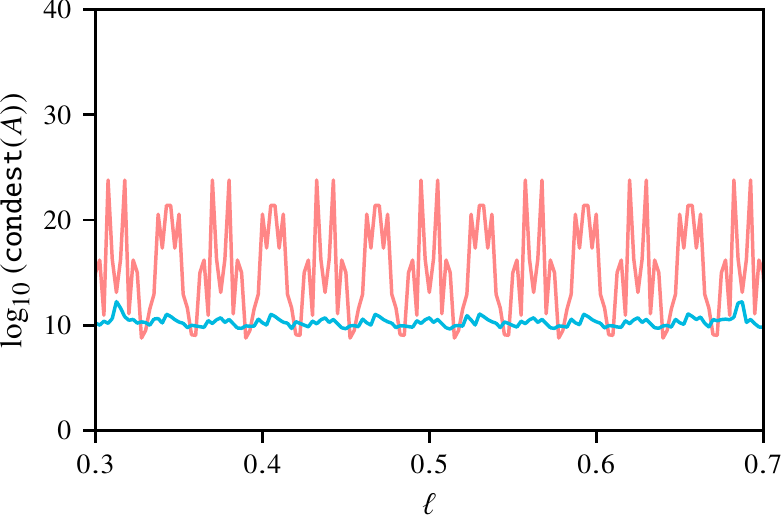}
        \caption{2D case.}
    \label{fig:mov-condest-2d}
  \end{subfigure}
  \hspace{2em}
  \begin{subfigure}{0.45\textwidth}
    \includegraphics[scale=1]{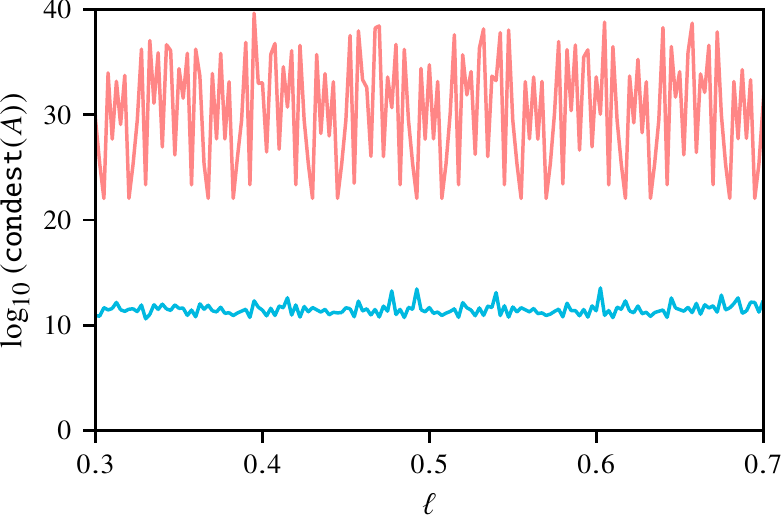}
    \caption{3D case.}
    \label{fig:mov-condest-3d}
  \end{subfigure}
\caption{\secttitle: Condition number vs. domain position.}
\label{fig:mov-condest}
\end{figure}

\renewcommand{\secttitle}{Complex 3D examples}
\subsection{\secttitle}
\label{sec:misc-3d}

We conclude the numerical examples with the simulation of two complex
geometries in order to show that the cell aggregation can be effectively used
also in more complex settings. The first complex example is the simulation of a
Stokes flow around a set of randomly spherical obstacles (see
Fig.~\ref{fig:complex-case1-sol}). The (fluid) domain is the set difference of
the unit cube $[0,1]^3$ and the spherical obstacles. We consider homogeneous
Dirichlet conditions (no-slip conditions) in the surfaces of the spherical
obstacles using Nitsche's method. The inflow boundary is the face $x=0$ of the
unit cube (see Fig.~\ref{fig:complex-case1-sol-a}), where we impose a
prescribed polynomial inflow velocity profile with value:
\begin{equation} \u = (  10y(y-1)z(z-1)  ,\ 0,\
0),\quad(x,y,z)\in\Gamma^\mathrm{in}=\{0\}\times[0,1]^2.
\end{equation}
The outflow boundary is the face $x=1$ of the unit cube, where we impose
homogeneous Neumann boundary conditions. We impose homogeneous Dirichlet
conditions on the remaining faces of the cube. The problem is simulated using a
background Cartesian mesh defined on the cube with $2^5$ elements per
direction. The obtained numerical solution is plotted in
Figs.~\ref{fig:complex-case1-sol-b} and~\ref{fig:complex-case1-sol-c}.  Note
that the approximation of the velocities clearly conforms to the
unfitted surfaces even though the interpolation is slightly coarsened near
these surfaces by the cell aggregation.

\begin{figure}[ht!]
  \centering
  \begin{subfigure}[c]{0.32\textwidth}
      \centering
      \tikz{\fill[myblue]  (0,0) rectangle (2.0em,0.5em);} {\small Inflow}
      \hspace{1em}
      \tikz{\fill[mygreen]  (0,0) rectangle (2.0em,0.5em);} {\small Outflow}
  \end{subfigure}
  \begin{subfigure}[c]{0.32\textwidth}
    \centering
    \begin{tikzpicture}
      \node[inner sep=0pt] (cbar) at (0,0) {\includegraphics[width=0.4\textwidth]{fig_colorbar_h.pdf}};
      \node[anchor=west] at (cbar.east) {\small $75.0$};
      \node[anchor=east] at (cbar.west) {\small $0.0$};
    \end{tikzpicture}
  \end{subfigure}
  \begin{subfigure}[c]{0.32\textwidth}
    \centering
    \begin{tikzpicture}
      \node[inner sep=0pt] (cbar) at (0,0) {\includegraphics[width=0.4\textwidth]{fig_colorbar_h.pdf}};
      \node[anchor=west] at (cbar.east) {\small $53.0$};
      \node[anchor=east] at (cbar.west) {\small $-1.4$};
    \end{tikzpicture}
  \end{subfigure}

  \begin{subfigure}[b]{0.32\textwidth}
      \centering
      \includegraphics[width=0.99\textwidth]{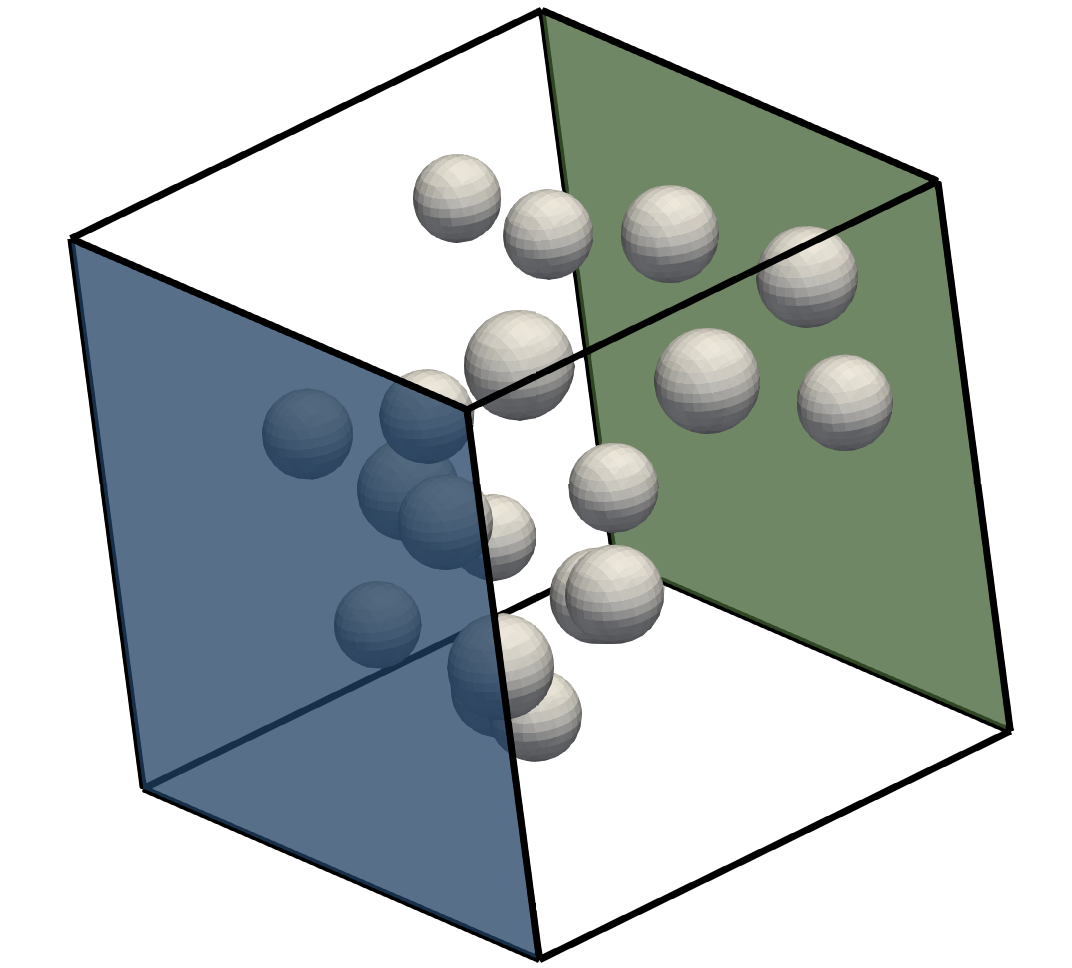}
      \caption{Problem geometry.}
    \label{fig:complex-case1-sol-a}
  \end{subfigure}
  \hspace{0em}
  \begin{subfigure}[b]{0.32\textwidth}
    \centering
    \includegraphics[width=0.99\textwidth]{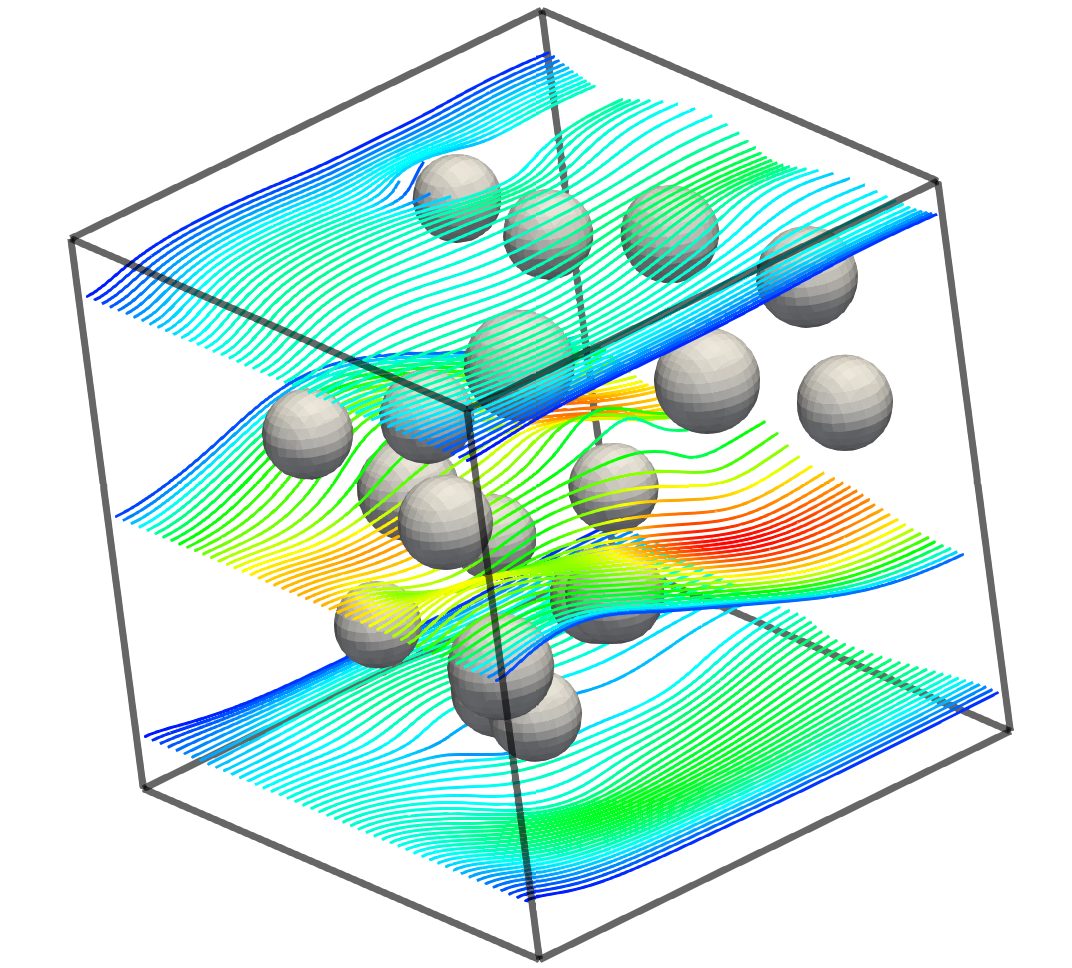}
    \caption{Velocity (magnitude).}
    \label{fig:complex-case1-sol-b}
  \end{subfigure}
  \begin{subfigure}[b]{0.32\textwidth}
    \includegraphics[width=0.99\textwidth]{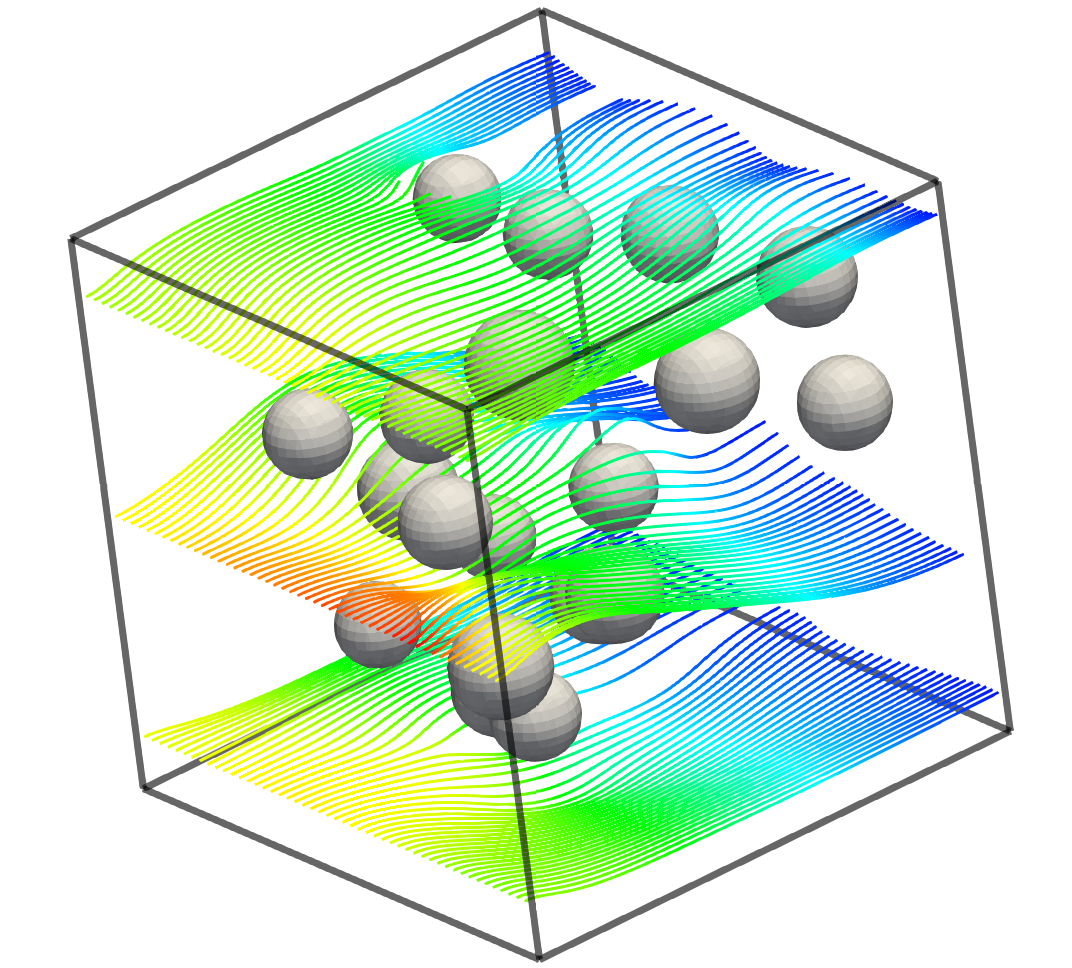}
    \caption{Pressure.}
    \label{fig:complex-case1-sol-c}
  \end{subfigure}
  \caption{\secttitle: Problem geometry and numerical solution  for the stokes
  flow around spherical obstacles (streamlines colored by velocity magnitude
  and pressure).}
  \label{fig:complex-case1-sol}
\end{figure}

The second complex example is a Stokes flow inside a spiral pipe (see
Fig.~\ref{fig:complex-case2-sol}). The radius of the tubular cross section of
the pipe is $0.1$, whereas the radius of the spiral central axis is $0.875$. We
impose homogeneous Dirichlet conditions on the walls of the spiral. The inflow
boundary is one of the two terminal cross sections of the pipe, i.e., the disk
of center $C=(0,0.875,0.86)$ and radius $R=0.1$ (see
Fig.~\ref{fig:complex-case2-sol-a}). On the inflow boundary we impose a
parabolic velocity profile with value:
\begin{equation} \u = ( 10 - 10 \frac{r^2}{R^2}  ,\ 0,\ 0), \end{equation}
where $r\in[0,R]$ is the distance between a point $x$ in the inflow boundary
and the center $C$.  Homogeneous Neumann boundary conditions are considered on
the outflow boundary. Like in the previous example, the problem is simulated
using a uniform Cartesian mesh of the unit cube with $2^5$ elements at each
direction. The
results are showed in Figs.~\ref{fig:complex-case2-sol-b}
and~\ref{fig:complex-case2-sol-c}. Note that, even though this is a very
challenging example for the cell-aggregation strategy because the surface to
volume ratio is very high, the computed results reproduce a perfectly laminar
velocity field that flows smoothly through the spiral pipe.

\begin{figure}[ht!]
  \centering
  \begin{subfigure}[c]{0.32\textwidth}
      \centering
      \tikz{\fill[myblue]  (0,0) rectangle (2.0em,0.5em);} {\small Inflow}
      \hspace{1em}
      \tikz{\fill[mygreen]  (0,0) rectangle (2.0em,0.5em);} {\small Outflow}
  \end{subfigure}
  \begin{subfigure}[c]{0.32\textwidth}
    \centering
    \begin{tikzpicture}
      \node[inner sep=0pt] (cbar) at (0,0) {\includegraphics[width=0.4\textwidth]{fig_colorbar_h.pdf}};
      \node[anchor=west] at (cbar.east) {\small $100$};
      \node[anchor=east] at (cbar.west) {\small $1.8$};
    \end{tikzpicture}
  \end{subfigure}
  \begin{subfigure}[c]{0.32\textwidth}
    \centering
    \begin{tikzpicture}
      \node[inner sep=0pt] (cbar) at (0,0) {\includegraphics[width=0.4\textwidth]{fig_colorbar_h.pdf}};
      \node[anchor=west] at (cbar.east) {\small $2.8\cdot10^{4}$};
      \node[anchor=east] at (cbar.west) {\small $-2.8\cdot10^{-5}$};
    \end{tikzpicture}
  \end{subfigure}

  \begin{subfigure}[b]{0.32\textwidth}
      \centering
      \includegraphics[width=0.99\textwidth]{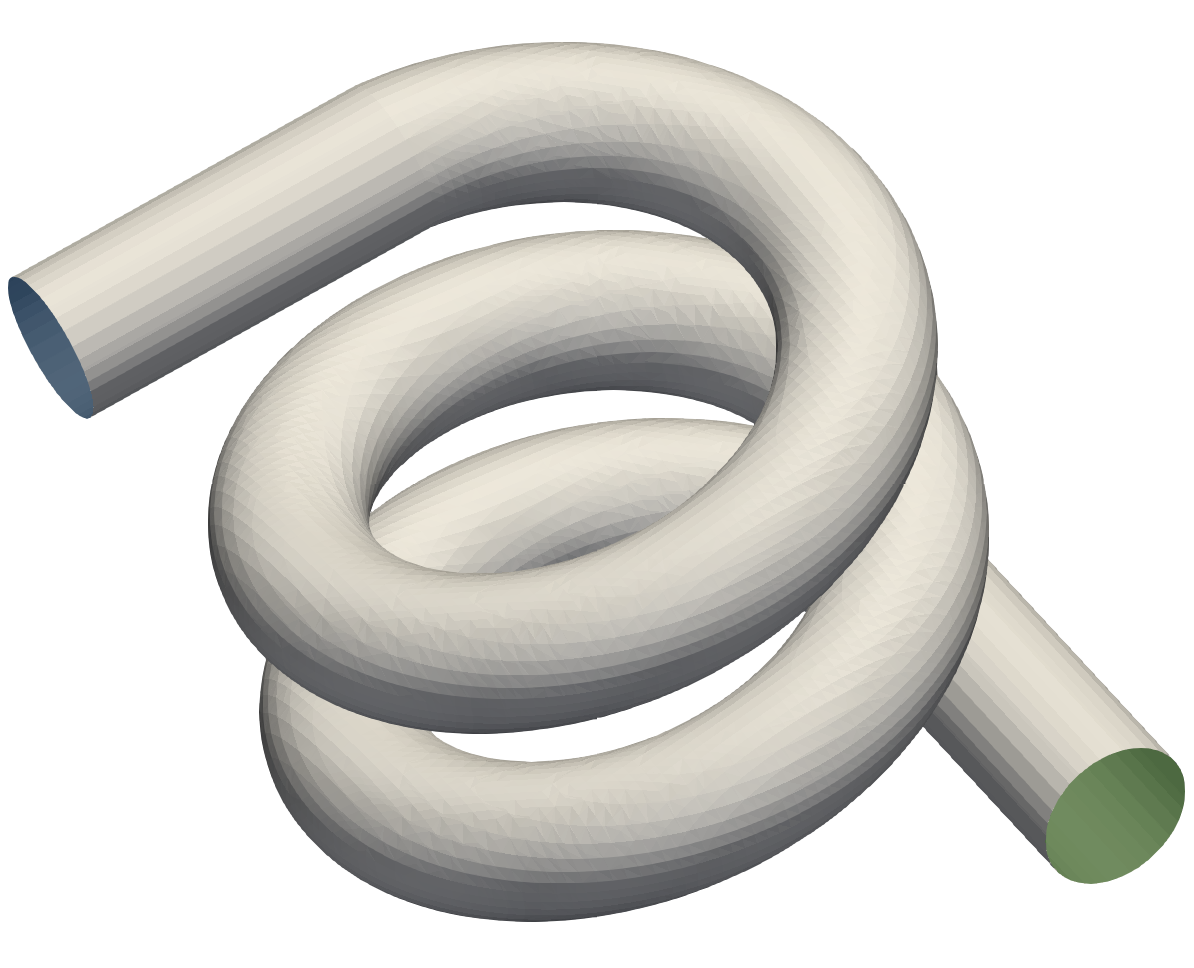}
      \caption{Problem geometry.}
    \label{fig:complex-case2-sol-a}
  \end{subfigure}
  \hspace{0em}
  \begin{subfigure}[b]{0.32\textwidth}
    \centering
    \includegraphics[width=0.99\textwidth]{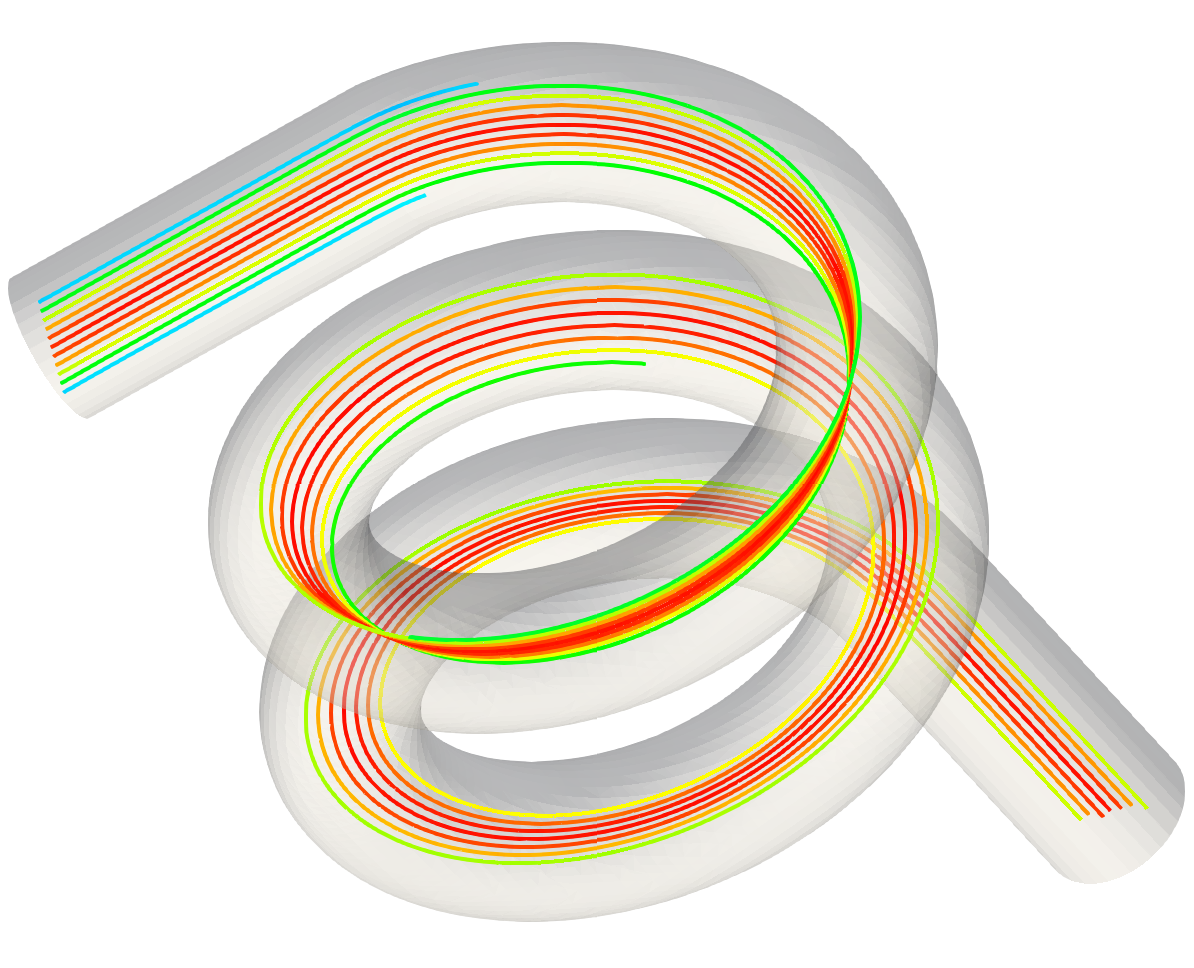}
    \caption{Velocity (magnitude).}
    \label{fig:complex-case2-sol-b}
  \end{subfigure}
  \begin{subfigure}[b]{0.32\textwidth}
    \includegraphics[width=0.99\textwidth]{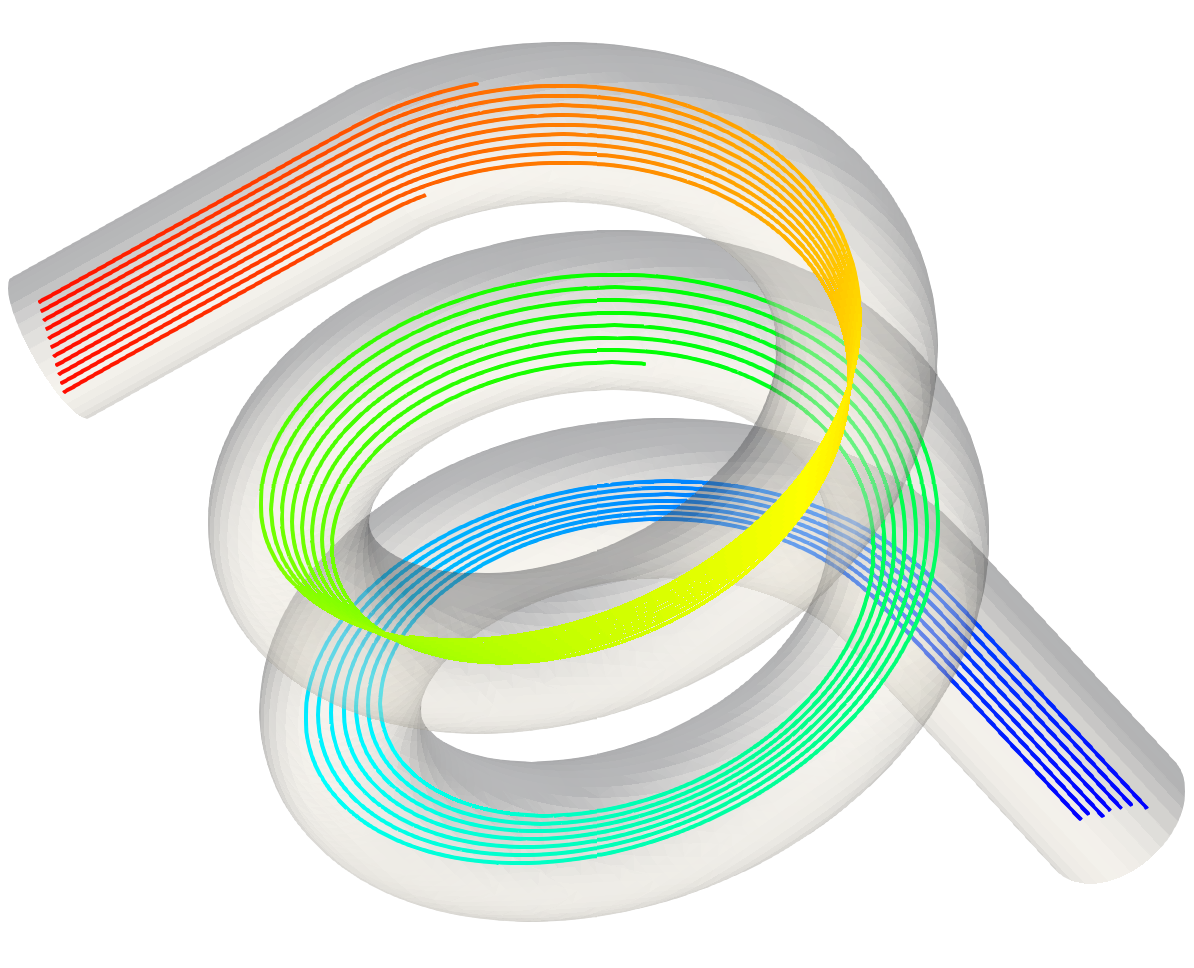}
    \caption{Pressure.}
    \label{fig:complex-case2-sol-c}
  \end{subfigure}
  \caption{\secttitle: Problem geometry and numerical solution  for the stokes
  flow in a spiral pipe (streamlines colored by velocity magnitude and
  pressure).}
  \label{fig:complex-case2-sol}
\end{figure}

\section{Conclusions}\label{sec:concl}

In this work, we have developed  mixed \acp{agfem} for the approximation of the Stokes problem on unfitted meshes. We have considered the standard extension operator for the definition of \ac{agfe} spaces and a new one that relies on the extension of the serendipity component only (for hex meshes). A cell aggregation algorithm allows one to start with a \ac{fe} mesh and create an aggregate partition with some desired properties. The \ac{agfe} space is readily computed from a typical \ac{fe} space plus simple cell-wise constraints.

For the sake of conciseness, we have considered as starting point mixed \ac{fe} methods on body-fitted meshes with discontinuous pressure spaces on hexahedral meshes, considering both the standard and serendipity extension for the velocity field. We have performed an abstract stability analysis that relies on a set of assumptions, in order to prove a weak inf-sup condition for mixed \ac{agfe} spaces. Such analysis shows the potential deficiency of the unfitted discrete inf-sup for such spaces. It allows us to identify a subset of aggregates/facets (\emph{close} to the boundary), coined \emph{improper} aggregates/facets; these subsets depend on the  mixed \ac{agfe} space being used.

Based on the abstract stability analysis, we have defined two different algorithms that satisfy the required assumption for having stability. The first algorithm relies on a standard velocity extension plus interior (residual-based) stabilization in improper aggregates and pressure jump stabilization on improper facets. The second algorithm relies on the serendipity extension for the velocity field components and pressure jump stabilization on improper facets. For these algorithms, a complete numerical analysis proves stability, \emph{a priori} error estimates, and condition number bounds that are not affected by the small cut cell problem.

A complete set of numerical experiments bears out the numerical analysis. Finally, the mixed \ac{agfem} is applied to two problems with non-trivial geometries, viz., free flow in a medium with inclusions and confined flow in a spiral.

\bibliographystyle{myabbrvnat}
\bibliography{art030}

\end{document}